\documentclass[a4paper, 10pt, reqno]{amsart}

\usepackage[utf8]{inputenc}
\usepackage[T1]{fontenc}
\usepackage{lmodern}
\usepackage{microtype}
\usepackage{amsmath, amssymb, amsthm, thmtools}
\usepackage{mathrsfs}								
\usepackage{esint}									
\usepackage{enumitem}								
\usepackage{listings}								
\usepackage[dvipsnames]{xcolor}
\usepackage[linktocpage]{hyperref}
\definecolor{colorred}{HTML}{B00000}
\definecolor{colorgreen}{HTML}{258300}
\hypersetup{colorlinks=true, linkcolor=colorred, citecolor=colorgreen, urlcolor=black}								
\usepackage{xifthen}								
\usepackage{url}
\usepackage{booktabs}								
\usepackage{todonotes}
\usepackage{caption}
\usepackage{bbm}									
\usepackage{stmaryrd}								
\numberwithin{equation}{section}

\let\oldtocsection=\tocsection

\let\oldtocsubsection=\tocsubsection

\let\oldtocsubsubsection=\tocsubsubsection

\renewcommand{\tocsection}[2]{\hspace{0em}\oldtocsection{#1}{#2}}
\renewcommand{\tocsubsection}[2]{\hspace{1em}\oldtocsubsection{#1}{#2}}
\renewcommand{\tocsubsubsection}[2]{\hspace{2em}\oldtocsubsubsection{#1}{#2}}

\newcommand{\nocontentsline}[3]{}
\newcommand{\tocless}[2]{\bgroup\let\addcontentsline=\nocontentsline#1{#2}\egroup}

\newcommand{\nlb}{{\ensuremath{\textnormal{b}}}}

\newcommand{\nld}{{\ensuremath{\textnormal{d}}}}

\newcommand{\nli}{{\ensuremath{\textnormal{i}}}}

\newcommand{\nlC}{{\ensuremath{\textnormal{C}}}}

\newcommand{\nlE}{{\ensuremath{\textnormal{E}}}}

\newcommand{\nlS}{{\ensuremath{\textnormal{S}}}}

\newcommand{\rmc}{{\ensuremath{\mathrm{c}}}}
\newcommand{\rmd}{{\ensuremath{\mathrm{d}}}}
\newcommand{\rme}{{\ensuremath{\mathrm{e}}}}

\newcommand{\rmi}{{\ensuremath{\mathrm{i}}}}

\newcommand{\rmE}{{\ensuremath{\mathrm{E}}}}

\newcommand{\sfa}{{\ensuremath{\mathsf{a}}}}
\newcommand{\sfb}{{\ensuremath{\mathsf{b}}}}

\newcommand{\sfe}{{\ensuremath{\mathsf{e}}}}

\newcommand{\sfg}{{\ensuremath{\mathsf{g}}}}
\newcommand{\sfh}{{\ensuremath{\mathsf{h}}}}

\newcommand{\sfp}{{\ensuremath{\mathsf{p}}}}

\newcommand{\sfr}{{\ensuremath{\mathsf{r}}}}

\newcommand{\sfA}{{\ensuremath{\mathsf{A}}}}
\newcommand{\sfB}{{\ensuremath{\mathsf{B}}}}

\newcommand{\sfG}{{\ensuremath{\mathsf{G}}}}
\newcommand{\sfH}{{\ensuremath{\mathsf{H}}}}

\newcommand{\sfL}{{\ensuremath{\mathsf{L}}}}

\newcommand{\sfP}{{\ensuremath{\mathsf{P}}}}

\newcommand{\sfS}{{\ensuremath{\mathsf{S}}}}
\newcommand{\sfT}{{\ensuremath{\mathsf{T}}}}

\newcommand{\scrD}{{\ensuremath{\mathscr{D}}}}

\newcommand{\scrL}{{\ensuremath{\mathscr{L}}}}
\newcommand{\scrM}{{\ensuremath{\mathscr{M}}}}
\newcommand{\scrN}{{\ensuremath{\mathscr{N}}}}

\newcommand{\calE}{{\ensuremath{\mathcal{E}}}}

\newcommand{\bdGamma}{{\ensuremath{\boldsymbol{\Gamma}}}}
\newcommand{\bdDelta}{{\ensuremath{\boldsymbol{\Delta}}}}

\newcommand{\N}{\boldsymbol{\mathrm{N}}}						
\newcommand{\R}{\boldsymbol{\mathrm{R}}}						
\newcommand{\C}{\boldsymbol{\mathrm{C}}}						
\renewcommand{\S}{\boldsymbol{\mathrm{S}}}						
\renewcommand{\d}{\,\mathrm{d}}				

\let\liminf\undefined
\let\div\undefined
\DeclareMathOperator*{\liminf}{liminf}		
\DeclareMathOperator*{\esssup}{esssup}		
\DeclareMathOperator{\div}{div}				
\DeclareMathOperator{\Span}{span}			
\DeclareMathOperator{\tr}{tr}				

\let\originalleft\left			
\let\originalright\right
\renewcommand{\left}{\mathopen{}\mathclose\bgroup\originalleft}
\renewcommand{\right}{\aftergroup\egroup\originalright}

\newcommand{\mapdef}[3][]{\ifthenelse{\isempty{#1}}{#2\quad\longmapsto\quad #3}{#1\colon\quad #2\quad\longmapsto\quad #3}}		
\newcommand{\der}[2][]{\ifthenelse{\isempty{#1}}{\frac{\nld}{\nld #2}}{\left.\frac{\nld}{\nld #2}\right\vert_{#1}}}				
\makeatletter
\newcommand{\checknarg}{\@ifnextchar\bgroup{\gobblenarg}{}}
\newcommand{\gobblenarg}[1]{\@ifnextchar\bgroup{,\ \! #1\gobblenarg}{,\ \! #1}}

\theoremstyle{definition}
\newtheorem{bump}{Bump}[section]

\theoremstyle{plain}
\newtheorem{theorem}[bump]{Theorem}
\newtheorem{proposition}[bump]{Proposition}
\newtheorem{definition}[bump]{Definition}
\newtheorem{lemma}[bump]{Lemma}
\newtheorem{corollary}[bump]{Corollary}
\newtheorem{introtheorem}{Theorem}

\theoremstyle{remark}
\newtheorem{remark}[bump]{Remark}
\newtheorem{example}[bump]{Example}

\newtheoremstyle{cited}
{\topsep}		
{\topsep}		
{\itshape}		
{}				
{\bfseries}		
{\textbf{.}}	
{.5em}			
{\thmname{#1} \thmnumber{#2} \thmnote{\normalfont#3}}		

\theoremstyle{cited}			

\makeatletter
\let\@fnsymbol\@arabic	 		
\makeatother

\makeatletter
\def\nonumberfootnote{\xdef\@thefnmark{}\@footnotetext}			
\makeatother

\newcommand{\mms}{M}				
\newcommand{\met}{\mathsf{d}}						
\newcommand{\meas}{\mathfrak{m}}			
\newcommand{\Leb}{\scrL}					
\newcommand{\vol}{\mathrm{vol}}				

\newcommand{\RCD}{\mathrm{RCD}}				
\newcommand{\CD}{\mathrm{CD}}				
\newcommand{\BE}{\mathrm{BE}}				

\newcommand{\bounded}{\nlb}					
\newcommand{\bs}{\mathrm{bs}}				
\newcommand{\loc}{\mathrm{loc}}				
\newcommand{\pr}{\mathrm{pr}}				
\newcommand{\HS}{{\mathrm{HS}}}				
\newcommand{\WHS}{{\mathrm{wHS}}}			
\newcommand{\Ric}{\mathrm{Ric}}				
\newcommand{\Cont}{\nlC}					
\newcommand{\Ell}{\mathit{L}}				
\newcommand{\Lip}{\mathrm{Lip}}				
\newcommand{\Sob}{\mathit{W}}				
\newcommand{\Sobo}{\nlS}					
\newcommand{\Test}{\mathrm{Test}}			

\newcommand{\Dom}{\mathrm{Dom}}				

\DeclareMathOperator{\Hess}{Hess}			
\DeclareMathOperator{\diam}{diam}			
\newcommand{\ChHeat}{\sfP}					
\newcommand{\HHeat}{\sfH}					
\newcommand{\Hheat}{\sfh}					
\newcommand{\Schr}[1]{\sfP^{#1}}			

\newcommand{\One}{\boldsymbol{1}}				
\newcommand{\Harm}{\mathrm{Harm}}			
\newcommand{\RIC}{\boldsymbol{\mathrm{Ric}}}

\newcommand{\Hodge}{\smash{\vec{\Delta}}}		
\newcommand{\LSI}{\mathrm{LSI}}				
\newcommand{\fLSI}{\mathrm{fLSI}}			
\newcommand{\pt}{\mathrm{p}}					
\newcommand{\ess}{\mathrm{e}}				
\newcommand{\DELTA}{\bdDelta}		
\newcommand{\Meas}{\mathfrak{M}}				
\newcommand{\cl}{\mathrm{cl}}				

\usepackage{blindtext}

\providecommand{\bysame}{\leavevmode\hbox to3em{\hrulefill}\thinspace}

\numberwithin{equation}{section}

\begin{document}
	
	\title[Heat flow on 1-forms under lower Ricci bounds]{Heat flow on 1-forms under lower Ricci bounds. Functional inequalities, spectral theory, 
		and heat kernel}
	
	\author{Mathias Braun}
	\address{Department of Mathematics, University of Toronto Bahen Centre, 40 St. George Street Room 6290, Toronto, Ontario M5S 2E4, Canada}
	\email{braun@math.toronto.edu}
	\thanks{The author warmly thanks Karl-Theodor Sturm and Batu Güneysu for many valuable discussions, their constant interest and helpful comments. Funding by the European Research Council through the ERC-AdG ``RicciBounds'', ERC project 10760021, is gratefully acknowledged.}
	
	
	
	\date{\today}
	
	\begin{abstract} We study the canonical heat flow $(\mathsf{H}_t)_{t\geq 0}$ on the cotangent module $L^2(T^*M)$ over an $\mathrm{RCD}(K,\infty)$ space $(M,\mathsf{d},\mathfrak{m})$, $K\in\boldsymbol{\mathrm{R}}$. We show Hess--Schrader--Uhlenbrock's inequality and, if $(M,\mathsf{d},\mathfrak{m})$ is also an $\mathrm{RCD}^*(K,N)$ space, $N\in(1,\infty)$, Bakry--Ledoux's inequality for $(\mathsf{H}_t)_{t\geq 0}$ w.r.t.~the heat flow $(\mathsf{P}_t)_{t\geq 0}$ on $L^2(M)$. 
	Variable versions of these estimates are discussed as well. In conjunction with a study of logarithmic Sobolev inequalities for $1$-forms, the previous inequalities yield various $L^p$-properties of $(\mathsf{H}_t)_{t\geq 0}$, $p\in [1,\infty]$.
		
		Then we establish explicit inclusions between the spectrum of its generator, the Hodge Laplacian $\smash{\vec{\Delta}}$, of the negative functional Laplacian $-\Delta$, and of the Schrödinger operator $-\Delta+K$. In the $\mathrm{RCD}^*(K,N)$ case, we prove compactness of $\smash{\vec{\Delta}^{-1}}$ if $M$ is compact, and the independence of the $L^p$-spectrum of $\smash{\vec{\Delta}}$ on $p \in [1,\infty]$ under a volume growth condition. 
		
		We terminate by giving an appropriate interpretation of a heat kernel for $(\mathsf{H}_t)_{t\geq 0}$. We show its existence in full generality without any local compactness or doubling, and derive fundamental estimates and properties of it.
	\end{abstract}
	
	\maketitle
		\thispagestyle{empty}
	
	\tableofcontents
	
	\section{Introduction}\label{Ch:Introduction}
	
	Let $(\mms,\met,\meas)$ be a metric measure space, i.e.~a complete and separable metric space $(\mms,\met)$ endowed with a locally finite, fully supported Borel measure $\meas$. We always assume that $(\mms,\met,\meas)$ is an $\RCD(K,\infty)$ space for some $K\in\R$. See  \autoref{Ch:Preliminaries} for a more detailed account on these.
	
	In the seminal papers \cite{gigli,gigli2018}, based on the notion of \emph{$\Ell^\infty$-modules} over $(\mms,\met,\meas)$, a first and, in particular, a \emph{second} order differential structure on such, possibly quite singular, spaces has been built. In high analogy with the setting of a complete Riemannian manifold with Ricci curvature bounded from below by $K$ endowed with its Riemannian volume measure, \cite{gigli2018} makes sense of differential geometric objects such as gradients, differentials, Hessians, vector fields, $1$-forms, etc. --- in fact, even in smooth situations, the axiomatization of \cite{gigli2018} covers certain Schrödinger-type operators on weighted Riemannian manifolds. In particular, building upon a notion of \emph{Hodge Laplacian} $\Hodge$, a nonsmooth cohomology theory and the \emph{heat flow} $(\HHeat_t)_{t\geq 0}$  with generator $-\Hodge$ acting on differential $1$-forms in the \emph{cotangent module} $\Ell^2(T^*\mms)$ have been introduced in \cite{gigli2018}. 
	
	On Riemannian manifolds --- weighted or not --- $(\HHeat_t)_{t\geq 0}$ has been introduced long before and was studied extensively over the last decades. Let us mention e.g.~\cite{charalambous2007,chavel1984, coulhon2007, devyver2014, lohoue2010, magniez2020} for geometric and analytic, and \cite{bismut1984, elworthy1994} for probabilistic studies on the heat flow on $1$-forms, its integral kernel, etc. See also \cite{elworthy2008} for the development of a Hodge theory on the Wiener space over a Riemannian manifold. On the other hand, apart from \cite{gigli2018}, little is known about $(\HHeat_t)_{t\geq 0}$ in the more general $\RCD(K,\infty)$ framework. This article aims in a thorough study of properties of $(\HHeat_t)_{t\geq 0}$ as well as its generator, which has not been the central objective of \cite{gigli2018}. The final outcome of our discussion is an appropriate definition and the construction of a \emph{heat kernel} for $(\HHeat_t)_{t\geq 0}$ in the nonsmooth setting.

	Our motivation for deeply studying the heat flow on $1$-forms over $\RCD$ spaces comes from different directions. First,  we provide a further contribution to the large diversity of  works generalizing important ``classical smooth''  statements to nonsmooth spaces. Second, we believe that $\RCD$ spaces (or more general spaces with ``lower Ricci bounds'' \cite{braun2019,sturm2019,erbar2020}) with the tensor language of \cite{gigli,gigli2018} are the correct framework to develop nonsmooth notions of stochastic differential geometry, e.g.~(damped) stochastic  parallel transports, a project which we attack in the future. Therein, in establishing Bismut--Elworthy--Li-type derivative formulas for the functional heat flow $(\ChHeat_t)_{t\geq 0}$ \cite{bismut1984, elworthy1994} --- which in turn are expected to provide further regularity information about it --- a good understanding of its $1$-form counterpart $(\HHeat_t)_{t\geq 0}$ is essential \cite{braun2020,driver2001}. Lastly, in smooth contexts, heat kernel methods for $1$-forms are useful in many important applications, all of which could lead to $\RCD$ analogues. Exemplary, let us quote
	\begin{itemize}
	\item a deeper understanding of the Hodge theorem (see  \cite{gigli2018} for the $\RCD$ result) by the study of the heat kernel \cite{milgram1951},
		\item a proof variant of index theorems in Riemannian geometry \cite{bismut1984, chavel1984, hsu2002},
		\item the study of boundedness of the Riesz transform, see e.g.~\cite{coulhon2008, coulhon2013, coulhon2020, devyver2014} and the references therein, and
		\item the study of its short-time asymptotics playing dominant roles in theoretical physics and quantum gravity \cite{avramidi2000, pleijel1949, rosenberg1997}. 
	\end{itemize}

	\subsubsection*{Hess--Schrader--Uhlenbrock's inequality} Let us now start by outlining the content of the paper. If $\mms$ is a Riemannian manifold with volume measure $\vol$, the heat flow $(\HHeat_t)_{t\geq 0}$ is well-known \cite{gallot1975, hess1980, rosenberg1988, strichartz1983} to be tightly linked  to the Ricci curvature $\Ric$ of  $\mms$ via $\Hodge$ through the \emph{vector Bochner formula}
	\begin{align*}
	\Delta \frac{\vert X\vert^2}{2} + \big\langle X, (\Hodge X^\flat)^\sharp\big\rangle = \vert\nabla X \vert^2 + \Ric(X,X),
	\end{align*}
	valid for every sufficiently regular vector field $X$. Since this identity also involves the Laplace--Beltrami operator $\Delta$ --- which is the generator of the heat semigroup $(\ChHeat_t)_{t\geq 0}$ acting on $\Ell^2(\mms)$ --- it implies important comparison estimates between $(\HHeat_t)_{t\geq 0}$ and $(\ChHeat_t)_{t\geq 0}$ as described now.
	
	The main result in \cite{gigli2018} in the $\RCD(K,\infty)$ case tells us that the vector Bochner formula can be made sense of in a weak form involving the measure-valued Laplacian $\DELTA$ and Ricci tensor $\RIC$ --- actually, the latter is \emph{defined} via such an identity. A classical interpolation argument from Bakry--Émery's theory then implies that for every $\omega\in\Ell^2(T^*\mms)$ and every $t\geq 0$, 
	\begin{align}\label{Eq:DIMDE}
	\vert \HHeat_t\omega\vert^2\leq \rme^{-2Kt}\,\ChHeat_t\big(\vert\omega\vert^2\big)\quad\meas\text{-a.e.}
	\end{align}
	
	This estimate, albeit being pointwise, is too weak to derive $\Ell^p$-properties of $(\HHeat_t)_{t\geq 0}$ for the important range $p\in [1,2)$. On Riemannian manifolds with Ricci lower bound $K$ and $\meas := \vol$, however, Bochner's formula is known to entail the stronger \emph{Hess--Schrader--Uhlenbrock inequality}
	\begin{align}\label{Eq:Kato Simon intro}
	\vert\HHeat_t\omega\vert\leq \rme^{-Kt}\,\ChHeat_t\vert\omega\vert\quad\meas\text{-a.e.}
	\end{align}
	for every $\omega\in\Ell^2(T^*\mms)$ and every $t\geq 0$. In this context, \eqref{Eq:Kato Simon intro} is due to \cite{hess1977,hess1980}. 
	
	We prove the bound \eqref{Eq:Kato Simon intro} in full generality on $\RCD(K,\infty)$ spaces $(\mms,\met,\meas)$.
	
	\begin{introtheorem}[see \autoref{Th:Kato-Simon}]\label{Th:Kato Simon intro} For every $\omega\in\Ell^2(T^*\mms)$ and every $t\geq 0$, \eqref{Eq:Kato Simon intro} is satisfied.
	\end{introtheorem}
	
	\autoref{Th:Kato Simon intro} opens the door for plenty of new insights into the behavior of $(\HHeat_t)_{t\geq 0}$ and its generator and is used at many places in this paper. For instance, $(\HHeat_t)_{t\geq 0}$ extends to a semigroup of bounded operators mapping $\Ell^p(T^*\mms)$ into $\Ell^p(T^*\mms)$ for every $p\in [1,\infty]$, strongly continuous if $p<\infty$, with operator norm of $\HHeat_t$ no larger than $\rme^{-Kt}$, see \autoref{Cor:Extension}. 
	

	\subsubsection*{Bakry--Ledoux's inequality} Before further commenting on the proof of \autoref{Th:Kato Simon intro}, we introduce improved versions of \eqref{Eq:DIMDE} we will obtain if $(\mms,\met,\meas)$  obeys the more restrictive $\RCD^*(K,N)$ condition, $N\in(1,\infty)$. To the best of our knowledge and apart from exact $1$-forms \cite{bakry2006, erbar2015, erbar2020}, these seem to be new even in the smooth case. At this point, let us concentrate on the  strongest --- we say that a  $1$-form $\omega$ obeys the \emph{strong Bakry--Ledoux inequality} if for every $t>0$,
		\begin{align*}
		\vert\HHeat_t\omega\vert^2 + \frac{2}{N}\int_0^t\rme^{-2Ks}\,\ChHeat_s\big(\vert\hspace*{-0.05cm}\div \HHeat_{t-s}\omega^\sharp\vert^2\big)\d s \leq \rme^{-2Kt}\,\ChHeat_{t}\big(\vert\omega\vert^2\big)\quad\meas\text{-a.e.}
		\end{align*}
		
	\begin{introtheorem}[see  \autoref{Th:Bakry-Ledoux}]\label{Th:BL intro} If $(\mms,\met,\meas)$ is an $\RCD^*(K,N)$ space, the strong Bakry--Ledoux inequality holds for every $\omega\in\Ell^2(T^*\mms)$.
	\end{introtheorem}
	
	The link between the $\RCD^*(K,N)$ condition and the strong version of Bakry--Ledoux is created by the same interpolation technique as in the dimension-free case \cite{gigli2018} in combination with the dimensional vector $2$-Bochner-inequality obtained by the study of $\RIC$ on $\RCD^*(K,N)$ spaces from \cite{han2018}. 
	
	The Bakry--Ledoux inequality --- already if valid for sufficiently many \emph{exact} $1$-forms --- characterizes the $\RCD^*(K,N)$ condition for $(\mms,\met,\meas)$ according to \cite{erbar2015, erbar2020} and the commutation relation between $(\HHeat_t)_{t\geq 0}$ and $(\ChHeat_t)_{t\geq 0}$ from \autoref{Le:Ht vs Pt} below.
	
	\subsubsection*{Vector 1-Bochner inequality} A crucial ingredient for \autoref{Th:Kato Simon intro} is that the vector $2$-Bochner inequality --- or, in clearer terms, vector $\Gamma_2$-inequality --- is \emph{self-improving}. As proven in \cite{gigli2018}, this means that from the a priori weaker inequality
	\begin{align}\label{Eq:2-Gamma-2}
	\DELTA\frac{\vert X\vert^2}{2} + \big\langle X, (\Hodge X^\flat)^\sharp\big\rangle\,\meas \geq K\,\vert X\vert^2 \,\meas
	\end{align}
	for a sufficiently large class of vector fields $X$, it already follows that
	\begin{align*}
	\DELTA\frac{\vert X\vert^2}{2} + \big\langle X, (\Hodge X^\flat)^\sharp\big\rangle\,\meas \geq \big[K\,\vert X\vert^2 + \vert\nabla X\vert^2\big]\,\meas,
	\end{align*}
	This property 
	has been crucial in defining $\RIC$ in \cite{gigli2018}. 
	
	Once having proven that $\vert X\vert\in\Dom(\DELTA)$ for sufficiently many $X$, the chain rule 
	\begin{align*}
	\DELTA \vert X\vert^2 = 2\,\vert X\vert\,\DELTA\vert X\vert + 2\,\vert\nabla \vert X\vert\vert^2\,\meas
	\end{align*}
	then implies the \emph{vector $1$-Bochner inequality}
	\begin{align*}
	\DELTA \vert X\vert + \vert X\vert^{-1}\,\big\langle X,(\Hodge X^\flat)^\sharp\big\rangle\,\meas \geq K\,\vert X\vert \,\meas,
	\end{align*}
	as a byproduct, where the key feature canceling out the covariant term is
	\begin{align}\label{Eq:Kato intro}
	\vert\nabla \vert X\vert\vert \leq \vert\nabla X\vert\quad\meas\text{-a.e.}
	\end{align}
	The latter, known as  \emph{Kato's inequality}, is well-known in the smooth framework, while its nonsmooth analogue, stated in \autoref{Pr:Kato's inequality}, has recently appeared in \cite{debin2020}. On the other hand, the vector $1$-Bochner inequality for vector fields not necessarily of gradient-type is new in the dimension-free setting.
	
	In an integrated version, the link between this vector $1$-Bochner inequality and $\Ell^1$-type bounds of the form \eqref{Eq:Kato Simon intro} is known as \emph{form domination}. See the classical references \cite{beurling1958, hess1977,simon1977} for motivations. For \autoref{Th:Kato Simon intro}, we adopt a similar strategy.
	
	It is worth noting --- and shortly addressed in \autoref{Re:Variable}, as the constant case will suffice for our purposes --- that the above-mentioned key estimates also hold true in greater generality if $(\mms,\met,\meas)$ has a stronger variable lower Ricci bound $\smash{k\in\Ell_\loc^1(\mms)}$ with $k\geq K$ on $\mms$ in the sense of \cite{braun2019,sturm2019}. For instance, the estimate \eqref{Eq:Kato Simon intro} appearing in \autoref{Th:Kato Simon intro} can be phrased in terms of Brownian motion $(B_t)_{t\geq 0}$ on $(\mms,\met,\meas)$ through
	\begin{align*}
	\vert \HHeat_t\omega\vert\leq \boldsymbol{\mathrm{E}}^\cdot\big[\rme^{-\int_0^{2t}k(B_r)/2 \d r}\,\vert\omega\vert(B_{2t})\big]\quad\meas\text{-a.e.}
	\end{align*}
	
	More generally, we expect these results to hold in the metric measure space settings of  \cite{braun2020,magniez2020,sturm2019} and even on tamed Dirichlet spaces \cite{erbar2020}  without any uniform lower Ricci bounds.
	
	\subsubsection*{Logarithmic Sobolev inequalities} \emph{Logarithmic Sobolev inequalities} for functions and their connections to the functional heat flow $(\ChHeat_t)_{t\geq 0}$, initiated in \cite{gross1975}, have been an active field of research in past decades. For an overview over the vast literature on this subject, see \cite{bakry2014, davies1989}. In a similar manner, in this article we relate logarithmic Sobolev inequalities for $1$-forms to certain further integral properties of $(\HHeat_t)_{t\geq 0}$ described below. There is some ambiguity in formulating the former depending on whether one regards $1$-forms as vector fields or really as contravariant objects. For brevity, we only outline \autoref{Def:LSI}, where we say that a sufficiently regular vector field $X$ obeys the $2$-logarithmic Sobolev inequality $\LSI_2(\beta,\chi)$ with parameters $\beta>  0$ and $\chi\in \R$ if
	\begin{align*}
	\int_\mms \vert X\vert^2\log\vert X\vert\d\meas \leq \beta\,\Vert\nabla X\Vert_{\Ell^2}^2 + \chi\,\Vert X\Vert_{\Ell^2}^2 + \Vert X\Vert_{\Ell^2}\log\Vert X\Vert_{\Ell^2}.
	\end{align*}
	The advantage of this form is that it follows from logarithmic Sobolev inequalities for functions, known to hold in various cases \cite{cavalletti2017, villani2009}, via \eqref{Eq:Kato intro}, see \autoref{Th:Log Sobolev for RCD}. It also implies its contravariant pendant from \autoref{Def:p log Sob} for arbitrary exponents, see \autoref{Le:2LSI to pLSI}.
	
	The integral properties of $(\HHeat_t)_{t\geq 0}$ to be derived are the following. We call $(\HHeat_t)_{t\geq 0}$ 
	\begin{itemize}
		\item \emph{hypercontractive} if there exist $T\in (0,\infty]$ and a strictly increasing $\Cont^1$-function $p\colon [0,T)\to (1,\infty)$ such that $\HHeat_t$ is bounded from $\Ell^{p(0)}(T^*\mms)$ to $\Ell^{p(t)}(T^*\mms)$ for every $t\in (0,T)$, and
		\item \emph{ultracontractive} if there exist $p_0\in (1,\infty)$ and $T>0$ such that $\HHeat_T$ is bounded from $\Ell^{p_0}(T^*\mms)$ to $\Ell^\infty(T^*\mms)$.
	\end{itemize}
	
	In great generality, in \autoref{Th:pLSI to Ultra} we study when certain logarithmic Sobolev inequalities imply hyper- or ultracontractivity of $(\HHeat_t)_{t\geq 0}$. We also treat a partial converse in \autoref{Th:LSI 2 from hyper}.
	
	Read in concrete applications, according to all these discussions and the known functional examples from \cite{cavalletti2017, villani2009}, we deduce the following hypercontractivity.
	
	\begin{introtheorem}\label{Th:LSI intro} On any compact $\RCD^*(K,N)$ space or, for $K>0$, any $\RCD(K,\infty)$ space, there exists a constant $\beta > 0$ such that for every $p_0 \in (1,\infty)$, $\HHeat_t$ is bounded from $\Ell^{p_0}(T^*\mms)$ to $\Ell^{p(t)}(T^*\mms)$ with operator norm no larger than $\rme^{-Kt}$ for every $t>0$, where the function $p\colon [0,\infty)\to (1,\infty)$ is given by $p(t) := 1+(p_0-1)\,\rme^{2t/\beta}$.
	\end{introtheorem}
	
	Many of our arguments and results entailing \autoref{Th:LSI intro} are inspired by the functional treatise \cite{davies1989}. In the case of non-weighted Riemannian manifolds, logarithmic Sobolev inequalities for $1$-forms have been studied with similar results in \cite{charalambous2007}.
	
	\subsubsection*{Spectral behavior of Hodge's Laplacian} Kato's inequality \eqref{Eq:Kato intro} also explicitly connects the spectra of the Hodge and the functional Laplacian. The study of the former is our goal in \autoref{Ch:Spectral props}.
	
	The following is first shown in full generality.
	
	\begin{introtheorem}[see \autoref{Th:Spectrum of fcts and of forms} and  \autoref{Cor:Spectral gap inclusion}]\label{Th:Spectral bounds Intro} If a positive real number belongs to the spectrum of $-\Delta$, then it is also contained in the spectrum of $\Hodge$. Similar inclusions hold between the respective point and essential spectra. In particular, 
		\begin{align*}
		\inf\sigma(-\Delta+K)\leq \inf\sigma(\Hodge)\leq \inf\sigma(\Hodge)\setminus\{0\}\leq \inf\sigma(-\Delta)\setminus \{0\}.
		\end{align*}
	\end{introtheorem}
	
	The stated spectral inclusions are known in the non-weighted Riemannian setting by \cite{charalambous2019}. Our proof of the former adopts a similar strategy, relying on a suitable variant of Weyl's criterion. The first stated spectral gap inequality follows by basic spectral theory and is well-known in the smooth setting. See e.g.~\cite{gueneysu2017} for a more general smooth treatise and further references.
	
	On compact $\RCD^*(K,N)$ spaces, similarly to the case of functions, the spectrum of $\Hodge$ can be characterized much better. A key tool towards an explicit understanding of it in this case is the following Rellich-type compact embedding theorem.
	
	\begin{introtheorem}[see  \autoref{Th:Compact embedding}]\label{Th:Compactness Delta^-1 intro} If $(\mms,\met,\meas)$ is a compact $\RCD^*(K,N)$ space, the formal operator $\smash{\Hodge^{-1}}$ is compact.
	\end{introtheorem}
	
	For Ricci limit spaces, i.e.~noncollapsed mGH-limits of sequences of non-weighted Riemannian manifolds with uniformly lower bounded Ricci curvatures, \autoref{Th:Compactness Delta^-1 intro} is due to \cite{honda2017, honda2018a}. In the very recent work \cite{honda2020}, \autoref{Th:Compactness Delta^-1 intro} has been proven independently in a different way using so-called $\delta$-splitting maps.
	
	The proof of \autoref{Th:Compactness Delta^-1 intro} uses several powerful properties of $(\HHeat_t)_{t\geq 0}$ on compact $\RCD^*(K,N)$ spaces. Using that $(\ChHeat_t)_{t\geq 0}$ admits a heat kernel which obeys Gaussian bounds \cite{jiang2016, tamanini2019}, together with \autoref{Th:Kato Simon intro} and Bishop--Gromov's inequality, we see in \autoref{Th:Lp Linfty} that the heat operator $\HHeat_t$ maps $\Ell^p(T^*\mms)$ boundedly into $\Ell^\infty(T^*\mms)$ for every $t>0$ and every $p\in [1,\infty]$. In particular, $\HHeat_t$ is a Hilbert--Schmidt operator on $\Ell^2(T^*\mms)$, and \autoref{Th:Compactness Delta^-1 intro} as well as expected properties of the spectrum of $\Hodge$ stated in \autoref{Pr:Eigenbasis} are then deduced by abstract functional analysis. We also establish $\Ell^\infty$-estimates on eigenforms of $\Hodge$, with an explicit growth rate for positive eigenvalues. See \autoref{Pr:Harmonic forms bounded} and \autoref{Th:L infty estimates}.
	
	The last part of \autoref{Ch:Spectral props}, especially \autoref{Th:Lp independence theorem}, is devoted to the proof of the independence of the $\Ell^p$-spectrum of $\Hodge$ on $p\in [1,\infty]$, provided $(\mms,\met,\meas)$ is an $\RCD^*(K,N)$ space satisfying, for every $\varepsilon > 0$, the volume growth condition
	\begin{align*}
	\sup_{x\in\mms} \int_\mms \rme^{-\varepsilon\,\met(x,y)}\,\meas[B_1(x)]^{-1/2}\,\meas[B_1(y)]^{-1/2}\d\meas(y) < \infty.
	\end{align*}
	On non-weighted Riemannian manifolds, this is shown in \cite{charalambous2005}. Our proof, based on a perturbation argument, \autoref{Th:Kato Simon intro} and functional heat kernel bounds, is inspired by similar results for the functional Laplacian \cite{hempel1986, hempel1987, saloff-coste1992, sturm1993}. See also \cite{chen2012, deleva2010, kusuoka2014, takeda2007, takeda2009} for further works in this direction for Markov processes and Feynman--Kac semigroups.
	
	\subsubsection*{Heat kernel} Up to now no general result ensuring the \emph{existence} of a heat kernel for $(\HHeat_t)_{t\geq 0}$ was known in the setting of \cite{gigli2018}. Outside the scope of noncompact, even weighted Riemannian manifolds \cite{gueneysu2017, patodi1971, rosenberg1997}, there are only few metric measure constructions under restrictive structural (existence of a \emph{continuous} covector bundle with constant fiber dimensions) and volume doubling assumptions \cite{coulhon2008, sikora2004}. 	Our axiomatization and existence proof of a heat kernel for $(\HHeat_t)_{t\geq 0}$ on $\RCD(K,\infty)$ spaces is hoped to push forward research in the above areas on such spaces. Our general study applies to non-locally compact or non-doubling, possibly infinite-dimensional spaces.
	
	Let us motivate our axiomatization via the heat kernel $\sfp$ of $(\ChHeat_t)_{t\geq 0}$ from \cite{ambrosio2014b}. Slightly abusing notation, it induces a map $\sfp\colon (0,\infty)\times \Ell^0(\mms)^2 \to \Ell^0(\mms^2)$ sending $t>0$ and $(g,f)\in\Ell^0(\mms)^2$ to the $\meas^{\otimes 2}$-measurable function given by $\sfp_t[g,f](x,y) := \sfp_t(x,y)\,g(x)\,f(y)$ in such a way that for  a sufficiently large class of functions $f$ and $g$, we have $\sfp_t[g,f]\in \Ell^1(\mms^2)$ as well as
	\begin{align*}
	g\,\ChHeat_tf = \int_\mms \sfp_t[g,f](\cdot, y)\d\meas(y)\quad\meas\text{-a.e.}
	\end{align*}
	
	Let us turn to $1$-forms. Recall that a heat kernel for $(\HHeat_t)_{t\geq 0}$ in the smooth, possibly weighted setting is a jointly smooth map $\Hheat\colon (0,\infty)\times\mms^2 \to (T^*\mms)^* \boxtimes T^*\mms$ --- i.e.~for every $t>0$ and every $(x,y)\in \mms^2$, $\Hheat_t(x,y)$ is a homomorphism mapping $T_y^*\mms$ to $T_x^*\mms$ --- satisfying
	\begin{align}\label{Eq:Heat kernel identity intro}
	\HHeat_t\omega = \int_\mms \Hheat_t(\cdot,y)\,\omega(y)\d\meas(y)\quad\meas\text{-a.e.}
	\end{align}
	for every $\omega\in\Ell^2(T^*\mms)$. The heat kernel for $1$-forms has first been constructed on compact spaces by \cite{patodi1971} using the so-called \emph{parametrix construction}. See also \cite{gueneysu2017, rosenberg1997}.
	
	Since $\RCD(K,\infty)$ spaces a priori do neither come with any covector bundle nor with a smooth structure, the fiberwise notion \eqref{Eq:Heat kernel identity intro} is replaced by ``testing \eqref{Eq:Heat kernel identity intro} pointwise against sufficiently many $1$-forms''. Motivated by our functional  considerations, we understand a mapping $\Hheat\colon (0,\infty)\times \Ell^0(T^*\mms)^2 \to \Ell^0(\mms^2)$ to be a \emph{heat kernel} for $(\HHeat_t)_{t\geq 0}$ if, for every $t>0$, $\Hheat_t$ is $\Ell^0$-bilinear, and for all sufficiently regular $1$-forms $\omega$ and $\eta$, we have $\Hheat_t[\eta,\omega]\in\Ell^1(\mms^2)$ and
	\begin{align*}
	\langle\eta,\HHeat_t\omega \rangle = \int_\mms \Hheat_t[\eta,\omega](\cdot, y)\d\meas(y)\quad\meas\text{-a.e.}
	\end{align*}
	
	\begin{introtheorem}[see \autoref{Th:Heat kernel existence}]\label{Th:Heat kernel intro} The heat kernel for $(\HHeat_t)_{t\geq 0}$ in the indicated sense exists and is unique.
	\end{introtheorem}
	
	The proof strategy is the following. Motivated by similar functional results \cite{sturm1995, saloff-coste2010}, a crucial tool to obtain integral kernels for certain operators is a \emph{Dunford--Pettis}-type theorem \cite{dunford1940, dunford1958}, a very general $\Ell^\infty$-module version of which we prove in \autoref{Th:Dunford Pettis forms}. Boiled down to the $1$-form setting, it states that any linear operator which is bounded from $\Ell^1(T^*\mms)$ to $\Ell^\infty(T^*\mms)$ in the Banach sense admits an integral kernel, the concept of which is similar to the axiomatization of the $1$-form heat kernel. Now for $t>0$, the heat operator $\HHeat_t$ is not bounded from $\Ell^1(T^*\mms)$ to $\Ell^\infty(T^*\mms)$ in this generality. But by \cite{tamanini2019}, given any $\varepsilon > 0$ there exist constants $C_1, C_2 > 0$ with
	\begin{align*}
	\sfp_t(x,y) \leq \meas\big[B_{\sqrt{t}}(x)\big]^{-1/2}\,\meas\big[B_{\sqrt{t}}(x)\big]^{-1/2}\,\exp\!\Big(C_1\big(1+C_2\,t\big) - \frac{\met^2(x,y)}{(4+\varepsilon)t}\Big)
	\end{align*} 
	for every $t>0$ and $\meas^{\otimes}$-a.e.~$(x,y)\in\mms^2$. By \autoref{Th:Kato Simon intro}, the  perturbed operator
	\begin{align*}
	\sfA_t := \phi_t\,\HHeat_t\,\phi_t,
	\end{align*}
	where $\smash{\phi_t(x) := \meas[B_{\sqrt{t}}(x)]^{1/2}}$, 	is thus bounded from $\Ell^1(T^*\mms)$ to $\Ell^\infty(T^*\mms)$ and thus admits an integral kernel --- formally multiplying $\sfA_t$ by $\phi_t^{-1}$ from both sides then yields the desired integral kernel $\Hheat_t$ for $\HHeat_t$. Note that for this argument, it is essential that $(\ChHeat_t)_{t\geq 0}$ has a heat kernel. (This explains best our restriction to uniform lower Ricci bounds, a more general result is not available up to now.)
	
	Having existence of $\Hheat$ at our disposal, further properties of $\Hheat$ such as symmetry, Hess--Schrader--Uhlenbrock's inequality for the ``pointwise operator norm'' $\vert\Hheat_t\vert_\otimes$ of $\Hheat_t$, i.e.~for every $t>0$,
	\begin{align*}
	\vert\Hheat_t\vert_\otimes(x,y)\leq \rme^{-Kt}\,\sfp_t(x,y)
	\end{align*}
	for $\meas^{\otimes 2}$-a.e.~$(x,y)\in\mms^2$, and Chapman--Kolmogorov's formula are shown in \autoref{Th:Properties heat kernel}.
	
	Two further results are then finally given on the class of $\RCD^*(K,N)$ spaces. In \autoref{Pr:Trace}, for every $t>0$ we first prove the trace inequality
	\begin{align*}
	\tr \HHeat_t \leq (\dim_{\met,\meas}\mms)\,\rme^{-Kt}\,\tr\ChHeat_t.
	\end{align*}
	Here, $\dim_{\met,\meas}\mms$, a positive integer not larger than $N$, is the \emph{essential dimension} of $(\mms,\met,\meas)$ in the sense of \cite{brue2020,mondino2019}. This generalizes similar results on possibly weighted Riemannian manifolds \cite{gueneysu2017, hess1980, rosenberg1988}. Furthermore, our spectral analysis for $\Hodge$ from \autoref{Th:Compactness Delta^-1 intro} entails a spectral resolution identity for $\Hheat_t$ in \autoref{Cor:Spectral resol} as soon as $\mms$ is also compact.
	
	\subsubsection*{Organization} In \autoref{Ch:Preliminaries}, we collect basic preliminaries about $\RCD$ spaces and differential geometric notions on these.
	
	In \autoref{Ch:Heat flow pointwise props}, we recall basic properties of the heat equation for $1$-forms and its solution. We establish the vector $1$-Bochner inequality, and prove the important pointwise properties \autoref{Th:Kato Simon intro} and \autoref{Th:BL intro} of $(\HHeat_t)_{t\geq 0}$.
	
	\autoref{Ch:Integral estimates} is devoted to integral consequences of the results from \autoref{Ch:Heat flow pointwise props}. We demonstrate that $(\HHeat_t)_{t\geq 0}$ extends to a semigroup acting on $\Ell^p(T^*\mms)$ for every $p\in [1,\infty]$ and discuss the $\Ell^p$-$\Ell^\infty$-regularization property of $(\HHeat_t)_{t\geq 0}$.  This chapter is then closed by a discussion on hyper- and ultracontractivity properties of $(\HHeat_t)_{t\geq 0}$ and their relation to logarithmic Sobolev inequalities, for instance, the treatise of \autoref{Th:LSI intro}.
	
	\autoref{Ch:Spectral props} contains spectral properties of $\Hodge$, including the proofs of \autoref{Th:Spectral bounds Intro} and \autoref{Th:Compactness Delta^-1 intro}. \autoref{Sec:Independence spectrum} is devoted to the independence of the $\Ell^p$-spectrum of $\Hodge$ on  $p\in[1,\infty]$.
	
	The most important part of this work, \autoref{Ch:Heat kernel}, consists of a careful axiomatization of the notion of integral kernels in the context of $\Ell^\infty$-modules, the proof of \autoref{Th:Heat kernel intro}, and further basic properties of the heat kernel for $(\HHeat_t)_{t\geq 0}$, i.e.~the previously described trace inequality and the spectral resolution identity for $\Hheat_t$ in \autoref{Sub:Fund sol}.
	
	\section{Synthetic Ricci bounds and Ricci curvature}\label{Ch:Preliminaries}
	
	\subsection{Basic preliminaries}
	
	\subsubsection*{Notation} The triple $(\mms,\met,\meas)$ consists of a complete and separable metric space $(\mms,\met)$ and a Borel measure $\meas$ on $(\mms,\met)$ with full support. We always assume that $(\mms,\met,\meas)$ is an $\RCD(K,\infty)$ space, $K\in \R$. In particular, we have the following volume growth condition according to \cite{sturm2006a}. For every $z\in\mms$, there exists $C<\infty$ such that for every $r>0$,
	\begin{align}\label{Eq:Volume growth}
	\meas[B_r(z)] \leq C\,\rme^{Cr^2}.
	\end{align}
	
	By $\Ell^0(\mms)$, we denote the space of ($\meas$-a.e.~equivalence classes of) $\meas$-measurable functions $f\colon\mms\to\R$. Given any partition $(E_j)_{j\in\N}$ of $\mms$ into Borel sets of finite and positive $\meas$-measure, it is a complete and separable metric space w.r.t.~the metric $\met_{\Ell^0}$ defined through
	\begin{align*}
	\met_{\Ell^0}(f,g) := \sum_{j=1}^\infty\frac{2^{-j}}{\meas[E_j]}\int_{E_j}\min\!\big\lbrace \vert f-g\vert,1\big\rbrace\d\meas.
	\end{align*}
	The induced topology on $\Ell^0(\mms)$ does not depend on the choice of $(E_j)_{j\in\N}$ --- indeed, $(f_n)_{n\in\N}$ is a Cauchy sequence w.r.t.~$\met_{\Ell^0}$ if and only if it is a Cauchy sequence w.r.t.~convergence in $\meas$-measure on any Borel set $B\subset\mms$ with $\meas[B] < \infty$. Similar facts hold --- and definitions are used --- for the space $\Ell^0(\mms^2)$ of $\meas^{\otimes 2}$-measurable $f\colon\mms^2\to\R$. We write $\Ell^p(\mms)$ for the $p$-th order Lebesgue space w.r.t.~$\meas$, $p\in [1,\infty]$, endowed with the norm $\Vert \cdot\Vert_{\Ell^p}$. $\Cont(\mms)$ and $\Lip(\mms)$ denote the sets of $\met$-continuous and $\met$-Lipschitz continuous functions $f\colon\mms\to\R$ --- their subspaces of \emph{bounded} or \emph{boundedly supported} functions are marked with the subscript $\bounded$ or $\bs$, respectively. 
	
	For $i\in \{1,2\}$, the projection maps $\pr_i\colon\mms^2\to\mms$ are  $\pr_i(x_1,x_2) := x_i$. Similarly, we define $\pr_i\colon\mms^3\to\mms$ for $i\in\{1,2,3\}$. The distance function from a given point $z\in\mms$ is denoted by $\sfr_z\colon\mms\to\R$ and defined by $\sfr_z(x) := \met(z,x)$.
	
	Let $\Meas(\mms)$ be the space of signed Radon measures on $(\mms,\met)$ with finite total variation.
	
	For two Banach spaces $\scrM$ and $\scrN$, the operator norm of an operator $\sfA\colon \scrM \to \scrN$ is $\Vert \sfA\Vert_{\scrM,\scrN}$.

	\subsubsection*{Notions of $\RCD(K,\infty)$ spaces} We assume the definition of $\RCD(K,\infty)$ spaces, $K\in\R$, to be known to the reader. In this paragraph, we only collect basic properties of them --- details can be found in \cite{ambr2015,ambrosio2014a, ambrosio2014b, gigli2018, lott2009, savare2014, sturm2006a} and the references therein. 
	
	Let $\Sobo^2(\mms)$ be the class of functions $f\in\Ell^0(\mms)$ which have a (squared) \emph{minimal weak upper gradient} $\Gamma(f) \in \Ell^1(\mms)$, whose polarization is also denoted by $\Gamma$. By the $\RCD(K,\infty)$ assumption, $(\mms,\met,\meas)$ is \emph{infinitesimally Hilbertian} --- that is, $(W^{1,2}(\mms),\Vert\cdot\Vert_{W^{1,2}})$ is a Hilbert space, where
	\begin{align*}
	W^{1,2}(\mms) &:= \Sobo^2(\mms)\cap \Ell^2(\mms),\\
	\Vert \cdot \Vert_{W^{1,2}}^2 &:= \Vert \cdot\Vert_{\Ell^2}^2 + \Vert \Gamma(\cdot)\Vert_{\Ell^1}.
	\end{align*}
	
	The infinitesimal generator of the Dirichlet form induced by $\Gamma$ (the so-called \emph{Cheeger energy}) is the \emph{Laplacian} $\Delta$. It is linear, nonpositive, self-adjoint and densely defined on $\Ell^2(\mms)$ with domain $\Dom(\Delta)$. It also gives rise to the $\meas$-symmetric, mass-preserving (\emph{functional}) \emph{heat flow} $(\ChHeat_t)_{t\geq 0}$ on $\Ell^2(\mms)$ via $\ChHeat_t := \rme^{\Delta t}$, which extends to a contractive semigroup of linear operators to $\Ell^p(\mms)$ for every $p\in [1,\infty]$, strongly continuous if $p<\infty$. If $f\in W^{1,2}(\mms)$, the curve $t\mapsto \ChHeat_tf$ is  $W^{1,2}$-continuous on $[0,\infty)$. For every $t>0$, $\ChHeat_t \Delta = \Delta\ChHeat_t$ on $\Dom(\Delta)$, and for every $f\in\Ell^2(\mms)$, 
	\begin{align}\label{Eq:A priori estimate Laplacian}
	\Vert \Delta\ChHeat_tf\Vert_{\Ell^2}\leq \frac{1}{\sqrt{2}t}\,\Vert f\Vert_{\Ell^2}.
	\end{align}
	
	The $\RCD(K,\infty)$ property entails further properties of $(\ChHeat_t)_{t\geq 0}$ crucially exploited in this work. First, as a consequence of heat flow analysis, the \emph{Sobolev-to-Lipschitz property} holds, i.e.~every $f\in W^{1,2}(\mms)$ with $\Gamma(f)\in\Ell^\infty(\mms)$ has an $\meas$-a.e.~representative in $\Lip(\mms)$ whose Lipschitz constant is no larger than $\Vert \Gamma(f)^{1/2}\Vert_{\Ell^\infty}$. Second, $\ChHeat_t$ has a version which maps $\Ell^\infty(\mms)$ to $\Lip(\mms)$ for every $t>0$. Third, the \emph{1-Bakry--Émery inequality} holds, i.e.~for every $f\in W^{1,2}(\mms)$ and every $t\geq 0$,
	\begin{align}\label{Eq:Bakry-Emery L^p}
	\Gamma(\ChHeat_t f)^{1/2} \leq \rme^{-Kt}\,\ChHeat_t\big(\Gamma(f)^{1/2}\big)\quad\meas\text{-a.e.}
	\end{align}
	Fourth, the functional heat flow admits a \emph{heat kernel}, i.e.~for every $t>0$ there exists a symmetric $\meas^{\otimes 2}$-measurable map $\sfp_t\colon \mms^2 \to (0,\infty)$ such that 
	\begin{align*}
		\int_\mms \sfp_t(\cdot,y)\d\meas(y) = 1
	\end{align*}
	and, for every $f\in\Ell^2(\mms)$,
	\begin{align*}
	\ChHeat_t f = \int_\mms \sfp_t(\cdot,y)\,f(y)\d\meas(y)\quad\meas\text{-a.e.}\noindent
	\end{align*}
	
	The following result from \cite{tamanini2019} establishes a Gaussian upper bound for $\sfp_t$, $t>0$. 
	\begin{theorem}\label{Th:Heat kernel bound thm} For every $\varepsilon > 0$, there exist finite constants $C_1>0$, depending only on $\varepsilon$, and $C_2 \geq 0$, depending only on $K$, such that for every $t>0$ and $\meas^{\otimes 2}$-a.e.~$(x,y)\in\mms^2$,
		\begin{align*}
		\sfp_t(x,y) \leq \meas\big[B_{\sqrt{t}}(x)\big]^{-1/2}\,\meas\big[B_{\sqrt{t}}(y)\big]^{-1/2}\,\exp\!\Big(C_1\big(1+C_2t\big) - \frac{\met^2(x,y)}{(4+\varepsilon)t}\Big).
		\end{align*}
		If $K\geq 0$, the constant $C_2$ can be chosen equal to zero.
	\end{theorem}
	
	Closely related to $(\ChHeat_t)_{t\geq 0}$ and its regularizing properties is the set of \emph{test functions}
	\begin{align*}
	\Test(\mms) := \big\lbrace f \in \Dom(\Delta)\cap\Ell^\infty(\mms) : \Gamma(f)\in\Ell^\infty(\mms),\ \Delta f\in W^{1,2}(\mms) \big\rbrace,
	\end{align*}
	which is an algebra w.r.t.~pointwise multiplication. For instance, the heat operator $\ChHeat_t$ maps $\Ell^2(\mms)\cap\Ell^\infty(\mms)$ to $\Test(\mms)$ for every $t>0$, whence $\Test(\mms)$ is dense in $W^{1,2}(\mms)$. Variants of this have widely been used in the literature, e.g.~yielding the subsequent useful results \cite{gigli2018,savare2014}.
	
	\begin{lemma}\label{Le:Mollified heat flow} For every $f\in W^{1,2}(\mms)$ with $a \leq f\leq b$ $\meas$-a.e., $a,b \in [-\infty,\infty]$, there exists a sequence $(f_n)_{n\in\N}$ in $\Test(\mms)$ converging to $f$ in $W^{1,2}(\mms)$ with $a\leq f_n\leq b$ $\meas$-a.e.~and $\Delta f_n\in \Ell^\infty(\mms)$ for all $n\in\N$.		If $\Gamma(f)\in \Ell^\infty(\mms)$, this sequence can be chosen such that $(\Gamma(f_n))_{n\in\N}$ is bounded in $\Ell^\infty(\mms)$.
	\end{lemma}
	
	\begin{lemma}\label{Le:Cutoff function} Assume that $\mms$ is locally compact, and let $U, V\subset\mms$ be two open subsets with $U\subset V$. Suppose that $U$ and $V^\rmc$ have positive $\met$-distance to each other. Then there exists a continuous function $\psi_{U,V} \in \Test(\mms)$ satisfying $\psi_{U,V}(\mms) \subset [0,1]$, $\psi_{U,V} = 1$ ond $U$ and $\psi_{U,V} = 0$ on $V^\rmc$.
	\end{lemma}
	
	\subsubsection*{Notions of $\RCD^*(K,N)$ spaces} We recall useful properties of $\RCD(K,\infty)$ spaces admitting a synthetic notion of ``upper dimension bound'' $N\in (1,\infty)$ in addition, the so-called $\RCD^*(K,N)$ spaces. See \cite{erbar2015, lott2009, sturm2006b} and the references therein for details.
	
	The first is the following corollary of \emph{Bishop--Gromov's inequality}. For every $D>0$, whenever $0< r < R < D$, there exists a constant $C<\infty$ depending only on $K$, $N$ and $D$ such that
	\begin{align}\label{Eq:Bishop-Gromov}
	\frac{\meas[B_R(x)]}{\meas[B_r(x)]}  \leq C\,\Big(\frac{R}{r}\Big)^N
	\end{align}
	for every $x\in\mms$ --- in particular, since $B_1(y) \subset B_{1+\met(x,y)}(x)$, for every $y\in\mms$,
	\begin{align}\label{Eq:AHPT}
	\meas[B_1(y)]\leq C\,\rme^{N\met(x,y)}\,\meas[B_1(x)].
	\end{align}
	A further consequence of \eqref{Eq:Bishop-Gromov} is that $\meas$ is \emph{locally doubling}, that is, for every $x\in\mms$ and $r\in (0,D)$ we have
	\begin{align}\label{Eq:Doubling condition}
	\meas[B_{2r}(x)] \leq 2^N\,C\,\meas[B_r(x)].
	\end{align}
	In turn, this condition implies local compactness of $\mms$. In particular, every finite diameter $\RCD^*(K,N)$ space is necessarily compact --- this is in particular the case when $K>0$.
	
	Since $\RCD^*(K,N)$ spaces also satisfy local $(1,1)$- and $(2,2)$-Poincaré inequalities \cite{erbar2015, rajala2012}, the general study from \cite{sturm1995, sturm1996} yields the existence of a locally Hölder continuous representative of the heat kernel $\sfp$ on $(0,\infty)\times\mms^2$. By \cite{jiang2016}, for every $\varepsilon > 0$, there exist constants $C_3,C_4> 1$ depending only on $K$, $N$ and $\varepsilon$, such that for every $x,y\in\mms$ and every $t>0$,
	\begin{align}\label{Eq:Heat kernel bound}
	\sfp_t(x,y) \leq C_3\,\meas\big[B_{\sqrt{t}}(x)\big]^{-1}\,\exp\!\Big(C_4t -\frac{\met^2(x,y)}{(4+\varepsilon)t}\Big).
	\end{align}
	
	\subsection{A glimpse on nonsmooth differential geometry}\label{Sec:A glimpse}
	
	\subsubsection*{$\Ell^\infty$-modules} The next two paragraphs summarize \cite[Ch.~1]{gigli2018}, creating the framework of spaces of higher order differential objects. This toolbox is particularly needed in \autoref{Sec:Dunford-Pettis}.
	
	\begin{definition} Given any $p\in [1,\infty]$, we call a real Banach space $(\scrM,\Vert\cdot\Vert_{\scrM})$ (or simply $\scrM$ if its norm is understood) an \emph{$\Ell^p$-normed $\Ell^\infty$-module} over $(\mms,\met,\meas)$ if is endowed with
		\begin{enumerate}[label=\textnormal{\alph*.}]
			\item a bilinear \emph{multiplication mapping} $\cdot\colon \Ell^\infty(\mms)\times \scrM\to\scrM$ satisfying
			\begin{align*}
			(f\, g)\cdot v&= f\cdot(g\cdot v),\\
			\One_\mms\cdot v &= v,
			\end{align*}
			\item a nonnegatively valued mapping $\vert \cdot\vert_\scrM\colon\scrM\to \Ell^p(\mms)$, the \emph{pointwise norm}, obeying
			\begin{align*}
			\Vert v\Vert_{\scrM} &= \Vert\vert v\vert_\scrM\Vert_{\Ell^p},\\
			\vert f\cdot v\vert_\scrM &= \vert f\vert\,\vert v\vert_\scrM\quad\meas\text{-a.e.}
			\end{align*}
		\end{enumerate}
		for every $v\in\scrM$ and every $f,g\in\Ell^\infty(\mms)$.
	\end{definition}
	
	We shall drop the $\cdot$ sign and leave out the subscript $\scrM$ from all pointwise norms --- it is always clear from the context which one is considered. A simple example of an $\Ell^p$-normed $\Ell^\infty$-module is $\Ell^p(\mms)$ itself. The pointwise norm $\vert\cdot\vert$ is local, i.e.~for every $v\in\scrM$, $\One_B\,v = 0$ if and only if $\vert v\vert = 0$ $\meas$-a.e.~on $B$ for every Borel set $B\subset\mms$, and it satisfies the pointwise $\meas$-a.e.~triangle inequality. The set of all $v\in\scrM$ such that $\One_{B^\rmc}\,v = 0$ for some bounded Borel set $B\subset\mms$ will be termed $\scrM_\bs$. 
	
	A \emph{Hilbert module} is an $\Ell^2$-normed $\Ell^\infty$-module $\scrM$ which is a Hilbert space. In this case, $\vert\cdot\vert$ satisfies the pointwise $\meas$-a.e.~parallelogram identity, hence induces a \emph{pointwise scalar product} $\langle\cdot,\cdot\rangle \colon \scrM^2 \to \Ell^1(\mms)$. The latter is $\Ell^\infty$-bilinear, local in both components and obeys the pointwise $\meas$-a.e.~Cauchy--Schwarz inequality.

	The \emph{dual module} $\scrM^*$ of $\scrM$ is the set of all linear maps $\sfL\colon \scrM\to\Ell^1(\mms)$ for which $\Vert\sfL\Vert_{\scrM,\Ell^1} <\infty$ and $\sfL(f\,v) = f\,\sfL v$ for every $v\in\scrM$ and every $f\in\Ell^\infty(\mms)$. We endow $\scrM^*$ with the usual operator norm. Whenever convenient, we denote the pairing $\sfL v$ of $\sfL\in\scrM^*$ and $v\in\scrM$ by $\langle v\mid \sfL\rangle$, $\sfL(v)$ or $v(\sfL)$. Then $\scrM^*$ is an $\Ell^q$-normed $\Ell^\infty$-module, where  $q\in [1,\infty]$ satisfies $1/p+1/q=1$, its pointwise norm $\vert\cdot\vert\colon\scrM^*\to\Ell^q(\mms)$ being given by
	\begin{align*}
	\vert\sfL\vert := \esssup\!\big\lbrace \vert \langle v \mid \sfL\rangle\vert: v\in \scrM,\ \vert v\vert\leq 1\ \meas\text{-a.e.}\big\rbrace.
	\end{align*}
	
	\subsubsection*{$\Ell^0$-modules} Let $\scrM$ be an $\Ell^p$-normed $\Ell^\infty$-module, $p\in [1,\infty]$. By $\scrM^0$ we intend the \emph{$\Ell^0$-module} corresponding to $\scrM$ from \cite[Sec.~1.3]{gigli2018} --- the completion of $\scrM$ w.r.t.~the metric $\met_{\scrM^0}$ given by
	\begin{align*}
	\met_{\scrM^0}(v,w) := \sum_{j=1}^\infty \frac{2^{-j}}{\meas[E_j]}\int_{E_j}\min\!\big\lbrace \vert v-w\vert,1\big\rbrace\d\meas.
	\end{align*}
	Here, $(E_j)_{j\in\N}$ is a partition of $\mms$ into Borel sets of finite and positive $\meas$-measure. 
	
	Roughly speaking, $\scrM^0$ is some larger space of ``$\meas$-measurable elements $v$'' for which $v\in\scrM$ if and only if $\vert v\vert\in\Ell^p(\mms)$. Along with the construction of $\scrM^0$, the multiplication, the pointwise norm and the pointwise pairing operations on $\scrM$ uniquely extend to maps $\cdot\colon\Ell^0(\mms)\times \scrM^0\to\scrM^0$, $\vert\cdot\vert\colon\scrM^0\to\Ell^0(\mms)$ and $\langle\cdot\mid\cdot\rangle\colon\scrM^0\times (\scrM^0)^*\to \Ell^0(\mms)$  which satisfy similar properties as their former counterparts w.r.t.~elements in $\scrM^0$ and $\Ell^0(\mms)$ instead of $\scrM$ and $\Ell^\infty(\mms)$, respectively.
	
	\subsubsection*{Nonsmooth first order calculus} This paragraph surveys \cite[Ch.~2]{gigli2018}, i.e.~the introduction of the relevant spaces of $\Ell^2$-$1$-forms and $\Ell^2$-vector fields over $(\mms,\met,\meas)$. 
	
	We denote the \emph{cotangent module} constructed in \cite[Sec.~2.2]{gigli2018} by $\Ell^2(T^*\mms)$. Elements of it are called \emph{\textnormal{(}differential\textnormal{)} $1$-forms}. The \emph{tangent module}, whose elements are called \emph{vector fields}, is defined by $\Ell^2(T\mms) := \Ell^2(T^*\mms)^*$. The norms on both spaces are termed $\Vert\cdot\Vert_{\Ell^2}$. $\Ell^2(T^*\mms)$ and $\Ell^2(T\mms)$ are separable Hilbert modules and are isometrically isomorphic. The isometry is even true pointwise $\meas$-a.e., via the \emph{musical isomorphisms} $\flat\colon \Ell^2(T\mms)\to\Ell^2(T^*\mms)$ and $\sharp\colon \Ell^2(T^*\mms)\to\Ell^2(T\mms)$ with
	\begin{align*}
	X^\flat(Y) := \langle X,Y\rangle\quad\text{and}\quad\langle\omega^\sharp,Y\rangle := \omega(Y). 
	\end{align*}
	
	The construction of $\Ell^2(T^*\mms)$ comes with a linear and continuous \emph{differential} map $\rmd\colon \Sobo^2(\mms)\to\Ell^2(T^*\mms)$ for which $\vert \rmd f\vert = \Gamma(f)^{1/2}$ $\meas$-a.e.~for every $f\in\Sobo^2(\mms)$. It is closed in the sense that if $(f_n)_{n\in\N}$ is a sequence in $\Sobo^2(\mms)$ converging to $f\in\Ell^0(\mms)$ pointwise $\meas$-a.e., and if $(\rmd f_n)_{n\in\N}$ converges to some $\omega\in\Ell^2(T^*\mms)$ w.r.t.~$\Vert\cdot\Vert_{\Ell^2}$, then $f\in\Sobo^2(\mms)$ and $\omega = \rmd f$.  The \emph{gradient} $\nabla$ of $f\in\Sobo^2(\mms)$ is then defined by $\nabla f := (\rmd f)^\sharp\in\Ell^2(T\mms)$ --- equivalently, $\nabla f$ is the unique element $Y\in \Ell^2(T\mms)$ such that $\rmd f(Y) =  \vert\rmd f\vert^2= \vert Y\vert^2$ $\meas$-a.e. It obeys similar calculus rules as the following ones for $\rmd$.
	\begin{lemma}\label{Le:Calculus rules d} \begin{enumerate}[label=\textnormal{(\roman*)}]
			\item \textnormal{\textsc{Locality.}} For every $\Leb^1$-negligible Borel set $A\subset\R$ and every $f\in\Sobo^2(\mms)$, we have $\One_{f^{-1}(A)}\,\rmd f = 0$. In particular, $\One_{\{c\}}(f)\,\rmd f = 0$ for every $c\in\R$.
			\item \textnormal{\textsc{Chain rule.}} For every $f\in\Sobo^2(\mms)$ and every $\Phi\in\Lip(\R)$, define $\Phi'(f)$ arbitrarily on the preimage of all non-differentiability points of $\Phi$ under $f$. Then $\Phi(f)\in\Sobo^2(\mms)$ and
			\begin{align*}
			\rmd\Phi(f) = \Phi'(f)\,\rmd f.
			\end{align*}
			\item \textnormal{\textsc{Leibniz rule.}} For every $f,g\in\Sobo^2(\mms)\cap\Ell^\infty(\mms)$, $f\,g\in\Sobo^2(\mms)$ with
			\begin{align*}
			\rmd(f\,g) = f\,\rmd g + g\,\rmd f.
			\end{align*}
		\end{enumerate} 
	\end{lemma}
	\begin{definition} Let $\Dom(\div)$ be the space of all $X\in\Ell^2(T\mms)$ for which there exists a function $f\in\Ell^2(\mms)$ such that for every $g\in W^{1,2}(\mms)$,
		\begin{align*}
		\int_\mms g\,f\d\meas = -\int_\mms \rmd g(X)\d\meas.
		\end{align*}
		In case of existence, $f$ is unique, denoted by $\div X$ and called the \emph{divergence} of $X$.
	\end{definition}
	
	By the integration by parts formula for $\Delta$, for every $f\in\Dom(\Delta)$, we have $\nabla f\in\Dom(\div)$ with the usual identity $\div \nabla f = \Delta f$. Moreover, given any $X\in\Dom(\div)$ and any $f\in\Sobo^2(\mms)\cap\Ell^\infty(\mms)$ with $\vert \rmd f\vert \in\Ell^\infty(\mms)$, we have $f\,X\in\Dom(\div)$ with
	\begin{align}\label{Eq:div product rule}
	\div(f\,X) = \rmd f(X) + f\div X\quad\meas\text{-a.e.}
	\end{align}
	
	\subsubsection*{Further Lebesgue spaces and $\Ell^\infty$-module tensor products} We set
	\begin{align*}
	\Ell^0(T^*\mms) &:= \Ell^2(T^*\mms)^0,\\
	\Ell^0(T\mms) &:= \Ell^2(T\mms)^0.
	\end{align*}
	
	$\Ell^p(T^*\mms)$ is the class of all $v \in \Ell^0(T^*\mms)$ with $\vert v\vert\in\Ell^p(T^*\mms)$, where $p\in [1,\infty]$. It is an $\Ell^p$-normed $\Ell^\infty$-module w.r.t.~the canonical norm $\Vert\cdot\Vert_{\Ell^p}$, separable if $p<\infty$. For every $p,q\in [1,\infty]$ with $1/p+1/q=1$, $\Ell^p(T^*\mms)^*$ and $\Ell^q(T^*\mms)$ are isometrically isomorphic to each other as modules. Define $\Ell^p(T\mms)$ analogously for $p\in [1,\infty]$.
	
	Denote by $\Ell^2((T^*)^{\otimes 2}\mms)$ and $\Ell^2(T^{\otimes 2}\mms)$ the two-fold \emph{tensor products} of $\Ell^2(T^*\mms)$ and $\Ell^2(T\mms)$ in the \emph{$\Ell^\infty$-module sense} of \cite[Sec.~1.5]{gigli2018}. These spaces are separable Hilbert modules in which the linear spans of elements of the kind $\omega_1\otimes\omega_2$ and $X_1\otimes X_2$ with $\omega_1,\omega_2\in\Ell^2(T^*\mms)\cap\Ell^\infty(T^*\mms)$ and $X_1,X_2\in\Ell^2(T\mms)\cap\Ell^\infty(T\mms)$, are respectively dense. The corresponding pointwise norms $:$ on $\Ell^2((T^*)^{\otimes 2}\mms)$ and $\Ell^2(T^{\otimes 2}\mms)$ --- which turn both spaces isometrically isomorphic --- as well as $\langle \cdot,\cdot\rangle$ on $\Ell^2(\Lambda^2T^*\mms)$ are initially defined --- and then extended by approximation --- by
	\begin{align*}
	(\omega_1\otimes\omega_2) : (X_1^\flat\otimes X_2^\flat) &:= \omega_1(X_1)\,\omega_2(X_2),\\
	\langle \omega_1\wedge\omega_2, X_1^\flat\wedge X_2^\flat\rangle &:= \det \omega_i(X_j).
	\end{align*}
	
	\subsubsection*{Hilbert--Schmidt operators and Hilbert space tensor products} We call a linear operator $\sfS\colon \Ell^2(T^*\mms)\to\Ell^2(T^*\mms)$ a \emph{Hilbert--Schmidt operator} if for some --- or equivalently \emph{any} --- countable orthonormal bases $(\omega_i)_{i\in\N}$ and $(\eta_{i'})_{i'\in\N}$ of $\Ell^2(T^*\mms)$,
	\begin{align*}
	\Vert \sfS\Vert_{\HS}^2 := \sum_{i,i'=1}^\infty \Big[\! \int_\mms \langle \sfS\omega_i, \eta_{i'}\rangle\d\meas\Big]^2 < \infty.
	\end{align*}
	
	The two-fold \emph{Hilbert space tensor product} $\Ell^2(T^*\mms)^{\otimes 2}$ of $\Ell^2(T^*\mms)$ --- see e.g.~\cite{gigli2018,kadison1983} for the precise definition --- which will be needed in \autoref{Sub:Fund sol}, is isometrically isomorphic to the space of all Hilbert--Schmidt operators from $\Ell^2(T^*\mms)'$ to $\Ell^2(T^*\mms)$, endowed with the norm $\Vert\cdot\Vert_\HS$. Moreover, up to isomorphism, it is characterized by the following universal property \cite[Thm.~2.6.4]{kadison1983}. Given a real Hilbert space $H$, a bilinear mapping $\sfG\colon \Ell^2(T^*\mms)^2\to H$ is termed \emph{weakly Hilbert--Schmidt} if for some --- or equivalently \emph{any} --- countable orthonormal bases $(\omega_i)_{i\in\N}$ and $(\eta_{i'})_{i'\in\N}$  as above,
	\begin{align*}
	\Vert \sfG\Vert_{\WHS}^2 := \sup\!\Big\lbrace \Vert h\Vert_H^{-2}\sum_{i,i'=1}^\infty \big(\sfG(\omega_i,\eta_{i'}), h\big)_{H}^2 : h\in H\setminus\{0\}\Big\rbrace < \infty.
	\end{align*}
	
	\begin{theorem}\label{Th:Univ property} The mapping $\sfe\colon\Ell^2(T^*\mms)^2\to\Ell^2(T^*\mms)^{\otimes 2}$ which is defined by $\sfe(\eta,\omega) :=  \eta\otimes\omega$ is weakly Hilbert--Schmidt. Moreover, given any real Hilbert space $H$, for every weakly Hilbert--Schmidt mapping $\sfG\colon \Ell^2(T^*\mms)^2\to H$, there exists a unique bounded operator $\sfT\colon \Ell^2(T^*\mms)^{\otimes 2}\to H$ such that
		\begin{align*}
		\sfG = \sfT(\sfe)\quad\text{and}\quad \Vert\sfT\Vert_{\Ell^2,H} = \Vert \sfG\Vert_\WHS.
		\end{align*}
	\end{theorem}
	
	From the $\Ell^\infty$-module perspective, it is not difficult to prove the subsequent result which actually holds true for the Hilbert space tensor product of \emph{any} two Hilbert modules over $(\mms,\met,\meas)$. Note that we use the $\otimes$ sign in both cases, although $\Ell^2(T^*\mms)^{\otimes 2}$ differs from $\Ell^2((T^*)^{\otimes 2}\mms)$. For instance, we always have $\omega_1\otimes \omega_2\in \Ell^2(T^*\mms)^{\otimes 2}$ for every $\omega_1,\omega_2\in\Ell^2(T^*\mms)$, but in Gigli's sense, $\omega_1\otimes\omega_2$  does not necessarily belong to $\Ell^2((T^*)^{\otimes 2}\mms)$ unless $\omega_i\in\Ell^\infty(T^*\mms)$ for at least one $i\in\{1,2\}$. (Indeed, one should rather formally think of $\omega_1\otimes\omega_2\in\Ell^2(T^*\mms)^{\otimes 2}$ as section $(x,y) \mapsto \omega_1(x)\otimes\omega_2(y)$ and of $\omega_1\otimes\omega_2\in\Ell^2((T^*)^{\otimes 2}\mms)$ as section $x\mapsto \omega_1(x)\otimes\omega_2(x)$.)

	\begin{proposition}\label{Pr:HS tensor product} The Hilbert space tensor product $\Ell^2(T^*\mms)^{\otimes 2}$ has a natural structure of a Hilbert module over the product space $(\mms^2,\met^2,\meas^{\otimes 2})$ such that the multiplication $\cdot\colon\Ell^\infty(\mms^2)\times\Ell^2(T^*\mms)^{\otimes 2}\to\Ell^2(T^*\mms)^{\otimes 2}$ and the pointwise norm $\vert\cdot\vert\colon \Ell^2(T^*\mms)^{\otimes 2}\to\Ell^2(\mms^2)$ satisfy
		\begin{align*}
		\big(f(\pr_1)\,g(\pr_2)\big)\,(\omega_1\otimes \omega_2) &= (f\,\omega_1)\otimes (g\,\omega_2),\\
		\vert\omega_1\otimes\omega_2\vert &= \vert\omega_1\vert(\pr_1)\,\vert\omega_2\vert(\pr_2)
		\end{align*}
		for every $\omega_1,\omega_2\in\Ell^2(T^*\mms)$ and every $f,g\in\Ell^\infty(\mms)$.
	\end{proposition}
	
	\subsubsection*{Nonsmooth second order calculus} The next two  paragraphs summarize \cite[Ch.~3]{gigli2018}. In Gigli's treatise, Hessian, covariant derivative and exterior differential are defined by integration by parts procedures. However, we do not need their precise defining formulas in these cases, which are thus omitted --- we focus on their calculus rules.
	
	We denote the space of \emph{test $1$-forms} and \emph{test vector fields}, respectively, by
	\begin{align*}
	\Test(T^*\mms) &:= \Span\!\big\lbrace g\d f : f,g\in\Test(\mms)\big\rbrace,\\
	\Test(T\mms) &:= \Test(T^*\mms)^\sharp.
	\end{align*}
	Then $\Test(T^*\mms)$ and $\Test(T\mms)$ are dense in $\Ell^2(T^*\mms)$ and $\Ell^2(T\mms)$, respectively.
	
	Define the linear space $W^{2,2}(\mms)$ as the space of all $f\in W^{1,2}(\mms)$ which admit a \emph{Hessian} $\Hess f\in\Ell^2((T^*)^{\otimes 2}\mms)$. The Hessian of a fixed function is symmetric and $\Ell^\infty$-bilinear. Its graph is a closed subset of $W^{1,2}(\mms)\times \Ell^2((T^*)^{\otimes 2}\mms)$. As an important consequence of the $\RCD(K,\infty)$ assumption, a class which is dense in $\Ell^2(\mms)$ is contained in $W^{2,2}(\mms)$, and the \emph{integrated $2$-Bochner inequality} holds --- more precisely, for every $f\in\Dom(\Delta)$, we have $f\in W^{2,2}(\mms)$ with
	\begin{align}\label{Eq:Integrated Bochner}
	\int_\mms \vert\hspace*{-0.05cm}\Hess f\vert^2\d\meas \leq \int_\mms \big[\vert \Delta f\vert^2 - K\,\vert\rmd f\vert^2\big]\d\meas.
	\end{align}
	
	Next, let $W^{1,2}(T\mms)$ be the linear space of all $X\in \Ell^2(T\mms)$ admitting a \emph{covariant derivative} $\nabla X\in\Ell^2(T^{\otimes 2}\mms)$. The graph of the  operator $\nabla$ is a closed subset of $\Ell^2(T\mms)\times \Ell^2(T^{\otimes 2}\mms)$. We have $\Test(T\mms)\subset W^{1,2}(T\mms)$ --- more precisely, given any $f,g\in \Test(\mms)$, we have $g\,\nabla f\in W^{1,2}(T\mms)$ with
	\begin{align*}
	\nabla(g\,\nabla f) = \nabla g\otimes \nabla f + g\,(\Hess f)^\sharp.
	\end{align*}
	Moreover, $\nabla f\in W^{1,2}(T\mms)$ for every $f\in W^{2,2}(\mms)$ with $\nabla^2 f := \nabla\nabla f = (\Hess f)^\sharp$. Lastly, $\nabla$ is compatible with $\langle\cdot,\cdot\rangle$ on $\Ell^2(T\mms)^2$ in the following way. Let
	\begin{align*}
	H^{1,2}(T\mms) &= \cl_{\Vert\cdot\Vert_{H^{1,2}}} \Test(T\mms),\\
	\Vert\cdot\Vert_{H^{1,2}}^2 &:= \Vert\cdot\Vert_{\Ell^2}^2 + \Vert \nabla\cdot\Vert_{\Ell^2}^2,
	\end{align*}
	a separable Hilbert space,  dense in $\Ell^2(T\mms)$.  Given any $X\in H^{1,2}(T\mms)\cap\Ell^\infty(T\mms)$ and $Z\in \Ell^2(T\mms)$, let $\nabla_Z X\in\Ell^0(T\mms)$ be the vector field uniquely defined by $\langle\nabla_Z X, V\rangle := \nabla X : (Z\otimes V)$ for every $V\in\Ell^0(T\mms)$. Then for every $Y\in H^{1,2}(T\mms)\cap\Ell^\infty(T\mms)$, we have $\langle X,Y\rangle\in W^{1,2}(\mms)$ with
	\begin{align}\label{Eq:Compatibility with the metric}
	\rmd\langle X,Y\rangle(Z) = \langle\nabla_Z X,Y\rangle + \langle\nabla_ZY,X\rangle\quad\meas\text{-a.e.}
	\end{align}
	
	We next turn to the contravariant picture. Let $\Dom(\rmd)$ be the linear space of all $\omega\in\Ell^2(T^*\mms)$ which have an \emph{exterior differential} $\rmd \omega\in\Ell^2(\Lambda^2T^*\mms)$. This gives rise to a closed operator $\rmd$ on $\Ell^2(T^*\mms)$. For every $f,g\in\Test(\mms)$, we have $g\d f\in \Dom(\rmd)$ with
	\begin{align*}
	\rmd(g\d f) := \rmd g\wedge\rmd f.
	\end{align*}
	Lastly, we have $\rmd f\in \Dom(\rmd)$ for every $f\in W^{1,2}(\mms)$ with $\rmd^2 f := \rmd\rmd f = 0$. The formal adjoint of the exterior differential is the \emph{codifferential} $\delta$. Its domain $\Dom(\delta)$ is the space of all differential $1$-forms $\omega\in\Ell^2(T^*\mms)$ for which $\omega^\sharp\in\Dom(\div)$, in which case we define $\delta\omega := -\div \omega^\sharp$. In particular, if $f\in\Dom(\Delta)$, then $\rmd f\in \Dom(\delta)$ with $\delta\rmd f = -\Delta f$. The graph of $\delta$ is a closed subset of $\Ell^2(T^*\mms)\times \Ell^2(\mms)$. Having now these two notions at our disposal,  define the \emph{Hodge space} 
	\begin{align*}
	H^{1,2}(T^*\mms) &:= \cl_{\Vert\cdot\Vert_{H^{1,2}}}\Test(T^*\mms),\\
	\Vert\cdot\Vert_{H^{1,2}}^2 &:= \Vert\cdot\Vert_{\Ell^2}^2 + \Vert \rmd\cdot\Vert_{\Ell^2}^2 + \Vert \delta\cdot\Vert_{\Ell^2}^2,
	\end{align*}
	a separable Hilbert space which is dense in $\Ell^2(T^*\mms)$. The \emph{Hodge energy functional} $\calE\colon\Ell^2(T^*\mms)\to [0,\infty]$ is 
	\begin{align}\label{Eq:Hodge energy functional}
	\calE(\omega) := \begin{cases}\displaystyle\frac{1}{2}\int_\mms\big[\vert\rmd \omega\vert^2 + \vert\delta\omega\vert^2\big]\d\meas & \text{if }\omega\in H^{1,2}(T^*\mms),\\
	\infty & \text{otherwise}.
	\end{cases}
	\end{align}
	
	We provide the subsequent two lemmata concerning the space $H^{1,2}(T^*\mms)$. We give a proof for \autoref{Le:Differential regularity} for convenience.  \autoref{Le:Multiplication with Sobolev functions} follows by similar results for $\Test(T^*\mms)$ from \cite[Thm.~3.5.2, Prop.~3.5.12]{gigli2018} by approximation as in  \autoref{Le:Mollified heat flow}.
	
	\begin{lemma}\label{Le:Differential regularity} For every $f\in \Ell^2(\mms)$ and every $t>0$, we have $\rmd\ChHeat_tf\in H^{1,2}(T^*\mms)$.
	\end{lemma}
	
	\begin{proof} Given any $f\in \Ell^2(\mms)$ and $n\in\N$, set $f_n := \min\{n,\max\{f,-n\}\} \in\Ell^2(\mms)\cap\Ell^\infty(\mms)$. Then $(\rmd\ChHeat_tf_n)_{n\in\N}$ is a sequence in $\Test(T^*\mms)$ which converges to $\rmd\ChHeat_t f$ in $\Ell^2(T^*\mms)$ for every $t>0$. 
		
		By the closedness of the exterior differential, it follows that $\rmd\ChHeat_t f\in\Dom(\rmd)$ and $\rmd^2\ChHeat_tf_n = \rmd^2\ChHeat_t f = 0$ for every $n\in\N$. Since $\delta\rmd\ChHeat_t f_n = - \Delta\ChHeat_tf_n \to -\Delta\ChHeat_t f = \delta\rmd\ChHeat_tf$ in $\Ell^2(\mms)$ as $n\to\infty$ 
		as a consequence of \eqref{Eq:A priori estimate Laplacian}, we obtain the claim. 
	\end{proof}
	
	\begin{lemma}\label{Le:Multiplication with Sobolev functions} 
		Let $f\in \Sobo^2(\mms)\cap\Ell^\infty(\mms)$ and $\omega\in H^{1,2}(T^*\mms)$. Assume moreover that $\rmd f\in\Ell^\infty(T^*\mms)$ or that $\omega\in\Ell^\infty(T^*\mms)$. Then $f\,\omega \in H^{1,2}(T^*\mms)$ with
		\begin{align*}
		\rmd(f\,\omega) &= f\,\rmd\omega + \rmd f\wedge\omega,\\
		\delta(f\,\omega) &= f\,\delta\omega - \langle\rmd f,\omega\rangle.
		\end{align*}
	\end{lemma}

	\begin{definition} A differential $1$-form $\omega\in H^{1,2}(T^*\mms)$ belongs to $\Dom(\Hodge)$  if there exists $\alpha\in \Ell^2(T^*\mms)$ such that for every $\eta\in H^{1,2}(T^*\mms)$,
		\begin{align*}
		\int_\mms \langle\eta,\alpha\rangle\d\meas =  \int_\mms \big[\langle \rmd\eta,\rmd\omega\rangle + \delta\eta\,\delta\omega\big]\d\meas.
		\end{align*}
		In case of existence, $\alpha$ is unique, denoted by $\Hodge \omega$ and termed  \emph{Hodge Laplacian} of $\omega$.
	\end{definition}
	
	The induced operator $\Hodge$ on $\Ell^2(T^*\mms)$ is nonnegative, self-adjoint and closed. The space of \emph{harmonic $1$-forms}, i.e.~those $\omega\in\Dom(\Hodge)$ with $\Hodge\omega = 0$ --- or equivalently, $\omega\in H^{1,2}(T^*\mms)$ with $\rmd\omega = 0$ and $\delta \omega = 0$ --- is termed $\Harm(T^*\mms)$. Furthermore, for every $f,g\in\Test(\mms)$, we have $g\d f \in\Dom(\Hodge)$ --- actually also with $\smash{\Hodge(g \d f)\in\Ell^1(T^*\mms)}$ --- and
	\begin{align}\label{Le:Hodge Laplacian on test forms}
	\Hodge(g\d f) = -g\d\Delta f -\Delta g\d f - 2\Hess f(\nabla g,\cdot).
	\end{align}

	\subsubsection*{Measure-valued Ricci curvature} The main result of \cite{gigli2018} is the appropriate definition of a measure-valued \emph{Ricci tensor} $\RIC$ on $\RCD(K,\infty)$ spaces. Its final outcome is stated in  \autoref{Th:Ricci tensor}. In the $\RCD^*(K,N)$ framework, $N\in(1,\infty)$, a dimensional $N$-Ricci tensor has been introduced and studied in \cite{han2018}. The main result therein is formulated in \autoref{Pr:Han's bound}.
	
	\begin{definition} Let $\Dom(\DELTA)$ consist of all functions $f\in W^{1,2}(\mms)$ for which there exists a signed measure $\mu\in \Meas(\mms)$ such that for every $g\in \Lip_\bs(\mms)$, 
		\begin{align*}
		\int_\mms g\d\mu = -\int_\mms \langle\nabla g,\nabla f\rangle\d\meas.
		\end{align*}
		In case of existence, $\mu$ is unique, denoted by $\DELTA f$ and called the \emph{measure-valued Laplacian} of $f$.
	\end{definition}
	
	This definition is compatible with the functional Laplacian in the sense that $f\in \Dom(\Delta)$ if and only if $f\in \Dom(\DELTA)$ and $\DELTA f$ has a density $h\in\Ell^2(\mms)$ w.r.t.~$\meas$, in which case $\Delta f = h$ $\meas$-a.e. The following is proven in \cite[Lem.~3.1]{braun2019}.
	
	\begin{lemma}\label{Le:Laplacian functions chain rule} For every $f\in\Dom(\DELTA)\cap\Ell^\infty(\mms)$, every interval $I\subset \R$ with $f(\mms)\cup\{0\}\subset I$ and every $\Phi\in\Cont^2(I)$ with $\Phi(0) = 0$, we have $\Phi(f) \in\Dom(\DELTA)$ with
		\begin{align*}
		\DELTA\Phi(f) = \Phi'(f)\,\DELTA f + \Phi''(f)\,\Gamma(f)\,\meas.
		\end{align*}
	\end{lemma}
	
	A crucial outcome of the $\RCD(K,\infty)$ condition is that $\vert X\vert^2\in\Dom(\DELTA)$ for every $X\in\Test(T\mms)$.
	
	Consider now the image $H^{1,2}(T^*\mms)^\sharp$ of $H^{1,2}(T^*\mms)$ under $\sharp$ and, abusing notation, equip it with the norm $\smash{\Vert \cdot\Vert_{H^{1,2}} := \Vert \cdot^\flat\Vert_{H^{1,2}}}$. A key feature yielding \autoref{Th:Ricci tensor} below is that the inclusion $H^{1,2}(T^*\mms)^\sharp\subset H^{1,2}(T\mms)$ is continuous, i.e.~$X\in H^{1,2}(T\mms)$ for every $X\in H^{1,2}(T^*\mms)^\sharp$, and
	\begin{align}\label{Eq:H embedding}
	\int_\mms \vert\nabla X\vert^2\d\meas \leq \int_\mms \big[\vert\rmd X^\flat\vert^2 + \vert\delta X^\flat\vert^2 - K\,\vert X\vert^2\big]\d\meas.
	\end{align}
	
	\begin{theorem}\label{Th:Ricci tensor} There exists a unique continuous map $\RIC\colon H^{1,2}(T^*\mms)^\sharp\to\Meas(\mms)$ which satisfies
		\begin{align*}
		\RIC(X,Y) = \frac{1}{2}\bdDelta \langle X,Y\rangle + \Big[\frac{1}{2}\big\langle X, (\Hodge Y^\flat)^\sharp\big\rangle + \frac{1}{2}\big\langle Y, (\Hodge X^\flat)^\sharp\big\rangle - \nabla X : \nabla Y\Big]\,\meas
		\end{align*}
		for every $X,Y\in\Test(T\mms)$. It is symmetric and $\R$-bilinear. Furthermore,  for every $X,Y\in H^{1,2}(T^*\mms)^\sharp$ it satisfies the following relations.
		\begin{enumerate}[label=\textnormal{(\roman*)}]
			\item\label{La:RIC lower bound} \textnormal{\textsc{Ricci bound.}} We have
			\begin{align*}
			\RIC(X,X) \geq K\,\vert X\vert^2\,\meas.
			\end{align*}
			\item \textnormal{\textsc{Integrated Bochner formula.}} We have
			\begin{align*}
			\RIC(X,X)[\mms] = \int_\mms \big[\langle \rmd X^\flat,\rmd Y^\flat \rangle + \delta X^\flat\,\delta Y^\flat - \nabla X:\nabla Y\big]\d\meas
			\end{align*}
		\end{enumerate}
	\end{theorem}
	
	\begin{proposition}\label{Pr:Han's bound} If $(\mms,\met,\meas)$ satisfies the $\RCD^*(K,N)$ condition, $N\in (1,\infty)$, then within the notation of \autoref{Th:Ricci tensor}, for every $X\in H^{1,2}(T^*\mms)^\sharp$ we have
		\begin{align*}
		\bdGamma_2(X) &\geq \Big[K\,\vert X\vert^2 + \frac{1}{N}\,\vert\!\div X\vert^2\Big]\,\meas,
\end{align*}
where $\smash{\bdGamma_2(X) := \RIC(X,X) + \vert\nabla X\vert^2\,\meas}$.
	\end{proposition}

	\section{Pointwise estimates for the heat flow}\label{Ch:Heat flow pointwise props}
	
	\subsection{Definition and basic properties}\label{Sub:Def bas props}
	
	The negative Hodge Laplacian $-\Hodge$ is nonpositive, self-adjoint and densely defined on $\Ell^2(T^*\mms)$. By spectral calculus, it generates a family of bounded linear operators $(\HHeat_t)_{t\geq 0}$ on $\Ell^2(T^*\mms)$ via $\smash{\HHeat_t := \rme^{-t\Hodge}}$ for which, given any $\omega\in \Ell^2(T^*\mms)$, the map $t\mapsto \HHeat_t\omega$ defined on $[0,\infty)$ is uniquely characterized by the subsequent three properties. 
	\begin{enumerate}[label=\textnormal{\alph*.}]
		\item\label{La:A} \textsc{Initial value.} $\HHeat_0\omega = \omega$.
		\item\label{La:C} \textsc{Strong continuity.} The map $t\mapsto\HHeat_t\omega$ is strongly $\Ell^2$-continuous on $[0,\infty)$.
		\item\label{La:D} \textsc{Kolmogorov forward equation.} For every $t>0$, it holds that $\HHeat_t\omega\in\Dom(\Hodge)$. Moreover, the limit of $\smash{h^{-1}\,\big(\HHeat_{t+h}\omega-\HHeat_t\omega\big)}$ as $h\to 0$ exists strongly in $\Ell^2(T^*\mms)$ and satisfies
		\begin{align*} 
		\frac{\rmd}{\rmd t} \HHeat_t\omega =-\Hodge\HHeat_t\omega.
		\end{align*}
	\end{enumerate}
	
	By uniqueness, it admits the \emph{semigroup property} $\HHeat_{t+s}\omega = \HHeat_t\HHeat_s\omega$, $s,t\geq 0$.

	\begin{definition}\label{Def:Heat flow} We refer to $(\HHeat_t)_{t\geq 0}$ as the \emph{heat semigroup} or \emph{heat flow} on $1$-forms.
	\end{definition}
	
	\begin{remark}\label{Re:GF definition} One can equivalently define $(\HHeat_t\omega)_{t\geq 0}$ as the gradient flow of the functional $\calE$ from \eqref{Eq:Hodge energy functional} starting in $\omega\in\Ell^2(T^*\mms)$ via Brézis--Komura's theory \cite{brezis1973}, as done in \cite[Sec.~3.6]{gigli2018}. 
		
		Such flow satisfies the conditions \ref{La:A} and \ref{La:C} by construction. The link between \ref{La:D} and the usual differential inclusion appearing in the previous gradient flow approach is due to the fact that the subdifferential of $\calE$ at a given $\eta\in H^{1,2}(T^*\mms)$ is nonempty if and only if $\eta\in \Dom(\Hodge)$, in which case it only contains $\Hodge\eta$. See \cite[Subsec.~3.4.4]{gigli2018} for a similar discussion.
	\end{remark}
	
	Let us list some properties of $(\HHeat_t)_{t\geq 0}$. \autoref{Le:Ht vs Pt} treats the link between $(\HHeat_t)_{t\geq 0}$ and $(\ChHeat_t)_{t\geq 0}$ on exact forms, while  \autoref{Th:L^2-contractivity} is due to standard semigroup theory \cite{brezis1973, reed1975, yosida1980} --- item \ref{La:Bakry-Emery L^2} therein has been proven in \cite[Prop.~3.6.10]{gigli2018} using \autoref{Th:Ricci tensor}.
	
	\begin{lemma}\label{Le:Ht vs Pt} For every $f\in\Sob^{1,2}(\mms)$ and every $t> 0$,  $\rmd \ChHeat_t f\in\Dom(\Hodge)$ and
		\begin{align*}
		\HHeat_t\rmd f = \rmd\ChHeat_t f.
		\end{align*}
	\end{lemma}
	
	\begin{proof} Set $f_n := \min\{n,\max\{f,-n\}\}\in W^{1,2}(\mms)\cap\Ell^\infty(\mms)$ for every $n\in\N$. Then the collection $(\rmd\ChHeat_tf_n)_{t\geq 0}$ in $\Ell^2(T^*\mms)$ clearly satisfies \ref{La:A} above, and \ref{La:C} follows by continuity of $t\mapsto \ChHeat_tf_n$ on $[0,\infty)$ w.r.t.~strong $W^{1,2}$-convergence. To check \ref{La:D}, given any $t>0$, recall that $\rmd\ChHeat_t f_n\in \Test(T^*\mms)$. The identity \eqref{Le:Hodge Laplacian on test forms} and the closedness of $\rmd$ \cite[Thm.~2.2.9]{gigli2018} together with $\Delta\ChHeat_t f_n = \ChHeat_{t/2}\Delta\ChHeat_{t/2}f_n\in W^{1,2}(\mms)$ then imply $\rmd\ChHeat_tf_n\in\Dom(\Hodge)$ and
		\begin{align*}
		\frac{\rmd}{\rmd t}\rmd\ChHeat_tf_n = \rmd\frac{\rmd}{\rmd t}\ChHeat_t f_n = \rmd\Delta\ChHeat_tf_n = -\Hodge\rmd \ChHeat_t f_n.
		\end{align*}
		Therefore, $\HHeat_t\rmd f_n = \rmd \ChHeat_t f_n$ for every $n\in\N$ by uniqueness.
		
		The claim follows by letting $n\to\infty$, using the boundedness of $\HHeat_t$ as well as the fact that $(\ChHeat_tf_n)_{n\in\N}$ converges to $\ChHeat_tf$ in $W^{1,2}(\mms)$. In particular, $\rmd \ChHeat_tf\in\Dom(\Hodge)$.
	\end{proof}
	
	\begin{remark}\label{Re:Commutation higher degree} Given any $k\in\N$ with $k\geq 2$, one can define a ``canonical'' heat flow $\smash{(\HHeat_t^k)_{t\geq 0}}$ acting on $\smash{\Ell^2(\Lambda^kT^*\mms)}$ as the semigroup corresponding to the negative Hodge Laplacian on  $k$-forms. However, we do not know if for every $\omega\in H^{1,2}(\Lambda^{k-1}T^*\mms)$ and every $t\geq 0$, we have
		\begin{align*}
		\HHeat_t^{k}\rmd\omega = \rmd\HHeat_t^{k-1}\omega.
		\end{align*}
		
		The subtle problem arising when mimicking the proof of \autoref{Le:Ht vs Pt} is that, to verify that $\smash{\rmd\HHeat_t^{k-1}\omega\in\Dom(\Hodge^k)}$, by definition of $\smash{\Hodge^k}$ one \emph{a priori} has to know that $\smash{\rmd\HHeat_t^{k-1}\omega \in H^{1,2}(\Lambda^kT^*\mms)}$, i.e.~to have an analogue of \autoref{Le:Differential regularity} for forms of higher degree, which is unclear to us.
	\end{remark}
	
	\begin{corollary}\label{Cor:delta H_t P_t delta} If $\omega\in\Dom(\delta)$ and $t\geq 0$, then $\HHeat_t\omega\in \Dom(\delta)$ with
		\begin{align*}
		\delta\HHeat_t\omega = \ChHeat_t\delta\omega.
		\end{align*}
	\end{corollary}
	
	\begin{theorem}\label{Th:L^2-contractivity} The following properties of $(\HHeat_t)_{t\geq 0}$ hold for every $\omega\in\Ell^2(T^*\mms)$ and every $t>0$.
		\begin{enumerate}[label=\textnormal{(\roman*)}]
			\item \textnormal{\textsc{Self-adjointness.}} The operator $\HHeat_t$ is self-adjoint in $\Ell^2(T^*\mms)$.
			\item \textnormal{\textsc{Kolmogorov backward equation.}} If $\omega\in\Dom(\Hodge)$, we have
			\begin{align*}
			\frac{\rmd}{\rmd t}\HHeat_t\omega = -\HHeat_t\Hodge\omega.
			\end{align*}
			In particular, the identity $\Hodge\HHeat_t = \HHeat_t\Hodge$ holds on $\Dom(\Hodge)$.
			\item \textnormal{$\Ell^2$\textsc{-contractivity.}}\label{La:Contractivity} $
			\Vert \HHeat_t\omega\Vert_{\Ell^2} \leq \Vert \HHeat_s\omega\Vert_{\Ell^2}$ for every $s\in [0,t]$.
			\item \textnormal{\textsc{Energy dissipation.}} The function $t\mapsto\calE(\HHeat_t\omega)$ belongs to the class $\Cont^1((0,\infty))$. Its derivative satisfies
			\begin{align*}
			\frac{\rmd}{\rmd t}\calE(\HHeat_t\omega) = -\int_\mms \vert\Hodge\HHeat_t\omega\vert^2\d\meas.
			\end{align*}
			In particular, $\calE(\HHeat_t\omega)\leq\calE(\HHeat_s\omega)$ for every $s\in [0,t]$.
			\item \textnormal{$H^{1,2}$\textsc{-continuity.}} If $\omega\in H^{1,2}(T^*\mms)$, the map $t\mapsto \HHeat_t\omega$ is continuous on $[0,\infty)$ w.r.t.~strong convergence in $H^{1,2}(T^*\mms)$.
			\item\label{La:A priori} \textnormal{\textsc{A priori estimates.}} We have
			\begin{align*}
			\calE(\HHeat_t\omega) &\leq \frac{1}{4t}\,\Vert \omega\Vert_{\Ell^2}^2,\\
			\Vert \Hodge\HHeat_t\omega\Vert_{\Ell^2}^2 &\leq \frac{1}{2t^2}\,\Vert \omega\Vert_{\Ell^2}^2.
			\end{align*}
			\item\label{La:Bakry-Emery L^2} \textnormal{\textsc{Contravariant Bakry--Émery inequality.}} We have
			\begin{align*}
			\vert\HHeat_t\omega\vert^2\leq\rme^{-2Kt}\,\ChHeat_t\big(\vert\omega\vert^2\big)\quad\meas\text{-a.e.}
			\end{align*}
		\end{enumerate}
	\end{theorem}
	
	\subsection{Kato's inequality}
	
	A further key to obtain vector $1$-Bochner inequalities and to derive various functional inequalities for vector fields and $1$-forms is \autoref{Pr:Kato's inequality} below. In the smooth framework, this estimate is known as \emph{Kato's inequality}. The elementary proof of its nonsmooth counterpart can be found in \cite[Lem.~2.5]{debin2020}. It also easily provides us with the chain rule from \autoref{Re:Chain rule}.
	
	Besides the rest of the current \autoref{Ch:Heat flow pointwise props}, further applications of \autoref{Pr:Kato's inequality} are treated in \autoref{Th:Log Sobolev for RCD} and \autoref{Cor:Spectral gap inclusion}. 
	
	\begin{proposition}[Kato inequality]\label{Pr:Kato's inequality} For every $X\in H^{1,2}(T\mms)$, we have $\vert X\vert\in \Sob^{1,2}(\mms)$ and
		\begin{align*}
		\vert\nabla\vert X\vert\vert\leq\vert\nabla X\vert\quad\meas\text{-a.e.}
		\end{align*}
	\end{proposition}
	
	Recalling that $H^{1,2}(T^*\mms)^\sharp \subset H^{1,2}(T\mms)$ by virtue of \eqref{Eq:H embedding}, \autoref{Pr:Kato's inequality} yields in particular that for every $\omega\in H^{1,2}(T^*\mms)$, we have $\vert\omega\vert\in \Sob^{1,2}(\mms)$ and
	\begin{align*}
	\vert\rmd \vert\omega\vert\vert \leq \vert \nabla\omega^\sharp\vert\quad\meas\text{-a.e.}
	\end{align*}
	
	\begin{corollary}\label{Re:Chain rule} For every $X\in H^{1,2}(T\mms)\cap\Ell^\infty(T\mms)$, we have $\vert X\vert^2\in W^{1,2}(\mms)$ with
		\begin{align*}
		\nabla\vert X\vert^2 = 2\,\vert X\vert\,\nabla\vert X\vert.
		\end{align*}
	\end{corollary}
	
	\begin{remark} One can drop the assumption that $X\in\Ell^\infty(T^*\mms)$ in \autoref{Re:Chain rule}, still retaining the stated identity for $\nabla \vert X\vert^2$. In this case, one has to understand $\nabla\vert X\vert^2$ as $\Ell^1$-covariant derivative of the $\Ell^1$-function $\vert X\vert^2$ belonging to the Sobolev space $H^{1,1}(T\mms)$ \cite[Subsec.~3.3.3, Prop.~3.4.6]{gigli2018}.
	\end{remark}
	
	\subsection{Vector 1-Bochner inequality}\label{Sec:Prop}
	
	A first important consequence of \autoref{Pr:Kato's inequality} is a $1$-Bochner inequality for vector fields, see \autoref{Th:1-Bochner vector}, in the dimension-free case. The results proven on the go will also yield the very important $\Ell^1$-type estimate between $(\HHeat_t)_{t\geq 0}$ and $(\ChHeat_t)_{t\geq 0}$ in \autoref{Th:Kato-Simon}. Similarly to the level of functions \cite[Lem.~2.6, Cor.~4.3]{savare2014}, the key point is to verify that $\vert X\vert\in \Dom(\DELTA)$ for a sufficiently large class of vector fields $X$. \autoref{Th:1-Bochner vector} is then essentially a consequence of the chain rule for the measure-valued Laplacian $\DELTA$ from \autoref{Le:Laplacian functions chain rule}.
	
	In this section, we state our results commonly for the case when $(\mms,\met,\meas)$ is an $\RCD(K,\infty)$ space. 
	
	A crucial estimate is established in the subsequent lemma.
	
	\begin{lemma}\label{Le:The key} Let $\smash{X\in H^{1,2}(T^*\mms)^\sharp\cap\Ell^\infty(T\mms)}$ satisfy $X^\flat\in\Dom(\Hodge)$, and let $\psi\in W^{1,2}(\mms)\cap\Ell^\infty(\mms)$ be nonnegative. Then
		\begin{align*}
		K\int_\mms \psi\,\vert X\vert^2\d\meas 
		&\leq -\frac{1}{2}\int_\mms \langle\nabla \psi, \nabla \vert X\vert^2 \rangle\d\meas +\int_\mms \psi\,\big\langle X, (\Hodge X^\flat)^\sharp\big\rangle\d\meas\\
		&\qquad\qquad\qquad\qquad - \int_\mms \psi\,\vert\nabla X\vert^2\d\meas.
		\end{align*}
	\end{lemma}
	
	\begin{proof} Let $(\psi_n)_{n\in\N}$ be a sequence of nonnegative functions in $\Test(\mms)$ converging to $\psi$ in $W^{1,2}(\mms)$ according to \autoref{Le:Mollified heat flow}. Moreover, let $(X_i)_{i\in\N}$ be a sequence in $ \Test(T\mms)$ converging to $X$ in $\smash{H^{1,2}(T^*\mms)^\sharp}$. For every $n,i\in\N$, by the definition of $\DELTA$ as well as \autoref{Th:Ricci tensor},
		\begin{align*}
		K\int_\mms \psi_n\,\vert X_i\vert^2\d\meas
		&\leq -\frac{1}{2}\int_\mms \langle\nabla \psi_n, \nabla \vert X_i\vert^2\rangle\d\meas +\int_\mms \psi_n\,\big\langle X_i, (\Hodge X_i^\flat)^\sharp\big\rangle\d\meas\\
		&\qquad\qquad - \int_\mms \psi_n\,\vert\nabla X_i\vert^2\d\meas.
		\end{align*}
		
		For every $n\in\N$, the convergence of the first, second,  and fourth term towards the desired quantities as $i\to \infty$ is clear by boundedness of $\psi_n$ and $\nabla\psi_n$, \autoref{Re:Chain rule} and \eqref{Eq:H embedding}. Moreover, integration by parts, \autoref{Le:Multiplication with Sobolev functions} and the facts that $\psi_n\in\Ell^\infty(\mms)$ and $\rmd \psi_n\in\Ell^\infty(T^*\mms)$ also entail the appropriate convergence of the third term as $i\to\infty$.
		
		The claim follows by letting $n\to\infty$ in the resulting inequality.
	\end{proof}
	
	\begin{proposition}\label{Le:Key lemma} Let $\smash{X\in H^{1,2}(T^*\mms)^\sharp\cap\Ell^\infty(T\mms)}$ satisfy $X^\flat\in\Dom(\Hodge)$, and let $\phi\in \Dom(\Delta)\cap\Ell^\infty(\mms)$ be nonnegative with $\Delta\phi\in\Ell^\infty(\mms)$. Then for every $\varepsilon > 0$, we have
		\begin{align*}
		&\int_\mms \Delta\phi\, \big[(\vert X\vert^2 + \varepsilon)^{1/2} - \varepsilon^{1/2}\big] \d\meas + \int_\mms \phi\,(\vert X\vert^2+\varepsilon)^{-1/2}\,\big\langle X,(\Hodge X^\flat)^\sharp\big\rangle\d\meas\\
		&\qquad\qquad\qquad\qquad - \int_\mms \phi\,(\vert X\vert^2+\varepsilon)^{-1/2} \,\vert\nabla X\vert^2\d\meas\\
		&\qquad\qquad\qquad\qquad\qquad\qquad  + \int_\mms \phi\,(\vert X\vert^2+\varepsilon)^{-1/2}\,\vert \nabla \vert X\vert\vert^2\d\meas\\
		&\qquad\qquad\geq K\int_\mms \phi\,(\vert X\vert^2+\varepsilon)^{-1/2}\,\vert X\vert^2\d\meas.
		\end{align*}
	\end{proposition}
	
	\begin{proof} Given any $\varepsilon > 0$, note that the function $\Phi_\varepsilon\in \Cont^\infty([0,\infty))$ defined by
		\begin{align*}
		\Phi_\varepsilon(r) := 2(r+\varepsilon)^{1/2} - 2\varepsilon^{1/2}
		\end{align*}
		obeys the inequality
		\begin{align}\label{Eq:Phi eps bound}
		-2\,\Phi_\varepsilon''(r)\,r \leq \Phi_\varepsilon'(r).
		\end{align}
		
		The stated inequality then follows by applying \autoref{Le:The key} to $\psi := \phi\,\Phi_\varepsilon'(\vert X\vert^2)$, which belongs to $W^{1,2}(\mms)\cap\Ell^\infty(\mms)$ by \autoref{Re:Chain rule} and the Leibniz rule, noting that by \eqref{Eq:Phi eps bound},
		\begin{align*}
		&-\frac{1}{2}\int_\mms \langle\nabla \psi,\nabla\vert X\vert^2\rangle\\
		&\qquad\qquad = - \frac{1}{2}\int_\mms \Phi_\varepsilon'(\vert X\vert^2)\,\langle\nabla\phi,\nabla\vert X\vert^2\rangle\d\meas\\
		&\qquad\qquad\qquad\qquad  -\frac{1}{2}\int_\mms\phi\,\langle\nabla \Phi_\varepsilon'(\vert X\vert^2),\nabla\vert X\vert^2\rangle\d\meas\\
		&\qquad\qquad = -\frac{1}{2}\int_\mms \langle\nabla\phi,\nabla\Phi_\varepsilon(\vert X\vert^2)\rangle\d\meas - 2\int_\mms\phi\,\Phi_\varepsilon''(\vert X\vert^2)\,\vert X\vert^2\,\vert\nabla \vert X\vert\vert^2\d\meas\\
		&\qquad\qquad \leq \frac{1}{2}\int_\mms \Phi_\varepsilon(\vert X\vert^2)\,\Delta\phi\d\meas + \int_\mms \phi\,\Phi_\varepsilon'(\vert X\vert^2)\,\vert \nabla\vert X\vert\vert^2\d\meas.\qedhere
		\end{align*}
	\end{proof}

	\begin{theorem}[Vector 1-Bochner inequality]\label{Th:1-Bochner vector} Let $(\mms,\met,\meas)$ be an $\RCD(K,\infty)$ space for $K\in\R$. Let $X\in\Test(T\mms)$. Then $\vert X\vert\in\Dom(\DELTA)$ and, after redefining $\smash{\vert X\vert^{-1}\,\big\langle X,(\Hodge X^\flat)^\sharp\big\rangle := 0}$ on $\vert X\vert^{-1}(\{0\})$,
		\begin{align*}
		\DELTA\vert X\vert +\vert X\vert^{-1}\,\big\langle X,(\Hodge X^\flat)^\sharp\big\rangle\,\meas \geq K\,\vert X\vert\,\meas.
		\end{align*}
	\end{theorem}
	
	\begin{proof} 
	We already know that $\vert X\vert\in W^{1,2}(\mms)$ under the given assumptions. 
	Furthermore, given any nonnegative $\phi\in \Dom(\Delta)\cap\Ell^\infty(\mms)$ with $\Delta\phi\in\Ell^\infty(\mms)$, via \autoref{Le:Key lemma}, 
		letting $\varepsilon \to 0$ with Lebesgue's theorem,  we arrive at
		\begin{align*}
		\int_\mms\Delta\phi\, \vert X\vert\d\meas &\geq \int_{\mms}\phi\, \big[\!-\!\vert X\vert^{-1}\,\big\langle X,(\Hodge X^\flat)^\sharp\big\rangle  + K\,\vert X\vert\big]\d\meas.
		\end{align*}
		Hence, it follows from \cite[Lem.~2.6]{savare2014} that $\vert X\vert\in\Dom(\DELTA)$ with
		\begin{align*}
		\DELTA\vert X\vert + \vert X\vert^{-1}\,\big\langle X,(\Hodge X^\flat)^\sharp\big\rangle\,\meas \geq K\,\vert X\vert \,\meas.\tag*{\qedhere}
		\end{align*}
	\end{proof}
	
	\begin{remark}\label{Re:X vec fields} 
	\autoref{Th:1-Bochner vector} is \emph{a priori} true for $\smash{X\in H^{1,2}(T^*\mms)^\sharp} \cap \Ell^1(T\mms) \cap\Ell^\infty(T\mms)$ with $\smash{X^\flat\in\Dom(\Hodge)}$ and $\Hodge X^\flat\in\Ell^1(T^*\mms)$. We restricted ourselves to the assumption that $X\in\Test(T\mms)$ to simplify the presentation. 
		
		It is worth to note that \emph{a posteriori}, by \autoref{Th:Kato-Simon} and \autoref{Cor:Extension} below, the assumption of \autoref{Th:1-Bochner vector} is satisfied for every vector field $\smash{X := (\HHeat_t\omega)^\sharp}$ with $t>0$ and $\omega\in \Ell^1(T^*\mms)\cap\Ell^\infty(T^*\mms)$. See also \autoref{Sub:Log Sob}.
	\end{remark}

	\subsection{Hess--Schrader--Uhlenbrock's inequality}\label{Sec:Kato-Simon}
	
	The ``almost'' vector $1$-Bochner inequality from \autoref{Le:Key lemma} is also an indispensable tool for \autoref{Th:Kato-Simon}. It is an instance of \emph{form domination} in smooth situations  \cite{hess1977, simon1977}, which also implicitly uses some sort of ``integrated $1$-Bochner inequality''.
	
	A slight restatement of \autoref{Le:Key lemma} and \autoref{Th:1-Bochner vector} is \autoref{Le:Form dom} below. \autoref{Th:Kato-Simon} then directly follows from the very general form domination result \cite[Thm.~2.15]{hess1977}. See  also \cite{simon1977, ouhabaz1999} for comprehensive proofs of the previous implication. 
	
	\begin{lemma}\label{Le:Form dom} For every $\omega\in \Dom(\Hodge)$ and every nonnegative $\phi\in W^{1,2}(\mms)$, 
		\begin{align}\label{Eq:Inequ form}
		\int_\mms \langle\nabla\phi,\nabla\vert \omega\vert\rangle \d\meas + K\int_\mms\phi\,\vert \omega\vert\d\meas \leq \int_\mms \phi\,\vert \omega\vert^{-1}\,\langle \omega,\Hodge \omega\rangle\d\meas.
		\end{align}
	\end{lemma}
	
	\begin{proof} If $\omega$ also belongs to $\Ell^\infty(T^*\mms)$, the claimed inequality originates from \autoref{Le:Key lemma} by letting $\varepsilon \to 0$, dropping the contributions by $\smash{\vert\nabla \omega^\sharp\vert^2}$ and $\smash{\vert\nabla\vert \omega\vert\vert^2}$ via \autoref{Pr:Kato's inequality}, and integrating by parts --- the resulting inequality easily extends to the asserted class of $\phi$ by \autoref{Le:Mollified heat flow}.
		
		Observe that given any $\omega\in\Dom(\Hodge)$, there exists a sequence $(\omega_n)_{n\in\N}$ in $\Dom(\Hodge)\cap\Ell^\infty(T^*\mms)$ converging to $\omega$ in $H^{1,2}(T^*\mms)$ such that $\Hodge \omega_n \to\Hodge \omega$ in $\Ell^2(T^*\mms)$ as $n\to\infty$ --- this then already provides \eqref{Eq:Inequ form} by approximation, possibly after taking pointwise $\meas$-a.e.~converging subsequences. Indeed, given any $z\in\mms$ and $R>0$, set $\omega_R := \One_{B_R(z)}\,\One_{[0,R]}(\vert\omega\vert)\,\omega$. By \autoref{Th:L^2-contractivity}, for every $t>0$ we have $\HHeat_t\omega_R \to \HHeat_t\omega$ in $H^{1,2}(T^*\mms)$ and $\Hodge \HHeat_t\omega_R\to\Hodge\HHeat_t\omega$ in $\Ell^2(T^*\mms)$ as $R\to\infty$. Since also $\HHeat_t\omega \to\omega$ in $H^{1,2}(T^*\mms)$ as well as $\Hodge\HHeat_t\omega = \HHeat_t\Hodge \omega \to \Hodge\omega$ as $t\to 0$ thanks to \autoref{Th:L^2-contractivity} as well, the claimed existence of an approximation sequence follows by a diagonal argument.
	\end{proof}
	
	\begin{theorem}[Hess--Schrader--Uhlenbrock inequality]\label{Th:Kato-Simon} Suppose that $(\mms,\met,\meas)$ is an $\RCD(K,\infty)$ space for some $K\in\R$. Then for every $\omega\in\Ell^2(T^*\mms)$ and every $t\geq 0$, we have
		\begin{align*}
		\vert\HHeat_t\omega\vert \leq \rme^{-Kt}\,\ChHeat_t\vert\omega\vert\quad\meas\text{-a.e.}
		\end{align*}
	\end{theorem}
	
	\begin{remark}\label{Re:Related Kato Simon} For Riemannian manifolds, \autoref{Th:Kato-Simon} is due to \cite[Ch.~3]{hess1980}. 
		
		A similar approach to \autoref{Th:Kato-Simon}, by encoding uniform lower Ricci bounds via Bakry--Émery calculus, comes from the study of form domination for Hilbert space valued functions \cite{shigekawa1997}. The analogy is created by the structural characterization of $\Ell^2(T^*\mms)$ from \cite[Thm.~1.4.11]{gigli2018}.
		
		The setting of \cite{shigekawa1997} is, however, more restrictive than ours. For instance, the assumptions ($\Gamma$) and ($\bdGamma_\lambda$) made in \cite[Ch.~3]{shigekawa1997}, see (3.1) and (3.7) therein, transferred to our notation require that $\langle \omega,\eta\rangle\in\Dom(\Delta)$ for a sufficiently large class of $\omega,\eta\in\Dom(\Hodge)$, an assumption we can make sense of only in a weak form using the measure-valued Laplacian $\DELTA$. (In \cite{shigekawa1997}, the previous regularity assumption has been used to define a contravariant $\Gamma$-operator.)
	\end{remark}
	
	\subsection{Bakry--Ledoux's inequality}
	
	We turn to a version of \ref{La:Bakry-Emery L^2} in \autoref{Th:L^2-contractivity} under synthetic upper dimension bounds, namely the  \emph{Bakry--Ledoux inequality}. 
	Our strategy for \autoref{Th:Bakry-Ledoux} closely follows its functional counterparts for \cite[Thm.~3.6, Prop.~3.7]{erbar2020} and in fact both uses and entails their equivalence with the $\RCD^*(K,N)$ condition on $\RCD(K,\infty)$ spaces (or, with the same arguments, on $\RCD(K',\infty)$ spaces with $K' < K$).

	\begin{theorem}[Bakry--Ledoux inequality]\label{Th:Bakry-Ledoux} Let $(\mms,\met,\meas)$ be an $\RCD(K,\infty)$ space for some $K\in\R$, and let $N\in (1,\infty)$. Then $(\mms,\met,\meas)$ is an $\RCD^*(K,N)$ space if and only if any of the subsequent inequalities holds for every $\omega\in\Ell^2(T^*\mms)$.
		\begin{enumerate}[label=\textnormal{(\roman*)}]
			\item\label{La:Strong BL} \textnormal{\textsc{Strong Bakry--Ledoux inequality.}}  For every $t>0$,
			\begin{align*}
			\vert\HHeat_t\omega\vert^2 + \frac{2}{N}\int_0^t\rme^{-2Ks}\,\ChHeat_s\big(\vert\delta\HHeat_{t-s}\omega\vert^2\big)\d s \leq \rme^{-2Kt}\,\ChHeat_t\big(\vert\omega\vert^2\big)\quad\meas\text{-a.e.}
			\end{align*}
			\item\label{La:Weak BL} \textnormal{\textsc{Integral weak Bakry--Ledoux inequality.}}  For every $t>0$,
			\begin{align*}
			\vert\HHeat_t\omega\vert^2 + \frac{2}{N}\int_0^t\rme^{-2Ks}\,\vert\ChHeat_s \delta\HHeat_{t-s}\omega\vert^2\d s \leq \rme^{-2Kt}\,\ChHeat_t\big(\vert\omega\vert^2\big)\quad\meas\text{-a.e.}
			\end{align*}
			\item\label{La:RCD EKS} \textnormal{\textsc{Non-integral weak Bakry--Ledoux inequality.}} For every $t>0$, with the interpretation of the prefactor in the second term as $2t/N$ in the case $K=0$,
			\begin{align*}
			\vert\HHeat_t\omega\vert^2 + \frac{4Kt^2}{N(\rme^{2Kr}-1)}\,\vert\delta\HHeat_t\omega\vert^2\leq \rme^{-2Kt}\,\ChHeat_t\big(\vert\omega\vert^2\big)\quad\meas\text{-a.e.}
			\end{align*}
		\end{enumerate}
	\end{theorem}
	
	\begin{proof} It suffices to prove that the $\RCD^*(K,N)$ condition implies \ref{La:Strong BL}. Indeed, \ref{La:Strong BL} easily implies \ref{La:Weak BL} by Jensen's inequality. If \ref{La:Weak BL} holds, then also \ref{La:RCD EKS} is satisfied, since then, by Jensen's inequality, \autoref{Cor:delta H_t P_t delta} and with the indicated interpretation in the case $K=0$, 
		\begin{align*}
		&\frac{4Kt^2}{N(\rme^{2Kr}-1)}\,\vert\delta\HHeat_t\omega\vert^2 = \frac{2t}{N}\,\vert\delta\HHeat_t\omega\vert^2\,\Big[\frac{1}{t}\int_0^t \rme^{2Ks}\d s\Big]^{-1}\\
		&\qquad\qquad\leq \frac{2}{N}\int_0^t \rme^{-2Ks}\,\vert\delta\HHeat_{t}\omega\vert^2\d s = \frac{2}{N}\int_0^t \rme^{-2Ks}\,\vert\ChHeat_s\delta\HHeat_{t-s}\omega\vert^2\d s
		\end{align*}
		Moreover, under the given topological assumptions, it is known by \cite[Thm.~4.19]{erbar2015} that \ref{La:RCD EKS}, restricted to exact differential $1$-forms, implies the $\RCD^*(K,N)$ condition on the metric measure space $(\mms,\met,\meas)$.
		
		For $R>0$, let $\omega_R\in \Ell^2(T^*\mms)\cap\Ell^\infty(T^*\mms)$ and  $F\colon [0,t]\to\R$ be defined by
		\begin{align*}
		\omega_R &:= \One_{[0,R]}(\vert\omega\vert)\,\omega,\\
		F(s) &:= \rme^{-2Ks}\int_\mms\ChHeat_s\phi\,\vert\HHeat_{t-s}\omega_R\vert^2\d \meas,
		\end{align*}
		where $\phi\in\Test(\mms)$ is nonnegative. 		By the respective continuity of $(\ChHeat_t)_{t\geq 0}$ and $(\HHeat_t)_{t\geq 0}$ in $\Ell^2(\mms)$ and $\Ell^2(T^*\mms)$ and the boundedness of $\phi$ and $\omega_R$,  $F$ is continuous. Its restriction to $(0,t]$ has $\Cont^1$-regularity, which is also inherited from the corresponding properties of $(\ChHeat_t)_{t\geq 0}$ and $(\HHeat_t)_{t\geq 0}$ as well as item \ref{La:Bakry-Emery L^2} in \autoref{Th:L^2-contractivity}.
		
		Therefore, employing \cite[Thm.~4.3]{han2018}, the claim follows by integration and the arbitrariness of $\phi$  after observing that
		\begin{align*}
		F'(s) &= -2K\rme^{-2Ks}\int_\mms\ChHeat_s\phi\,\vert\HHeat_{t-s}\omega_R\vert^2\d\meas +\rme^{-2Ks}\int_\mms \Delta\ChHeat_s\phi\,\vert\HHeat_{t-s}\omega_R\vert^2\d\meas\\
		&\qquad\qquad + 2\rme^{-2Ks}\int_\mms \ChHeat_s\phi\,\langle\HHeat_{t-s}\omega_R,\Hodge\HHeat_{t-s}\omega_R\rangle\d\meas\\
		&\geq \frac{2}{N}\,\rme^{-2Ks}\int_\mms \ChHeat_s\phi\,\vert\delta\HHeat_{t-s}\omega_R\vert^2\d\meas.\qedhere
		\end{align*}
	\end{proof}

	\begin{remark}[Generalization to variable Ricci bounds]\label{Re:Variable} To simplify the presentation and since most later results do not require refinements of the previous facts, we reduced ourselves to the case of constant lower bounds $K$ for the Ricci curvature in \autoref{Ch:Heat flow pointwise props}.
		
		However, the arguments in \autoref{Ch:Heat flow pointwise props} perfectly work in a similar manner if the $\RCD(K,\infty)$ space $(\mms,\met,\meas)$ obeys the stronger $2$-Bakry--Émery inequalities $\BE_2(k,\infty)$ or $\BE_2(k,N)$ with variable curvature bound $k$ in the sense of \cite[Def.~1.4]{braun2019} or \cite[Def.~3.3]{sturm2019}, respectively. Here, $k\colon\mms\to\R$ is a locally $\meas$-integrable function with $k\geq K$ on $\mms$. (The assumption on lower semicontinuity on $k$ in \cite{braun2019,sturm2019} is not needed in the purely Eulerian perspective of \cite{gigli2018}.) In particular, \autoref{Th:1-Bochner vector} remains true by just replacing $K$ by $k$. In addition, \autoref{Th:Kato-Simon} holds when replacing $\smash{\rme^{-qKt}\,\ChHeat_t}$, $q\in \{1,2\}$, by the operator $\smash{\Schr{qk}_t}$, where $\smash{(\Schr{qk}_t)_{t\geq 0}}$ is the \emph{Schrödinger semigroup} on $\Ell^2(\mms)$ with generator $\Delta - qk$ \cite{stollmann1996}. One way to read \autoref{Th:Kato-Simon} in terms of \emph{Brownian motion} $\smash{\big((\boldsymbol{\mathrm{P}}^x)_{x\in\mms}, (B_t)_{t\geq 0}\big)}$ on $(\mms,\met,\meas)$ --- defined w.r.t.~the generator $\Delta/2$ \cite{ambrosio2014b} --- is thus
		\begin{align*}
		\vert \HHeat_t\omega\vert \leq \boldsymbol{\mathrm{E}}^\cdot\big[\rme^{-\int_0^{2t} k(B_r)/2\d r}\,\vert\omega\vert(B_{2t})\big]\quad\meas\text{-a.e.}\tag*{\qedhere}
		\end{align*}
	\end{remark}
	
	\section{Integral estimates for the heat flow}\label{Ch:Integral estimates}
	
	\subsection{Basic $\Ell^p$-properties and $\Ell^p$-$\Ell^\infty$-regularization}\label{Sec:Basic props Lp}
	
	From \autoref{Th:Kato-Simon} and a standard procedure, the following is immediate by approximation. It is worth to  emphasize that the restriction of $\HHeat_t$ to $\Ell^\infty(T^*\mms)$ is defined as the Banach space adjoint of the restriction of $\HHeat_t$ to $\Ell^1(T^*\mms)$ for every $t\geq 0$. See also \autoref{Sec:Independence spectrum}.

	\begin{theorem}\label{Cor:Extension} Let $(\mms,\met,\meas)$ be an $\RCD(K,\infty)$ space for some $K\in\R$. For every $p\in [1,\infty]$, $(\HHeat_t)_{t\geq 0}$ then extends to a semigroup of bounded linear operators from $\Ell^p(T^*\mms)$ into $\Ell^p(T^*\mms)$, strongly continuous if $p<\infty$, which satisfies
		\begin{align*}
		\Vert \HHeat_t\omega \Vert_{\Ell^p} \leq \rme^{-Kt}\,\Vert \omega\Vert_{\Ell^p}
		\end{align*}
		for every $\omega\in\Ell^p(T^*\mms)$ and every $t\geq 0$.
	\end{theorem}
	
	\begin{remark} 	This result implies the $\Ell^p$-contractivity of $(\HHeat_t)_{t\geq 0}$ for every $p\in[1,\infty]$ under nonnegative lower Ricci bounds. However, we should not expect contractivity in larger generality, not even on Riemannian manifolds \cite{strichartz1983, strichartz1986}.
	\end{remark}
	
	On compact $\RCD^*(K,N)$ spaces, $N\in (1,\infty)$, the heat operator $\HHeat_t$ is not only bounded from $\Ell^p(T^*\mms)$ to $\Ell^p(T^*\mms)$, but also from $\Ell^p(T^*\mms)$ to $\Ell^\infty(T^*\mms)$ for every $t>0$ and every $p\in [1,\infty]$. This is the content of the following result which will be crucial in \autoref{Sec:Spectral properties compact} and \autoref{Sub:Fund sol}.
	
	\begin{theorem}[$\Ell^p$-$\Ell^\infty$-regularization]\label{Th:Lp Linfty} Let $(\mms,\met,\meas)$ be a compact $\RCD^*(K,N)$ space, $K\in\R$ and $N\in (1,\infty)$. Furthermore, let $t>0$ and $p\in [1,\infty]$. Then $\HHeat_t$ is bounded from $\Ell^p(T^*\mms)$ to $\Ell^\infty(T^*\mms)$.
	\end{theorem}
	
	\begin{proof} By Hölder's inequality, it suffices to prove boundedness of $\HHeat_t$ from $\Ell^1(T^*\mms)$ to $\Ell^\infty(T^*\mms)$.
		
		Let $\omega\in\Ell^1(T^*\mms)\cap\Ell^2(T^*\mms)$ with $\Vert\omega\Vert_{\Ell^1}\leq 1$ be arbitrary --- the consideration of such $1$-forms is enough by the density of $\Ell^1(T^*\mms)\cap\Ell^2(T^*\mms)$ in $\Ell^1(T^*\mms)$. By \eqref{Eq:AHPT}, there exists a constant $C>0$ such that $\smash{\meas\big[B_{\sqrt{t}}(\cdot)\big]^{-1}} \leq \smash{C\,\meas\big[B_{\sqrt{t}}(z)\big]^{-1}}$ on $\mms$. The conclusion follows by observing that by \autoref{Th:Kato-Simon} and \eqref{Eq:Heat kernel bound}, there exist constants $C_3,C_4 > 1$ depending only on $K$ and $N$ such that
		\begin{align*}
		\vert \HHeat_t\omega\vert &\leq \rme^{-Kt}\,\ChHeat_t\vert\omega\vert = \rme^{-Kt}\int_\mms \sfp_t(\cdot,y)\,\vert \omega\vert(y)\d\meas(y)\\
		&\leq C_3\,\rme^{-(K-C_4)t}\,\meas\big[B_{\sqrt{t}}(\cdot)\big]^{-1}\textcolor{white}{\int_M}\\
		&\leq C\,C_3\,\rme^{-(K-C_4)t}\,\meas\big[B_{\sqrt{t}}(z)\big]^{-1}\quad\meas\text{-a.e.}\textcolor{white}{\int_M}\qedhere
		\end{align*}
	\end{proof}

	\subsection{Logarithmic Sobolev inequalities}\label{Sec:Hypercontractivity}
	
	We come to an important class of functional inequalities, namely \emph{logarithmic Sobolev inequalities} for $1$-forms and their relation to integral-type inequalities for $(\HHeat_t)_{t\geq 0}$. More precisely, following \cite{charalambous2007,davies1989} we show that the former imply, for certain $t>0$ and every $p_0 \in (1,\infty)$, the boundedness of $\HHeat_t$ from $\Ell^{p_0}(T^*\mms)$ into $\Ell^{p(t)}(T^*\mms)$, where $p$ is a real-valued function with $p(0) = p_0$. This property is called \emph{hypercontractivity}. Under more restrictive assumptions, for some finite $T>0$ it is even possible to prove the boundedness of $\HHeat_T$ from $\Ell^{p_0}(T^*\mms)$ to $\Ell^\infty(T^*\mms)$, a property termed \emph{ultracontractivity}. See \autoref{Th:pLSI to Ultra}.
	
	A certain reverse implication also holds, see \autoref{Th:LSI 2 from hyper}.
	
	\begin{definition}\label{Def:LSI} Let $\beta > 0$ and $\chi\in\R$. We say that $X\in H^{1,2}(T\mms)\cap\Ell^1(T\mms)\cap\Ell^\infty(T\mms)$ satisfies the \emph{$2$-lo\-ga\-rithmic Sobolev inequality} with constants $\beta$ and $\chi$, briefly $\LSI_2(\beta,\chi)$, if
		\begin{align*}
		\int_\mms \vert X\vert^2\log\vert X\vert\d\meas \leq \beta\,\Vert \nabla X\Vert_{\Ell^2}^2 + \chi\,\Vert X\Vert_{\Ell^2}^2 + \Vert X\Vert_{\Ell^2}^2\log\Vert X\Vert_{\Ell^2}.
		\end{align*}
	\end{definition}
	
	\begin{definition}\label{Def:p log Sob} Let $\varepsilon > 0$, $\gamma\in\R$ and $p\in (1,\infty)$. We say that $\omega\in \Dom(\Hodge)\cap\Ell^1(T^*\mms)\cap\Ell^\infty(T^*\mms)$ obeys the \emph{form $p$-logarithmic Sobolev inequality} with constants $\varepsilon$ and $\gamma$, briefly $\fLSI_p(\varepsilon,\gamma)$, if $\Hodge\omega\in\Ell^p(T^*\mms)$ as well as, with the convention $0^0 := 0$,
		\begin{align*}
		\int_\mms \vert\omega\vert^p\log\vert\omega\vert\d\meas \leq \varepsilon\int_\mms\vert\omega\vert^{p-2}\,\langle\omega,\Hodge\omega\rangle\d\meas + \gamma\,\Vert\omega\Vert_{\Ell^p}^p + \Vert\omega\Vert_{\Ell^p}^p\log\Vert\omega\Vert_{\Ell^p}.
		\end{align*}
	\end{definition}
	
	\begin{remark}\label{Re:Contravariant} On Riemannian manifolds, a similar definition as \autoref{Def:p log Sob} has been given and considered in \cite[Def.~2.1]{charalambous2007} for the more restrictive case $p\in (2,\infty)$. The definition of a $2$-logarithmic Sobolev inequality therein, on the other hand, is similar to \autoref{Def:LSI}. 
	\end{remark}
	
	\begin{remark}\label{Re:p in (1,2)} We do not discuss the case of $1$-logarithmic Sobolev inequalities since it is not clear, even having an appropriate version of such inequality at our disposal, that for $\omega\in\Ell^1(T^*\mms)\cap\Ell^\infty(T^*\mms)$, the function $\vert\omega\vert\log\vert\omega\vert$ is integrable. 
		
		On the other hand, the integrability of $\vert\omega\vert^p\log\vert\omega\vert$ for $p\in (1,\infty)$ is clear by local boundedness of the function $r\mapsto r^\delta\log r$ on $[0,\infty)$ for every $\delta > 0$.
	\end{remark}
	
	Later, special interest will be devoted to the class
	\begin{align*}
	V_{1,\infty} := \bigcup_{t>0} \HHeat_t\big(\Ell^1(T^*\mms)\cap\Ell^\infty(T^*\mms)\big)
	\end{align*}
	By \autoref{Th:L^2-contractivity} and \autoref{Th:Kato-Simon}, $V_{1,\infty}$ is contained in $\Dom(\Hodge)$ as well as in $\Ell^p(T^*\mms)$ for every $p\in [1,\infty]$, is invariant under the action of $\HHeat_t$ for every $t>0$, and it is strongly dense in the latter space if $p<\infty$. Additionally, since the infinitesimal generator of the restriction of $(\HHeat_t)_{t\geq 0}$ onto $\Ell^p(T^*\mms)$ applied to any $\omega\in V_{1,\infty}$ coincides with $\Hodge$, we have $\Hodge\omega\in\Ell^p(T^*\mms)$ for every $p\in[1,\infty]$. (See \autoref{Sec:Basic props Lp} and \autoref{Sec:Independence spectrum} for the correct interpretation in the case $p=\infty$.)
	
	\subsubsection{Relations between different logarithmic Sobolev inequalities}\label{Sub:Log Sob}
	
	In view of \autoref{Le:2LSI to pLSI}, we first focus on the $2$-logarithmic Sobolev inequality, in particular showing how to derive it from its functional counterpart in the next \autoref{Th:Log Sobolev for RCD}. 
	
	Given any $\beta > 0$, following (1.2) of \cite{cavalletti2017} (replacing $\alpha$ by $1/\beta$ therein), a nonnegative $f\in \Lip(\mms)\cap\Ell^1(\mms)$ is said to obey the \emph{functional $2$-logarithmic Sobolev inequality} with constant $\beta$ if, with the convention $\vert\nabla f\vert^2/f := 0$ on $f^{-1}(\{0\})$,
	\begin{align}\label{Eq:LSI Functions}
	2\int_\mms f\log f\d\meas - 2\int_\mms f\d\meas\log\int_\mms f\d\meas \leq \beta\int_\mms  \frac{\vert\nabla f\vert^2}{f}\d\meas.
	\end{align}
	
	\begin{lemma}\label{Th:Log Sobolev for RCD} Let $\beta > 0$ be given. Suppose that every nonnegative $f\in\Lip(\mms)\cap\Ell^1(\mms)$ obeys the functional $2$-logarithmic Sobolev inequality with constant $\beta$. Then every $X\in H^{1,2}(T\mms)\cap\Ell^1(T\mms)\cap\Ell^\infty(T\mms)$ satisfies $\LSI_2(\beta,0)$.
	\end{lemma}
	
	\begin{proof} Let $R>1$ and $z\in\mms$, and let $\psi_R \in\Lip_\bs(\mms)$ be a cutoff function with $\psi_R(\mms) = [0,1]$, identically equal to $1$ on $B_R(z)$ and identically equal to $0$ on $\mms\setminus B_{R+1}(z)$.
		
		Given $X\in H^{1,2}(T\mms)\cap\Ell^1(T\mms)\cap\Ell^\infty(T\mms)$, observe that $\ChHeat_{1/n}\vert X\vert \in \Test(\mms) \cap \Lip(\mms)\cap \Ell^1(\mms)$ for every $n\in\N$ by \autoref{Pr:Kato's inequality}, where we identify $\ChHeat_{1/n}\vert X\vert$ with its Lipschitz $\meas$-a.e.~representative by the Sobolev-to-Lipschitz property of $(\mms,\met,\meas)$. We may and will assume that the sequence $(g_n)_{n\in\N}$ in $\Lip_\bs(\mms)\cap\Ell^1(\mms)$, where $g_n := \psi_R\,\ChHeat_{1/n}\vert X\vert$, converges to $\psi_R\,\vert X\vert$ pointwise $\meas$-a.e.~and strongly in $W^{1,2}(\mms)$. Setting $\smash{f_n := g_n^2}$ entails
		\begin{align*}
		\vert\nabla f_n\vert^2 = 4\,f_n\,\vert\nabla g_n\vert^2\quad\meas\text{-a.e.} 
		\end{align*}
		Lebesgue's theorem  as well as \eqref{Eq:LSI Functions} applied to $f_n$ for every $n\in\N$ yield
		\begin{align*}
		&2\int_\mms\psi_R^2\,\vert X\vert^2\log\!\big(\psi_R^2\,\vert X\vert^2\big)\d\meas - 2 \int_\mms\psi_R^2\,\vert X\vert^2\d\meas\log\int_\mms\psi_R^2\,\vert X\vert^2\d\meas\\
		&\qquad\qquad =\lim_{n\to\infty} \Big[2\int_\mms f_n\log f_n\d\meas - 2 \int_\mms f_n\d\meas \log\int_\mms f_n\d\meas\Big]\\
		&\qquad\qquad \leq \lim_{n\to\infty}\beta\int_\mms \frac{\vert \nabla f_n\vert^2}{f_n}\d\meas\\
		&\qquad\qquad = \lim_{n\to\infty} 4\beta\int_\mms \vert\nabla g_n\vert^2\d\meas = 4\beta\int_\mms \vert\nabla (\psi_R\,\vert X\vert)\vert^2\d\meas.
		\end{align*}
		
		The claim follows by letting $R\to\infty$, using Lebesgue's theorem and \autoref{Pr:Kato's inequality}.
	\end{proof}
	
	\begin{example}\label{Ex:alpha} By \cite[Thm.~30.21]{villani2009} if $(\mms,\met,\meas)$ is an $\RCD(K,\infty)$ space with $K>0$, or \cite[Thm.~1.9]{cavalletti2017} in the case when $(\mms,\met,\meas)$ is a compact $\RCD^*(K,N)$ space, $K\in\R$ and $N\in (1,\infty)$, the hypothesis of \autoref{Th:Log Sobolev for RCD} is known to be satisfied for some finite $\beta > 0$.
		
		The constant $\beta$ can explicitly be chosen to be $1/K$ and $(N-1)/KN$ if $K>0$, respectively.
	\end{example}
	
	\begin{proposition}\label{Le:2LSI to pLSI} Let $\beta >0$ and $\chi\in\R$. Define the functions $\varepsilon,\gamma\in \Cont((1,\infty))$ by
		\begin{align*}
		\varepsilon(p) &:= \frac{\beta p}{2(p-1)},\\
		\gamma(p) &:= \frac{2\chi}{p}- \frac{K \beta p}{2(p-1)}.
		\end{align*}
		Assume that every $X\in H^{1,2}(T^*\mms)^\sharp\cap\Ell^1(T\mms)\cap\Ell^\infty(T\mms)$ obeys $\LSI_2(\beta,\chi)$ according to \autoref{Def:LSI}. Then every $\omega\in V_{1,\infty}$ obeys $\smash{\fLSI_p\big(\varepsilon(p), \gamma(p)\big)}$ for every $p\in (1,\infty)$.
	\end{proposition}
	
	\begin{proof} The claim for $p= 2$ follows by \eqref{Eq:H embedding}. Thus we concentrate  on the case  $p\in (1,2)\cup (2,\infty)$.
		
		Given any $\tau > 0$, the function $\Phi_\tau\in \Cont^\infty([0,\infty))$ given by
		\begin{align*}
		\Phi_\tau(r) := (r+\tau)^{p/2-1}
		\end{align*}
obeys the inequalities
\begin{align}\label{Eq:Phi tau inequality}		
		0 \leq \frac{p}{p-2}\,\Phi_\tau'(r)\,r \leq \Phi_\tau(r) + \Phi_\tau'(r)\,r.
		\end{align}
		
		By \autoref{Le:Multiplication with Sobolev functions} and \autoref{Pr:Kato's inequality}, we have $\Phi_\tau(\vert\omega\vert)\in \Sobo^2(\mms)\cap\Ell^\infty(\mms)$ as well as $\Phi_\tau(\vert\omega\vert)\,\omega^\sharp\in H^{1,2}(T^*\mms)^\sharp\cap\Ell^1(T\mms)\cap\Ell^\infty(T\mms)$ for every $\tau > 0$. By $\LSI_2(\beta,\chi)$ applied to $\Phi_\tau(\vert\omega\vert)\,\omega^\sharp$ and letting $\tau\downarrow 0$, employing Lebesgue's theorem and \eqref{Eq:H embedding}, we infer that
		\begin{align}\label{Eq:pLSI proof}
		&\int_\mms \vert\omega\vert^p\log\vert\omega\vert \d\meas - \frac{2\chi}{p}\,\Vert \omega\Vert_{\Ell^p}^p - \Vert \omega\Vert_{\Ell^p}^p \log \Vert \omega\Vert_{\Ell^p}\nonumber\\
		&\qquad\qquad \leq \liminf_{\tau \downarrow 0} \Big[\frac{2}{p}\int_\mms \vert\Phi_\tau(\vert\omega\vert)\,\omega\vert^2\log\vert \Phi_\tau(\vert\omega\vert)\,\omega\vert \d\meas - \frac{2\chi}{p}\,\Vert \Phi_\tau(\vert\omega\vert)\,\omega\Vert_{\Ell^2}^2\nonumber\\
		&\qquad\qquad\qquad\qquad -\frac{2}{p}\,\Vert \Phi_\tau(\vert\omega\vert)\,\omega\Vert_{\Ell^2}^2 \log \Vert \Phi_\tau(\vert\omega\vert)\,\omega\Vert_{\Ell^2}\Big]\nonumber\\
		&\qquad\qquad \leq \liminf_{\tau\downarrow 0}\frac{2\beta}{p} \int_\mms \vert\nabla \big(\Phi_\tau(\vert\omega\vert)\,\omega^\sharp\big)\vert^2\d\meas. 
		\end{align}
		
		It remains to estimate the limit in \eqref{Eq:pLSI proof}. We start by recalling that, by definition,
		\begin{align*}
		\vert\rmd\vert\omega\vert\wedge\omega\vert^2 = \vert\rmd\vert\omega\vert\vert^2\,\vert\omega\vert^2 - \langle\rmd\vert\omega\vert,\omega\rangle^2.
		\end{align*}
		Therefore, for every $\tau > 0$, we observe by \eqref{Eq:H embedding} and \autoref{Le:Multiplication with Sobolev functions}, taking into account that $\Phi_\tau^2(\vert\omega\vert)\,\omega$ does also belong to $H^{1,2}(T^*\mms)$, and finally integration by parts that
		\begin{align}
		&\int_\mms \vert\nabla\big(\Phi_\tau(\vert\omega\vert)\,\omega^\sharp\big)\vert^2\d\meas + K\,\Vert\Phi_\tau(\vert\omega\vert)\,\omega\Vert_{\Ell^2}^2\label{Eq:Long calc I}\\
		&\qquad\qquad\leq \int_\mms \big[\vert \rmd\big(\Phi_\tau(\vert\omega\vert)\,\omega\big)\vert^2 + \vert \delta\big(\Phi_\tau(\vert\omega\vert)\,\omega\big)\vert^2\big]\d\meas\nonumber\\
		&\qquad\qquad = \int_\mms \big[\Phi_\tau^2(\vert\omega\vert)\,\vert\rmd\omega\vert^2 + 2\,\Phi_\tau(\vert\omega\vert)\,\Phi_\tau'(\vert\omega\vert)\,\langle\rmd\vert\omega\vert\wedge\omega,\rmd\omega\rangle\big]\d\meas\nonumber\\
		&\qquad \qquad\qquad\qquad + \int_\mms\big[(\Phi_\tau')^2(\vert\omega\vert)\,\vert\rmd\vert\omega\vert\wedge\omega\vert^2 + \Phi_\tau^2(\vert\omega\vert)\,\vert\delta\omega\vert^2\big]\d\meas\nonumber\\
		&\qquad\qquad\qquad\qquad -\int_\mms \big[2\,\Phi_\tau(\vert\omega\vert)\,\Phi_\tau'(\vert\omega\vert)\,\delta\omega\,\langle\rmd\vert\omega\vert,\omega\rangle\nonumber\\
		&\qquad\qquad\qquad\qquad\qquad\qquad\qquad\qquad - (\Phi_\tau')^2(\vert\omega\vert)\,\langle\rmd\vert\omega\vert,\omega\rangle^2\big]\d\meas\textcolor{white}{\int_M}\nonumber\\
		&\qquad\qquad = \int_\mms \big[\big\langle \rmd\big(\Phi_\tau^2(\vert\omega\vert)\,\omega\big),\rmd\omega\big\rangle + \delta\big(\Phi_\tau^2(\vert\omega\vert)\,\omega\big)\,\delta\omega\big]\d\meas\nonumber\\
		&\qquad\qquad\qquad\qquad + \int_\mms (\Phi_\tau')^2(\vert\omega\vert)\,\vert\omega\vert^2\,\vert\rmd\vert\omega\vert\vert^2\d\meas\nonumber\\
		&\qquad\qquad = \int_\mms \Phi_\tau^2(\vert\omega\vert)\,\langle\omega,\Hodge\omega\rangle\d\meas + \int_\mms (\Phi_\tau')^2(\vert\omega\vert)\,\vert\omega\vert^2\,\vert\rmd\vert\omega\vert\vert^2\d\meas.\label{Eq:Long calc II}
		\end{align}
		
		Thanks to \eqref{Eq:Phi tau inequality}, the Leibniz rule, the chain rule and \autoref{Pr:Kato's inequality}, we have
		\begin{align*}
		(\Phi_\tau')^2(\vert\omega\vert)\,\vert\omega\vert^2\,\vert\rmd\vert\omega\vert\vert^2  \leq \frac{(p-2)^2}{p^2}\,\vert\nabla\big(\Phi_\tau(\vert\omega\vert)\,\omega^\sharp\big)\vert^2\quad\meas\text{-a.e.}
		\end{align*}
		Rearranging the estimate resulting from this bound with the inequality between \eqref{Eq:Long calc I} and \eqref{Eq:Long calc II} and then sending $\tau\downarrow 0$ yields
		\begin{align*}
		&\liminf_{\tau\downarrow 0}\frac{2\beta}{p}\int_\mms\vert\nabla\big(\Phi_\tau(\vert\omega\vert)\,\omega^\sharp\big)\vert^2\d\meas\\
		&\qquad\qquad \leq \varepsilon(p)\int_\mms\vert\omega\vert^{p-2}\,\langle\omega,\Hodge\omega\rangle\d\meas - \frac{K\beta p}{2(p-1)}\,\Vert\omega\Vert_{\Ell^p}^p.
		\end{align*}
		
		From \eqref{Eq:pLSI proof}, this readily provides the claim.
	\end{proof}
	
	\begin{remark}\label{Re:Form LSI} Let $\beta > 0$ and $\chi\in\R$, and define $\varepsilon,\gamma\in\Cont((1,\infty))$ by
		\begin{align*}
		\varepsilon(p) &:= \frac{\beta p}{2(p-1)},\\
		\gamma(p) &:= \frac{2\chi}{p} - \frac{K\beta(p-2)^2}{2p(p-1)}.
		\end{align*} 
		
		With a slight modification of the proof of \autoref{Le:2LSI to pLSI}, it is possible to show that if every element in $\Dom(\Hodge)\cap \Ell^1(T^*\mms)\cap\Ell^\infty(T^*\mms)$ obeys $\fLSI_2(\beta,\chi)$, then every  $\omega \in V_{1,\infty}$ satisfies $\smash{\fLSI_p\big(\varepsilon(p),\gamma(p)\big)}$ for every $p\in (2,\infty)$. Up to changing the involved constants, this assumption is weaker compared to the one of \autoref{Le:2LSI to pLSI}.
	\end{remark}
	
	\subsubsection{From logarithmic Sobolev inequalitites to hyper- and ultracontractivity}
	
	\begin{theorem}\label{Th:pLSI to Ultra} Let $p_0 \in (1,\infty)$. Let $\varepsilon\in \Cont([p_0,\infty))$ be a positive function, and $\gamma \in \Cont([p_0,\infty))$. Suppose that the integrals
		\begin{align*}
		T &:= \int_{p_0}^\infty \frac{\varepsilon(r)}{r}\d r,\\		C &:= \int_{p_0}^\infty \frac{\gamma(r)}{r}\d r
		\end{align*}
		exist with values in $(0,\infty]$ and $(-\infty,\infty]$, respectively. Define $p\in \Cont^1([0,T))$ and $A\in \Cont^1([0,\infty))$ through the relations
		\begin{align*}
		\int_{p_0}^{p(t)} \frac{\varepsilon(r)}{r}\d r &:= t,\\		A(t) &:= \int_0^t \frac{\gamma(p(r))}{\varepsilon(p(r))}\d r.
		\end{align*}
		Assume $\smash{\fLSI_p\big(\varepsilon(p),\gamma(p)\big)}$ for every $\omega\in V_{1,\infty}$ and every $p\in [p_0,\infty)$. Then the following hold.
		\begin{enumerate}[label=\textnormal{(\roman*)}]
			\item\label{La:Hyper} \textnormal{\textsc{Hypercontractivity.}} For every $t\in [0,T)$, we have
			\begin{align*}
			\Vert\HHeat_t\Vert_{\Ell^{p_0},\Ell^{p(t)}} \leq \rme^{A(t)}.
			\end{align*}
			\item\label{La:Ultra} \textnormal{\textsc{Ultracontractivity.}} If $T< \infty$ and $C<\infty$, we have
			\begin{align*}
			\Vert \HHeat_{T}\Vert_{\Ell^{p_0},\Ell^\infty} \leq \rme^{C}.
			\end{align*}
		\end{enumerate}
	\end{theorem}
	
	\begin{proof} First observe that $p(0) = p_0$, that $A(0) = 0$, and that $p$ is strictly increasing with $p(t) \to \infty$ as $t\to T$. Moreover, $A(t)\to C$ as $t\to T$ thanks to the relations
		\begin{align}\label{Eq:Relations p and N i}
		\begin{split}
		p' &= \frac{p}{\varepsilon(p)},\\
		A' &= \frac{\gamma(p)}{\varepsilon(p)} = \frac{\gamma(p)\,p'}{p}.
		\end{split}
		\end{align}
		
		Owing to \ref{La:Hyper}, given any $\omega\in V_{1,\infty}\setminus \{0\}$, we assume that $\HHeat_t\omega \neq 0$ for every $t\in [0,T)$, which is always true at least for small times. Otherwise, the following computations are performed until the heat flow dies out. We consider the positive function $F\in \Cont^1([0,T))$ given by
		\begin{align*}
		F(t) := \rme^{-A(t)}\,\Vert\HHeat_t\omega\Vert_{\Ell^{p(t)}}.
		\end{align*}
		Note that the function $t\mapsto \vert \HHeat_t\omega\vert^2$ is continuously differentiable on $[0,\infty)$ in $\Ell^2(T^*\mms)$ with derivative $-2\,\langle\HHeat_t\omega,\Hodge\HHeat_t\omega\rangle\in\Ell^2(T^*\mms)$ for every $t\geq 0$ thanks to \autoref{Th:L^2-contractivity}. In particular,
		\begin{align*}
		\frac{\rmd}{\rmd t} \vert\HHeat_t\omega\vert^{p(t)} = \vert\HHeat_t\omega\vert^{p(t)}\,\big[p'(t)\log\vert\HHeat_t\omega\vert - p(t)\,\vert\HHeat_t\omega\vert^{-2}\,\langle\HHeat_t\omega,\Hodge\HHeat_t\omega\rangle\big]\quad\meas\text{-a.e.},
		\end{align*}
		and the assertion on the regularity of $F$ indeed follows by the integrability assumptions on $\omega$, $\Cont^1$-regularity of $p$, \autoref{Th:Kato-Simon} and arguing as in \autoref{Re:p in (1,2)}.
		
		Moreover, for every $t\in [0,T)$, from $\smash{\fLSI_{p(t)}\big(\varepsilon(p(t)),\gamma(p(t))\big)}$ for $\HHeat_t\omega\in V_{1,\infty}$ as well as \eqref{Eq:Relations p and N i},
		\begin{align*}
		\frac{\rmd}{\rmd t }\log F(t) &= -A'(t) + \Vert\HHeat_t\omega\Vert_{\Ell^{p(t)}}^{-1}\,\frac{\rmd}{\rmd t}\Vert\HHeat_t\omega\Vert_{\Ell^{p(t)}}\\
		&= -A'(t) - \frac{p'(t)}{p(t)}\log\Vert\HHeat_t\omega\Vert_{\Ell^{p(t)}}^{p(t)} + \frac{1}{p(t)}\,\Vert\HHeat_t\omega\Vert_{\Ell^{p(t)}}^{-p(t)}\,\frac{\rmd}{\rmd t}\Vert\HHeat_t\omega\Vert_{\Ell^{p(t)}}^{p(t)}\\
		&= -\frac{\gamma(p(t))}{\varepsilon(p(t))} - \frac{1}{\varepsilon(p(t))}\log\Vert\HHeat_t\omega\Vert_{\Ell^{p(t)}}^{p(t)}\\
		&\qquad\qquad + \frac{1}{\varepsilon(p(t))}\,\Vert\HHeat_t\omega\Vert_{\Ell^{p(t)}}^{-p(t)}\int_\mms \vert\HHeat_t\omega\vert^{p(t)}\log\vert\HHeat_t\omega\vert\d\meas\\
		&\qquad\qquad - \Vert\HHeat_t\omega\Vert_{\Ell^{p(t)}}^{-p(t)}\int_\mms \vert\HHeat_t\omega\vert^{p(t)-2}\,\langle\HHeat_t\omega,\Hodge\HHeat_t\omega\rangle\d\meas \leq 0.
		\end{align*}
		Since $\log$ is strictly increasing, $F$ is nonincreasing, yielding \ref{La:Hyper} by the density of $V_{1,\infty}$ in $\Ell^{p_0}(T^*\mms)$. 
		
		Concerning \ref{La:Ultra}, invoking the strict increasingness of $p$ and Hölder's inequality, for every $s,t\in [0,T)$ with $s<t$ and every bounded Borel set $B\subset \mms$ with positive $\meas$-measure we have
		\begin{align*}
		\Vert\One_B\,\HHeat_t\omega\Vert_{\Ell^{p(s)}} &\leq \meas[B]^{1-p(s)/p(t)}\,\Vert\HHeat_t\omega\Vert_{\Ell^{p(t)}}\\
		&\leq \meas[B]^{1-p(s)/p(t)}\,\rme^{A(t)-A(s)}\,\Vert\HHeat_s\omega\Vert_{\Ell^{p(s)}}\\ &\leq \meas[B]^{1-p(s)/p(t)}\,\rme^{A(t)}\,\Vert\omega\Vert_{\Ell^{p_0}}.
		\end{align*}
		The claim follows by letting $t\to T$ and $s\to T$ in such a way that $p(s)/p(t) \to 1$ and afterwards using the arbitrariness of $B$ as well as the density of $V_{1,\infty}$ in $\Ell^{p_0}(T^*\mms)$.
	\end{proof}
	
	\begin{example}\label{Ex:Coefficients} Given any $\beta > 0$,  the functions $\varepsilon,\gamma\in \Cont((1,\infty))$ with
		\begin{align*}
		\varepsilon(p) &:= \frac{\beta p}{2(p-1)},\\
		\gamma(p) &:= -\frac{K\beta p}{2(p-1)}.
		\end{align*}
		are the coefficients in \autoref{Le:2LSI to pLSI} arising from the setup of \autoref{Th:Log Sobolev for RCD} and \autoref{Ex:alpha}.
		
		Retaining the notation from \autoref{Th:pLSI to Ultra}, subject to these coefficients and any $p_0\in (1,\infty)$, the value $T$ is always infinite, while $C$ takes the values $-\infty$, $0$ or $\infty$ depending on whether $K>0$, $K=0$ or $K<0$. Moreover, the functions $p$ and $A$ from \autoref{Th:pLSI to Ultra} read
		\begin{align*}
		p(t) &= 1+ (p_0-1)\,\rme^{2 t/\beta},\\
		A(t) &=\int_{p_0}^{p(t)}\frac{\gamma(s)}{s}\d s  = -Kt.
		\end{align*}
	\end{example}
	
	\begin{corollary}\label{Cor: LSI for RCD} In the setting of \autoref{Ex:Coefficients}, given any $p_0\in (1,\infty)$, for every $t\geq 0$, we have
		\begin{align*}
		\Vert \HHeat_t\Vert_{\Ell^{p_0},\Ell^{p(t)}} \leq \rme^{-Kt}.
		\end{align*}
	\end{corollary}
	
	\begin{corollary}\label{Cor:eigenforms L^p} In the setting of \autoref{Ex:Coefficients}, let $\omega\in \Dom(\Hodge)$ be an eigenform for $\Hodge$ with eigenvalue $\lambda \geq 0$, i.e.~$\Hodge \omega = \lambda\,\omega$. Then $\omega\in\Ell^q(T^*\mms)$ for every $q\in (2,\infty)$ with the inequality
		\begin{align*}
		\Vert\omega\Vert_{\Ell^q} \leq (q-1)^{(\lambda-K)\beta/2}\,\Vert\omega\Vert_{\Ell^2}.
		\end{align*}
	\end{corollary}
	
	\begin{proof} Note that $\HHeat_t\omega = \rme^{-\lambda t}\,\omega$ for every $t\geq 0$. We apply \autoref{Cor: LSI for RCD} to $p_0 := 2$. Given any $q\in (1,\infty)$, since $p(t) = q$ if and only if $t = \log(q-1)\beta/2$, for this value of $t$ we have
		\begin{align*}
		\Vert \omega\Vert_{\Ell^q} = \rme^{\lambda t}\,\Vert\HHeat_t\omega\Vert_{\Ell^{p(t)}} \leq \rme^{(\lambda-K)t}\,\Vert\omega\Vert_{\Ell^{p_0}} = (q-1)^{(\lambda - K)\beta/2}\,\Vert\omega\Vert_{\Ell^2}.\tag*{\qedhere}
		\end{align*}
	\end{proof}
	
	\subsubsection{From ultracontractivity to logarithmic Sobolev inequalities}
	
	\begin{theorem}\label{Th:LSI 2 from hyper} Let $T\in (0,\infty]$ as well as $c\in \Cont((0,T))$. Suppose that
		\begin{align*}
		\Vert\HHeat_t\Vert_{\Ell^2,\Ell^\infty}\leq \rme^{c(t)}
		\end{align*}
		holds for every $t\in (0,T)$. Then every $\omega\in\Dom(\Hodge)\cap\Ell^1(T^*\mms)\cap\Ell^\infty(T^*\mms)$ satisfies $\fLSI_2(\varepsilon,c(\varepsilon))$ for every $\varepsilon \in (0,T)$.
	\end{theorem}
	
	\begin{proof} Let $\smash{U := \big\lbrace \xi \in \C : \Re\xi \in  [0,1]\big\rbrace}$. Given any $\varepsilon \in (0,T)$ and $\xi\in U$, we consider the operator
		\begin{align*}
		\sfS_\xi := \rme^{-\varepsilon \xi \Hodge}
		\end{align*}
		acting on $\Ell^2(T^*\mms) + \nli\,\Ell^2(T^*\mms)$. For every $\eta,\rho\in\Ell^2(T^*\mms) + \nli\,\Ell^2(T^*\mms)$, the canonical bilinear form in $\Ell^2(T^*\mms)+\nli\,\Ell^2(T^*\mms)$ induced by $\sfS_\xi$ evaluated at $(\eta,\rho)$ is continuous in $U$, and its restriction to the interior of $U$ is holomorphic. For every $\eta\in\Ell^2(T^*\mms) + \nli\,\Ell^2(T^*\mms)$ and every $\vartheta\in\R$, we have
		\begin{align*}
		\Vert \sfS_{\nli \vartheta}\eta\Vert_{\Ell^2} &\leq \Vert \eta\Vert_{\Ell^2},\\
		\Vert\sfS_{1+\nli\vartheta}\eta\Vert_{\Ell^\infty} &\leq \rme^{-c(\varepsilon)}\,\Vert\sfS_{\nli\vartheta}\eta\Vert_{\Ell^2}\leq\rme^{-c(\varepsilon)}\,\Vert\eta\Vert_{\Ell^2}.
		\end{align*}
		
		Given any $\omega\in\Dom(\Hodge)\cap\Ell^1(T^*\mms)\cap\Ell^\infty(T^*\mms)$, for every $\tau \in (0,1)$, via Stein's interpolation theorem we infer
		\begin{align}\label{Eq:Stein}
		\Vert \HHeat_{\varepsilon\tau}\omega\Vert_{\Ell^{2/(1-\tau)}} = \Vert\sfS_\tau\omega\Vert_{\Ell^{2/(1-\tau)}} \leq \rme^{-c(\varepsilon)\tau}\,\Vert\omega\Vert_{\Ell^2}.
		\end{align}
		
		Now we define $p\in \Cont^1([0,\varepsilon))$ by $p(t) := 2\varepsilon/(\varepsilon -t)$. Setting $\tau := t/\varepsilon$ in \eqref{Eq:Stein} translates into
		\begin{align*}
		\Vert\HHeat_t\omega\Vert_{\Ell^{p(t)}}^{p(t)} \leq \rme^{-c(\varepsilon)p(t)t/\varepsilon}\,\Vert\omega\Vert_{\Ell^2},
		\end{align*}
		and the claim follows after differentiating both sides at $0$ via
		\begin{align*}
		\int_\mms \vert\omega\vert^2\,\Big[\frac{2}{\varepsilon}\log\vert\omega\vert - 2\,\vert\omega\vert^{-2}\,\langle\omega,\Hodge\omega\rangle\Big]\d\meas \leq \frac{2c(\varepsilon)}{\varepsilon}\,\Vert\omega\Vert_{\Ell^2}.\tag*{\qedhere}
		\end{align*}
	\end{proof}
	
	\begin{example} If $\ChHeat_t$ is bounded from $\Ell^2(\mms)$ to $\Ell^\infty(\mms)$, then so is $\HHeat_t$ by means of \autoref{Th:Kato-Simon}. Compare this with (the proof of) \autoref{Th:Lp Linfty}.
		
		See \cite[Ch.~4]{charalambous2007} for an application to certain Gaussian upper bounds for the heat kernel on $1$-forms in the non-weighted smooth setting.
	\end{example}
	
	\section{Spectral properties of the Hodge Laplacian}\label{Ch:Spectral props}
	
	Next, we study properties of the \emph{spectrum} $\sigma(\Hodge)$ of $\Hodge$, i.e.~the set of all $\lambda\in\C$ such that the operator $\Hodge - \lambda$ fails to be bijective. We denote the \emph{resolvent set} of $\Hodge$ by $\rho(\Hodge) := \C\setminus\sigma(\Hodge)$. The \emph{point spectrum} of $\Hodge$ is denoted by $\sigma_\pt(\Hodge)$, and the \emph{essential spectrum} of $\Hodge$ will be termed $\sigma_\ess(\Hodge)$. The former is the set of all $\lambda\in \sigma(\Hodge)$ for which $\Hodge-\lambda$ is not injective, and the second is the set of all $\lambda\in\sigma(\Hodge)$ which are not an eigenvalue of $\Hodge$ with finite algebraic multiplicity. Similar notations are employed for the operators $-\Delta$ and $-\Delta+K$.
	
	One immediately sees that since $\Hodge$ is self-adjoint and nonnegative, we have
	\begin{align*}
	\sigma(\Hodge) \subset [0,\infty).
	\end{align*}
	In addition, eigenspaces w.r.t.~different eigenvalues are mutually orthogonal in $\Ell^2(T^*\mms)$.
	
	\subsection{Inclusion of spectra}\label{Sec:Basic sectral properties}
	
	In this section, we show that, except the critical value $0$, the spectrum of the negative functional Laplacian $-\Delta$ is contained in $\sigma(\Hodge)$. Similar inclusions hold between the respective point and essential spectra. See \autoref{Th:Spectrum of fcts and of forms}. Our proof follows the smooth treatise for \cite[Cor.~4.4, Cor.~4.5]{charalambous2019}.
	
	As an important application, in \autoref{Cor:Spectral gap inclusion} we derive explicit relations between the spectral gaps of the Schrödinger operator $-\Delta+K$, $\Hodge$ and $-\Delta$.
		
	We shall need the subsequent characterization of points in the (essential) spectrum of $\Hodge$. See \cite[Prop.~2.5]{charalambous2019} and the references therein for a more general statement.
	
	\begin{lemma}\label{Le:Generalized Weyl} For every $\lambda > 0$, we have $\lambda\in\sigma(\Hodge)$ if and only if there exist $\alpha < 0$ and a sequence $(\omega_n)_{n\in\N}$ in $\Dom(\Hodge)$ such that
		\begin{enumerate}[label=\textnormal{\alph*.}]
			\item\label{La:AA} $\Vert \omega_n\Vert_{\Ell^2} = 1$ for every $n\in\N$, and
			\item\label{La:BB} for every $j\in\{1,2\}$, one has
			\begin{align*}
			\lim_{n\to\infty}\int_\mms \big\langle (\Hodge-\alpha)^{-j}\omega_n, \Hodge\omega_n - \lambda\,\omega_n\big\rangle\d\meas = 0.
			\end{align*}
		\end{enumerate}
		Moreover, a number $\lambda > 0$ belongs to the essential spectrum of $\Hodge$ if and only if some sequence $(\omega_n)_{n\in\N}$ in $\Dom(\Hodge)$ satisfies the previous conditions \ref{La:AA} and \ref{La:BB} as well as
		\begin{enumerate}[label=\textnormal{\alph*.}]\setcounter{enumi}{2}
			\item\label{La:CC} $\omega_n \rightharpoonup 0$ in $\Ell^2(T^*\mms)$ as $n\to\infty$.
		\end{enumerate}
	\end{lemma}
	
	\begin{lemma}\label{Le:Commutation resolvent differential} Let $\alpha < 0$. Then for every $f\in W^{1,2}(\mms)$, we have
		\begin{align*}
		(\Hodge -\alpha)^{-1}\rmd f = \rmd(-\Delta-\alpha)^{-1}f,
		\end{align*}
		while for every $\omega\in\Dom(\delta)$, we have
		\begin{align*}
		(-\Delta -\alpha)^{-1}\delta\omega = \delta(\Hodge-\alpha)^{-1}\omega.
		\end{align*}
	\end{lemma}
	
	\begin{proof} Any $\alpha < 0$ belongs to $\rho(\Hodge)$ and $\rho(-\Delta)$, thus $\Hodge-\alpha$ and $-\Delta-\alpha$ are invertible with bounded inverse. Furthermore, the second identity follows from the first by definition of $\delta$ and the self-adjointness of $(\Hodge-\alpha)^{-1}$, since $\alpha$ is real --- we thus concentrate on the proof of the first equality.
		
		Given any $f\in W^{1,2}(\mms)$, let $u\in\Dom(\Delta)$ be the unique solution to the equation $-\Delta u - \alpha\, u = f$ on $\mms$. By \autoref{Le:Ht vs Pt}, for every $t>0$ we have $\rmd \ChHeat_tu \in \Dom(\Hodge)$ and
		\begin{align*}
		\Hodge \rmd\ChHeat_t u - \alpha\d\ChHeat_t u = -\rmd(\Delta\ChHeat_t u + \alpha\, \ChHeat_tu) = \rmd \ChHeat_tf.
		\end{align*}
		Therefore $\HHeat_t\rmd u = (\Hodge-\alpha)^{-1}\HHeat_t\rmd f$ again by \autoref{Le:Ht vs Pt}, and the claim follows by letting $t\to 0$.
	\end{proof}
	
	\begin{theorem}\label{Th:Spectrum of fcts and of forms} Let $(\mms,\met,\meas)$ be an $\RCD(K,\infty)$ space for some $K\in\R$.  Then
		\begin{align*}
		\sigma_\pt(-\Delta)&\subset \sigma_\pt(\Hodge),\\
			\sigma(-\Delta)\setminus\{0\} &\subset \sigma(\Hodge),\\
			\sigma_\ess(-\Delta)\setminus\{0\} &\subset \sigma_\ess(\Hodge).
		\end{align*}
	\end{theorem}
	
	\begin{proof} The first inclusion is elementary, since for every $\lambda \in\sigma_\pt(-\Delta)$ and its corresponding eigenfunction $f\in\Dom(\Delta)$, by \autoref{Le:Ht vs Pt},  $\rmd\ChHeat_1f\in\Dom(\Hodge)$ and
		\begin{align*}
		\Hodge\rmd\ChHeat_1f = -\rmd\Delta\ChHeat_1 f = \lambda\d\ChHeat_1f. 
		\end{align*}
		
		To prove the second inclusion, let $\lambda \in \sigma(-\Delta)\setminus\{0\}$. By Weyl's criterion \cite[Thm.~5.10]{hislop1996} applied to $-\Delta$, for every $n\in\N$ there exists $g_n \in\Dom(\Delta)$ with
		\begin{align*}
		\Vert g_n\Vert_{\Ell^2}&= 1,\\
		 \Vert \Delta g_n + \lambda\,g_n\Vert_{\Ell^2}&\leq 2^{-n}.
		\end{align*}
		Moreover, for every $n\in\N$ there exists $t_n > 0$ such that $f_n := \ChHeat_{t_n}g_n\in\Dom(\Delta)$ satisfies the similar estimates
		\begin{align}\label{Eq:Props of fn i}
		\begin{split}
		\sqrt{2}^{-1}&\leq \,\Vert f_n\Vert_{\Ell^2}\leq 1,\\
		\Vert\Delta f_n +\lambda f_n\Vert_{\Ell^2} &\leq 2^{-n}.
		\end{split}
		\end{align}
		
		Provided that $2^{-n} \leq \lambda/4$, from \eqref{Eq:Props of fn i}  we get
		\begin{align}\label{Eq:Props of fn II}
		\begin{split}
		\int_\mms \vert\rmd f_n\vert^2\d\meas &=-\int_\mms f_n\,\Delta f_n\d\meas\\
		&\geq - \Vert \Delta f_n + \lambda\,f_n\Vert_{\Ell^2} + \lambda\,\Vert f_n\Vert_{\Ell^2}^2 \geq \frac{\lambda}{4} >0.
		\end{split}
		\end{align}
		Possibly relabeling $(f_n)_{n\in\N}$, we may and will assume that the sequence $(\Vert\rmd f_n\Vert_{\Ell^2}^2)_{n\in \N}$ is uniformly bounded from below by $\lambda/4$. In particular, for  $j\in\{1,2\}$ it follows from \autoref{Le:Ht vs Pt}, \autoref{Le:Commutation resolvent differential}, contractivity of $(-\Delta+1)^{-j}$ in $\Ell^2(\mms)$, \eqref{Eq:Props of fn i} and  \eqref{Eq:Props of fn II} that
		\begin{align*}
		&\Big\vert\!\int_\mms \big\langle (\Hodge + 1)^{-j}\rmd f_n, (\Hodge-\lambda)\rmd f_n\big\rangle\d\meas\Big\vert\\
		&\qquad\qquad = \Big\vert\!\int_\mms \big[\delta(\Hodge+1)^{-j}\rmd f_n\big]\, (\Delta f_n+\lambda\, f_n) \d\meas\Big\vert \\
		&\qquad\qquad = \Big\vert\!\int_\mms \big[(-\Delta+1)^{-j}\Delta f_n\big]\,(\Delta f_n+\lambda\,f_n)\d\meas\Big\vert\\
		&\qquad\qquad \leq 2^{-n}\,\Vert \Delta f_n\Vert_{\Ell^2} \leq 2^{-n}\,\big(\lambda + 2^{-n}\big) \leq \frac{4(\lambda+1)}{\lambda}\,2^{-n}\,\Vert\rmd f_n\Vert_{\Ell^2}^2.\textcolor{White}{\int_\mms}
		\end{align*}
		In particular, the sequence $(\omega_n)_{n\in\N}$  given by
		\begin{align}\label{Eq:Sequence def}
		\omega_n := \Vert \rmd f_n\Vert_{\Ell^2}^{-1} \d f_n,
		\end{align}
		which takes values in $\Dom(\Hodge)$ by \autoref{Le:Ht vs Pt}, obeys \ref{La:AA} and \ref{La:BB} from \autoref{Le:Generalized Weyl}, whence $\lambda \in \sigma(\Hodge)$.
		
		Turning to the last inclusion, if $\lambda\in\sigma_\ess(-\Delta)$, then the sequence $(f_n)_{n\in\N}$ from the previous step can be constructed to satisfy $f_n \rightharpoonup 0$ in $\Ell^2(\mms)$ as $n\to\infty$ in addition to \eqref{Eq:Props of fn i} above \cite[Thm.~7.2]{hislop1996}. Therefore, for every $\eta\in\Test(T^*\mms)$ we obtain for the sequence $(\omega_n)_{n\in\N}$ defined in \eqref{Eq:Sequence def} that
		\begin{align*}
		\lim_{n\to\infty} \Big\vert\! \int_\mms \langle\omega_n,\eta\rangle\d\meas\Big\vert \leq \lim_{n\to\infty} \frac{2}{\sqrt{\lambda}}\,\Big\vert\!\int_\mms f_n\,\delta\eta\d\meas\Big\vert = 0.
		\end{align*}
		Since $\Vert \omega_n\Vert_{\Ell^2} = 1$ for every $n\in\N$, this provides \ref{La:CC} in \autoref{Le:Generalized Weyl}.
	\end{proof}
	
	\begin{corollary}\label{Cor:Spectral gap inclusion} Under the assumptions of \autoref{Th:Spectrum of fcts and of forms}, we have
		\begin{align*}
		\inf \sigma(-\Delta + K) \leq \inf\sigma(\Hodge)\leq  \inf \sigma(\Hodge)\setminus\{0\}\leq \inf \sigma(-\Delta)\setminus\{0\}.
		\end{align*}
	\end{corollary}
	
	\begin{proof} The last inequality follows from \autoref{Th:Spectrum of fcts and of forms}. 
		
		The proof of the first inequality basically reduces to an inequality between quadratic forms. Indeed, from \autoref{Th:Ricci tensor} and \autoref{Pr:Kato's inequality} we obtain, for every $\omega\in H^{1,2}(T^*\mms)$ with $\Vert \omega\Vert_{\Ell^2}=1$,
		\begin{align*}
		\int_\mms \big[\vert \rmd \omega\vert^2 + \vert\delta\omega\vert^2\big]\d\meas &\geq \int_\mms \big[\vert\nabla \omega^\sharp\vert^2 + K\, \vert \omega\vert^2\big]\d\meas\\
		&\geq \int_\mms \big[\vert\nabla \vert \omega\vert\vert^2 + K\, \vert \omega\vert^2\big]\d\meas \geq \inf\sigma(-\Delta+K),
		\end{align*}
		and we conclude by taking the infimum over $\omega$ as above.
	\end{proof}
	
	\begin{remark} In the setting of \autoref{Re:Variable}, without any change of the previous proof, under a variable, uniformly lower bounded lower Ricci bound $\smash{k\in\Ell_\loc^1(\mms)}$ for $(\mms,\met,\meas)$, one verifies that
		\begin{align*}
		\inf\sigma(-\Delta+k)\leq\inf(\Hodge).\tag*{\qedhere}
		\end{align*}
	\end{remark}
	
	\subsection{The spectrum in the compact case}\label{Sec:Spectral properties compact}
	
	Much more about $\sigma(\Hodge)$ can be said if $(\mms,\met,\meas)$ is a compact $\RCD^*(K,N)$ space. In this framework, adopted throughout this section, we prove that $\sigma(\Hodge)$ is discrete and only consists of eigenvalues, see \autoref{Pr:Eigenbasis}.  A closely related result is that the natural inclusion of $H^{1,2}(T^*\mms)$ into $\Ell^2(T^*\mms)$ is compact, see \autoref{Th:Compact embedding}. In turn, by abstract functional analysis, this follows if $\HHeat_t$ is a Hilbert--Schmidt operator on $\Ell^2(T^*\mms)$ for every $t>0$, which is the content of \autoref{Cor:Hilbert-Schmidt}. 
	
	Afterwards, we establish the boundedness of eigenforms for $\Hodge$ with an explicit growth rate for their $\Ell^\infty$-norms for positive eigenvalues, see \autoref{Pr:Harmonic forms bounded} and \autoref{Th:L infty estimates}. The entire discussion in this section heavily relies on the $\Ell^2$-$\Ell^\infty$-regularization property of $(\HHeat_t)_{t\geq 0}$ from \autoref{Th:Lp Linfty}.
	
	The simple proof of the subsequent lemma is taken from \cite[Subsec.~1.8.4]{arendt2006}.
	
	\begin{lemma} Suppose that $\sfS$ is a linear operator which maps $\Ell^2(T^*\mms)$ boundedly into $\Ell^\infty(T^*\mms)$. Then $\sfS$ is a Hilbert--Schmidt operator.
	\end{lemma}
	
	\begin{proof} Let $(\omega_i)_{i\in\N}$ be any orthonormal basis of the separable Hilbert space $\Ell^2(T^*\mms)$. Given $d\in\N$, we denote by $\smash{B^d}$ the closed unit ball of $\R^d$. Let $\smash{C^d}$ be fixed countable and dense subset of $\smash{B^d}$.
		
		It easily follows from the boundedness of $\sfS$ that, for every $\smash{a \in C^d}$,
		\begin{align}\label{Eq:S bound}
		\Big\vert\sfS \Big[\sum_{i=1}^d a_i\,\omega_i\Big]\Big\vert \leq \Vert \sfS\Vert_{\Ell^2,\Ell^\infty}\quad\meas\text{-a.e.}
		\end{align}
		Hence, there exists an $\meas$-null set $P\subset\mms$ such that for every $x\in \mms\setminus P$, the inequality from \eqref{Eq:S bound} holds true for every $\smash{a\in C^d}$. It follows that
		\begin{align*}
		\sum_{i=1}^d \vert \sfS\omega_i\vert^2 = \sup\!\Big\lbrace\Big[\! \sum_{i=1}^d a_i\,\sfS\omega_i\Big]^2 : a\in C^d\Big\rbrace\leq \Vert\sfS\Vert_{\Ell^2,\Ell^\infty}^2\quad\text{on }\mms\setminus P.
		\end{align*}
		
		Integrating this inequality and using the arbitrariness of $d$ concludes the proof.
	\end{proof}
	
	\begin{corollary}\label{Cor:Hilbert-Schmidt} For every $t>0$, $\HHeat_t$ is a Hilbert--Schmidt operator on $\Ell^2(T^*\mms)$.
	\end{corollary}
	
	In particular, $\HHeat_t$ is compact on $\Ell^2(T^*\mms)$ for every $t>0$. Employing standard functional analytic results, see e.g.~\cite[Cor.~1.5]{lenz2010}, this entails the following crucial Rellich-type theorem. (The reader is also invited to consult \cite[Thm.~6.1]{honda2020}, where the same result has independently been proven by different, more geometric means via $\delta$-splitting maps.) It implies in particular that $\Harm(T^*\mms)$ is a closed subspace of $\Ell^2(T^*\mms)$, hence the $\Ell^2$-orthogonal projection $\sfT$ onto $\Harm(T^*\mms)$ is well-defined. \autoref{Th:PC compact} is then easily argued by contradiction.
	
	\begin{theorem}[Compactness of $\Hodge^{-1}$]\label{Th:Compact embedding} Let $(\mms,\met,\meas)$ be a compact $\RCD^*(K,N)$ space with $K\in\R$ and $N\in (1,\infty)$. Then the natural inclusion of $H^{1,2}(T^*\mms)$ into $\Ell^2(T^*\mms)$ is compact.
	\end{theorem}
	
	\begin{lemma}\label{Th:PC compact} There exists a constant $C <\infty$ such that for every $\omega\in\Ell^2(T^*\mms)$, 
		\begin{align}\label{Eq:Inequality Poincare}
		\Vert \omega- \sfT\omega\Vert_{\Ell^2}^2 \leq 2C\,\calE(\omega).
		\end{align}
	\end{lemma}
	
		
	
	\begin{remark}\label{Re:Poincaré discussion} \autoref{Th:PC compact} can be seen as a qualitative global Poincaré inequality for $\Hodge$. 		In contrast to logarithmic Sobolev inequalities, we did not derive local or global Poincaré inequalities for $1$-forms from the corresponding functional estimates. Combining \autoref{Pr:Kato's inequality} with \cite[Thm.~1.1, Thm.~1.2]{rajala2012} or \cite[Thm.~30.24]{villani2009}, this would be possible to some extent.
		
		Paying the price of a less explicit constant, the point is however that the terms $\smash{\int_{B_r(x)} f\d\meas}$ or $\smash{\int_\mms f\d\meas}$ if $\meas[\mms] < \infty$, respectively, appearing in the functional versions are the $\Ell^2$-orthogonal projections of $f$ onto the space of harmonic functions on the respective $\Ell^2$-spaces, while their $1$-form counterparts $\smash{\int_{B_r(x)}\vert\omega\vert\d\meas}$ or $\smash{\int_\mms\vert\omega\vert\d\meas}$ appearing in the derived estimates arguing as for \autoref{Th:Log Sobolev for RCD} would clearly lack this interpretation.
	\end{remark}
	
	\begin{corollary}\label{Cor:Eigenvalue Cor} Let $C > 0$ be any constant for which \eqref{Eq:Inequality Poincare} holds, and let $\lambda\in \sigma(\Hodge)\setminus\{0\}$. Then $\lambda$ is an eigenvalue of $\Hodge$ and satisfies the inequality
		\begin{align*}
		\lambda \geq 1/C.
		\end{align*}
	\end{corollary}
	
	\begin{proof} Let $\lambda \in\sigma(\Hodge)\setminus\{0\}$. As in the proof of \autoref{Th:Spectrum of fcts and of forms}, by Weyl's criterion there exists a sequence $(\omega_n)_{n\in\N}$ in $\Dom(\Hodge)$ such that, for every $n\in\N$,
		\begin{align}\label{Eq:Weyl sequence cpt x}
		\begin{split}
		\Vert\omega_n\Vert_{\Ell^2} &= 1,\\
		\Vert \Hodge \omega_n - \lambda\,\omega_n\Vert_{\Ell^2} &\leq 2^{-n}.
		\end{split}
		\end{align}
		
		To prove that $\lambda$ is an eigenvalue, observe that $(\Vert\omega_n\Vert_{H^{1,2}})_{n\in\N}$ is uniformly bounded by \eqref{Eq:Weyl sequence cpt x}. According to \autoref{Th:Compact embedding}, a non-relabeled subsequence of $(\omega_n)_{n\in\N}$ converges weakly in $H^{1,2}(T^*\mms)$ and strongly in $\Ell^2(T^*\mms)$ to some $\omega\in H^{1,2}(T^*\mms)$ with $\Vert\omega\Vert_{\Ell^2}=1$. Given any $\rho\in\Test(T^*\mms)$, since
		\begin{align*}
		\lambda\int_\mms \langle\rho,\omega\rangle\d\meas &= \lim_{n\to\infty} \int_\mms \langle\rho,\Hodge\omega_n\rangle\d\meas = \lim_{n\to\infty} \int_\mms \big[\langle\rmd\rho,\rmd \omega_n\rangle + \delta\rho\,\delta\omega_n\big]\d\meas\\
		&= \int_\mms \big[\langle\rmd\rho,\rmd\omega\rangle + \delta\rho\,\delta\omega\big]\d\meas,
		\end{align*}
		we also obtain that $\omega\in\Dom(\Hodge)$ with $\Hodge\omega = \lambda\,\omega$, which is the claim.
		
		The bound $\lambda\geq 1/C$ then follows by inserting $\omega$ into \eqref{Eq:Inequality Poincare}, recalling that eigenspaces w.r.t.~different eigenvalues are orthogonal in $\Ell^2(T^*\mms)$.
	\end{proof}
	
	\begin{remark} The nonsmooth analogue of Hodge's theorem \cite[Thm.~3.5.15]{gigli2018} gives a one-to-one correspondence between the multiplicity of harmonic $1$-forms in terms of the dimension of the first de Rham cohomology group $H_{\mathrm{dR}}^1(\mms)$ as defined in \cite[Def.~3.5.8]{gigli2018}.
	\end{remark}
	
	In view of \autoref{Cor:Eigenvalue Cor}, given $\lambda \geq 0$ we denote the \emph{eigenspace of $\Hodge$ w.r.t.~$\lambda$} by
	\begin{align*}
	\nlE_\lambda(\Hodge) := \big\lbrace \omega\in\Dom(\Hodge) : \Hodge\omega = \lambda\,\omega\big\rbrace.
	\end{align*}
	The proofs of the following basic results are standard once having \autoref{Th:Compact embedding} and \autoref{Cor:Eigenvalue Cor} at our disposal. We refer to \cite[Thm.~4.3]{honda2017} as well as  \cite[Lem.~VI.3.6, Prop.~VI.3.8]{sakai1996} for similar statements in the smooth setting and for comprehensive proofs. 
 	
	\begin{theorem}\label{Pr:Eigenbasis} The spectrum $\sigma(\Hodge)$ has the following properties.
		\begin{enumerate}[label=\textnormal{(\roman*)}]
			\item \textnormal{\textsc{Finite dimensionality.}} For every $\lambda \geq 0$, the vector space dimension of $\nlE_\lambda(\Hodge)$ is finite.
			\item \textnormal{\textsc{Discreteness and unboundedness.}} The spectrum $\sigma(\Hodge)$ is discrete \textnormal{(}i.e.~for every $\lambda\in \sigma(\Hodge)$ there exists some $r>0$ such that $(\lambda-r,\lambda+r)\cap \sigma(\Hodge) = \{\lambda\}$\textnormal{)} and unbounded.
			\item \textnormal{\textsc{Variational principle.}} Let $(\lambda_i)_{i\in\N}$ be an increasing enumeration of the eigenvalues of $\Hodge$ counted with multiplicities. Let $\S$ denote the unit sphere in $\Ell^2(T^*\mms)$. Then for every $i\in\N$,
			\begin{align*}
			\lambda_i = \inf\!\Big\lbrace\sup_{\omega\in E\cap\S}2\,\calE(\omega) : E\subset H^{1,2}(T^*\mms)\text{ subspace with }\dim E = i\Big\rbrace.
			\end{align*}
			\item \textnormal{\textsc{Orthonormal eigenbasis.}} The direct sum
			\begin{align*}
			\nlE(\Hodge) := \bigoplus_{\lambda \in \sigma(\Hodge)} \nlE_\lambda(\Hodge)
			\end{align*}
			is dense both in $H^{1,2}(T^*\mms)$ and $\Ell^2(T^*\mms)$, endowed with their respective norms. In particular, there exists a countable orthonormal basis $(\omega_i)_{i\in\N}$ of $\Ell^2(T^*\mms)$ such that, for every $i\in\N$, we have $\omega_i \in \nlE_{\lambda_i}(\Hodge)$ for some  $\lambda_i \in \sigma(\Hodge)$.
		\end{enumerate}
	\end{theorem}
	
	Since $\omega = \HHeat_1\omega$ for every $\omega\in\Harm(T^*\mms)$, \autoref{Th:Lp Linfty}
	immediately provides \autoref{Pr:Harmonic forms bounded} below. An argument as in the proof of \autoref{Th:Lp Linfty} with a finer estimation then yields \autoref{Th:L infty estimates}. For similar statements, see \cite[Prop.~7.1]{ambrosio2018} for functions --- whose proof is adopted in our approach --- and \cite[Prop.~4.14]{honda2017} for arbitrary tensor fields in the Ricci limit framework.
	
	\begin{corollary}\label{Pr:Harmonic forms bounded} We have $\Harm(T^*\mms)\subset \Ell^\infty(T^*\mms)$.
	\end{corollary}
	
	\begin{theorem}\label{Th:L infty estimates} Let $(\mms,\met,\meas)$ be a compact $\RCD^*(K,N)$ space for some $K\in\R$ and $N\in (1,\infty)$. Let $D>0$ obey $\diam\mms \leq D$. Assume that $\omega\in \Dom(\Hodge)$ is an eigenform with eigenvalue $\lambda \in [D^{-2},\infty)$ and $\Vert \omega\Vert_{\Ell^2}=1$. Then there exists a constant $C<\infty$ depending only on $K$, $N$ and $D$ such that
		\begin{align*}
		\Vert\omega\Vert_{\Ell^\infty}\leq C\,\lambda^{N/4}.
		\end{align*}
	\end{theorem}
	
	\begin{proof} Since $\omega \in \nlE_\lambda(\Hodge)$, it follows that $\HHeat_t\omega = \rme^{-\lambda t}\omega$ for every $t\geq 0$. Thus, for $t\in (0, D^2]$ to be determined later, \autoref{Th:Kato-Simon} and then \eqref{Eq:Heat kernel bound} for $\varepsilon := 1$ yield the existence of constants $C_1,C_2 < \infty$ depending only on $K$ and $N$ such that
		\begin{align}\label{Eq:Huu}
		\vert\omega\vert &= \rme^{\lambda t}\,\vert\HHeat_t\omega\vert \leq \rme^{(\lambda + K^-)t}\int_\mms\sfp_t(\cdot,y)\,\vert\omega\vert(y)\d\meas(y)\nonumber\\
		&\leq \rme^{(\lambda + K^-)t}\,\Big[\!\int_\mms\sfp_t^2(\cdot,y)\d\meas(y)\Big]^{1/2}\nonumber\\
		&\leq C_1\,\rme^{(\lambda + K^-+C_2)t}\,\meas\big[B_{\!\sqrt{t}}(\cdot)\big]^{-1}\,\Big[\!\int_\mms \rme^{-2\met^2(\cdot,y)/5t}\d\meas(y) \Big]^{1/2}\quad\meas\text{-a.e.}
		\end{align}
		
		Arguing exactly as in the proof of \cite[Prop.~7.1]{ambrosio2018}, using \eqref{Eq:Bishop-Gromov}, under the given assumptions we find a constant $C<\infty$ depending only on $K$, $N$ and $D$ such that
		\begin{align*}
		\meas\big[B_{\!\sqrt{t}}(\cdot)\big]^{-1}\,\Big[\!\int_\mms \rme^{-2\met^2(\cdot,y)/5t}\d\meas(y) \Big]^{1/2} \leq C\,\Big(\frac{D}{\sqrt{t}}\Big)^{N/2}\quad\meas\text{-a.e.}
		\end{align*}
		The choice of $t:= 1/\lambda$ gives the desired estimate.
	\end{proof}
	
	\subsection[Independence of the $\Ell^p$-spectrum on $p$]{Independence of the $\Ell^p$-spectrum on $p$}\label{Sec:Independence spectrum}
	
	In this section, fix an $\RCD^*(K,N)$ space $(\mms,\met,\meas)$, where $K\in\R$ and $N\in (1,\infty)$. Under a volume growth assumption stated in \autoref{Def:Unif sub int}, following \cite{charalambous2005, hempel1986, sturm1993} we show that the $\Ell^p$-spectrum of $\Hodge$ is independent of $p\in[1,\infty]$, see \autoref{Th:Lp independence theorem} below. 
	
	To keep the presentation clear, in this section we denote by $\Hodge_{2} := \Hodge$ the Hodge Laplacian acting on $\Ell^2(T^*\mms)$ and by $(\HHeat_{2,t})_{t\geq 0}$ the associated semigroup $\HHeat_{2,t} := \HHeat_t$. Recalling \autoref{Cor:Extension}, by $(\HHeat_{p,t})_{t\geq 0}$ we denote the strongly continuous extension of $(\HHeat_{2,t})_{t\geq 0}$ to $\Ell^p(T^*\mms)$ for every $p\in [1,\infty)$. Let $\Hodge_{p}$ be the infinitesimal generator of $(\HHeat_{p,t})_{t\geq 0}$. We also define $\HHeat_{\infty,t}$ and $\Hodge_{\infty}$ on $\Ell^\infty(T^*\mms)$ as the adjoints of $\HHeat_{1,t}$ and $\Hodge_{1}$, respectively.
	
	Given any $p\in [1,\infty]$ and $n\in\N$, by \cite[Thm.~IX.4.1, Cor.~IX.4.1]{yosida1980} we know that, for every $\xi\in \rho(\Hodge_p)$ with $\Re\xi < K^-$, we have
	\begin{align}\label{Eq:Resolvent formula with heat flow}
	(\Hodge_{p}-\xi)^{-n} &= \frac{1}{(n-1)!}\int_0^\infty \rme^{\xi t}\,t^{n-1}\,\HHeat_{p,t}\d t.
	\end{align}
	
	\begin{definition}\label{Def:Unif sub int} We say that the reference measure $\meas$ on $(\mms,\met)$ is \emph{uniformly subexponentially integrable} if for every $\varepsilon>0$, we have
		\begin{align*}
		\sup_{x\in\mms} \int_\mms\rme^{-\varepsilon\met(x,y)}\,\meas[B_1(x)]^{-1/2}\,\meas[B_1(y)]^{-1/2}\d\meas(y) < \infty.
		\end{align*}
	\end{definition}
	
	\begin{remark}\label{Re:Suff cond Exp gr} By the same argument as for \cite[Prop.~1]{sturm1993}, $\meas$ is uniformly subexponentially integrable if for every $\varepsilon >0$, there exists  $C <\infty$ such that  for every $x\in\mms$ and every $r>0$,
		\begin{align*}
		\meas[B_r(x)] \leq C\,\rme^{\varepsilon r}\,\meas[B_1(x)].\tag*{\qedhere}
		\end{align*}
	\end{remark}
	
	\begin{example} If $(\mms,\met,\meas)$ is globally doubling \eqref{Eq:Doubling condition}, then $\meas$ is uniformly subexponentially integrable. Indeed, by a well-known iteration argument starting from \eqref{Eq:Doubling condition}, the sufficient condition from the previous \autoref{Re:Suff cond Exp gr} follows from the existence of finite constants $\alpha,\beta > 0$ such that
		\begin{align*}
		\meas[B_r(x)] \leq \beta\,r^\alpha\,\meas[B_1(x)]
		\end{align*}
		holds for every $x\in\mms$ and every $r\geq 1$.
		
		Since $\RCD^*(K,N)$ spaces with a nonnegative lower Ricci bound are globally doubling \cite[Cor.~2.4]{sturm2006b}, uniform subexponential integrability of $\meas$ is granted as soon as $K\geq 0$.
	\end{example}
	
	\begin{theorem}\label{Th:Lp independence theorem} Let $(\mms,\met,\meas)$ be an $\RCD^*(K,N)$ space for some $K\in\R$ and $N\in (1,\infty)$. Assume that $\meas$ is uniformly subexponentially integrable. Then the spectrum $\sigma(\Hodge_{p})$ of the operator $\Hodge_{p}$ acting on $\Ell^p(T^*\mms)$ is equal to $\sigma(\Hodge_{2})$ for every $p\in [1,\infty]$. 		Furthermore, for every $p,q\in [1,\infty]$, every isolated eigenvalue of $\Hodge_{p}$ with finite algebraic multiplicity is also an isolated eigenvalue of $\Hodge_{q}$ with the same algebraic multiplicity.
	\end{theorem}
	
	The key point of the proof of \autoref{Th:Lp independence theorem} is a perturbation argument whose core we outsource into \autoref{Le:L2 boundedness}, \autoref{Le:Resolvent bounds} and \autoref{Cor:Linfty boundedness of resolvent} below. Before that, we quickly fix some notation. 
	
	We define the measurable function $\phi_1\colon \mms\to \R$ by
	\begin{align*}
	\phi_1(x) := \meas[B_1(x)]^{1/2}.
	\end{align*}
	
	Given any $\varepsilon>0$, we consider the class
	\begin{align*}
	\Gamma_\varepsilon := \big\lbrace \psi \in W^{1,2}(\mms)\cap\Cont_\bounded(\mms) : \vert\rmd\psi\vert \leq \varepsilon\ \meas\text{-a.e.} \big\rbrace
	\end{align*}
	and recall from \cite[Thm.~4.17]{ambrosio2015} that, for every $x,y\in\mms$,
	\begin{align}\label{Eq:Intrinsic distance}
	\varepsilon\,\met(x,y) = \sup\!\big\lbrace \psi(x) - \psi(y) : \psi \in \Gamma_\varepsilon\big\rbrace.
	\end{align}
	
	Lastly, given $\psi \in \Gamma_\varepsilon$, by $\rme^\psi\,\Hodge_{2}\,\rme^{-\psi}$ we intend the linear, densely defined operator on $\Ell^2(T^*\mms)$ given by setting, for arbitrary $\omega,\eta\in\Test(T^*\mms)$,
	\begin{align}\label{Eq:Formula in proof}
	\begin{split}
	&\int_\mms \big\langle \eta,\big(\rme^\psi\,\Hodge_{2}\,\rme^{-\psi}\big)\omega\big\rangle\d\meas\\
	&\qquad\qquad := \int_\mms \langle\eta, \Hodge_{2}\omega\rangle\d\meas -\int_\mms \vert\nabla\psi\vert^2\,\langle\eta,\omega\rangle\d\meas \\
	&\qquad\qquad\qquad\qquad  + \int_\mms \big[\nabla\omega^\sharp\big(\nabla \psi,\eta^\sharp\big) - \nabla\eta^\sharp\big(\nabla \psi,\omega^\sharp\big)\big]\d\meas.
	\end{split}
	\end{align}
	
	\begin{remark} Observe that if $\psi$ is sufficiently regular, say, $\psi \in\Gamma_\varepsilon\cap\Test(\mms)$ with $\Delta\psi\in\Ell^\infty(\mms)$, then $\rme^{\psi}\,\Hodge_{2}\,\rme^{-\psi} \omega$ is pointwise well-defined on any $\omega\in\Test(T^*\mms)$ as composition of the multiplication operators $\rme^\psi$ and $\rme^{-\psi}$ as well as the Hodge Laplacian $\Hodge_{2}$ in the indicated order by \eqref{Le:Hodge Laplacian on test forms}. In this case, \eqref{Eq:Formula in proof} follows by a straightforward computation using \autoref{Le:Laplacian functions chain rule} and \eqref{Eq:Compatibility with the metric}. 
		
		The class of functions in $\Gamma_\varepsilon\cap \Test(\mms)$ with bounded Laplacian is dense in $\Gamma_\varepsilon$ w.r.t.~strong convergence in $W^{1,2}(\mms)$, see \autoref{Le:Mollified heat flow}.
	\end{remark}
	
	\begin{lemma}\label{Le:L2 boundedness} For every compact $V\subset \rho(\Hodge_{2})$, there exist $\varepsilon \in (0,1)$ and a constant $C< \infty$ such that for every $\xi \in V$ and every $\psi \in \Gamma_\varepsilon$, $\xi \in \rho(\rme^{\psi}\,\Hodge_{2}\,\rme^{-\psi})$ with
		\begin{align}\label{Eq:L2L2 bound}
		\Vert (\rme^\psi\,\Hodge_{2}\,\rme^{-\psi} - \xi)^{-1}\Vert_{\Ell^2,\Ell^2}\leq C.
		\end{align}
	\end{lemma}
	
	\begin{proof} Given any $\varepsilon\in (0,1)$ to be determined later and $\psi\in\Gamma_\varepsilon$, the operator
		\begin{align*}
		\sfT_\psi := \rme^\psi\,\Hodge_{2}\,\rme^{-\psi} - \Hodge_{2}
		\end{align*}
		is well-defined on $\Test(T^*\mms) + \rmi\,\Test(T^*\mms)$ and therefore a densely defined linear operator on the complexified vector space $\Ell^2(T^*\mms) + \rmi\,\Ell^2(T^*\mms)$. \eqref{Eq:Formula in proof} yields
		\begin{align*}
		\int_\mms \langle\sfT_\psi\omega,\overline{\omega}\rangle\d\meas = -\int_\mms \vert\nabla\psi\vert^2\,\vert\omega\vert^2\d\meas + 2\rmi\int_\mms\Im\big(\nabla\omega^\sharp\big(\nabla\psi,\overline{\omega}^\sharp\big)\big)\d\meas
		\end{align*}
		for every $\omega\in \Test(T^*\mms) + \rmi\,\Test(T^*\mms)$. Young's inequality and \eqref{Eq:H embedding} thus give
		\begin{align*}
		\Big\vert\!\int_\mms \langle\sfT_\psi\omega,\overline{\omega}\rangle\d\meas\Big\vert &\leq \varepsilon^2\,\Vert\omega\Vert_{\Ell^2}^2 + 2\varepsilon\int_\mms \vert\nabla\omega^\sharp\vert\,\vert\omega\vert\d\meas\\
		&\leq (\varepsilon+1\big)\,\varepsilon\,\Vert\omega\Vert_{\Ell^2}^2 + \varepsilon\int_\mms\vert\nabla\omega^\sharp\vert^2\d\meas\\
		&\leq (\varepsilon+1-K)\,\varepsilon\,\Vert \omega\Vert_{\Ell^2}^2 + \varepsilon\int_\mms \langle\Hodge_{2}\omega,\omega\rangle\d\meas.
		\end{align*}
		
		Given any compact $V\subset \rho(\Hodge_{2})$, there exists $\varepsilon \in (0,1)$ such that for every $\xi\in V$, 
		\begin{align*}
		&2\,\Vert \big((\varepsilon+1-K)\,\varepsilon + \varepsilon\,\Hodge_{2}\big)(\Hodge_{2}-\xi)^{-1}\Vert_{\Ell^2,\Ell^2}\\
		&\qquad\qquad \leq 2\varepsilon\,\big(3 + \vert K\vert +\vert\xi\vert\big)\,\Vert (\Hodge_{2}-\xi)^{-1}\Vert_{\Ell^2,\Ell^2} +2\varepsilon < 1.
		\end{align*}
		From the above form boundedness of $\sfT_\psi$ by $\Hodge_{2}$, which is uniform in $\psi\in\Gamma_\varepsilon$, and under the previous choice of $\varepsilon$, \cite[Thm.~VI.3.9]{kato1995} both gives $\xi\in\rho(\rme^{\psi}\,\Hodge_{2}\,\rme^{-\psi})$ for every $\psi\in\Gamma_\varepsilon$ and provides us with the existence of a finite constant $C$ such that \eqref{Eq:L2L2 bound} holds uniformly in $\xi\in V$ and $\psi\in\Gamma_\varepsilon$.
	\end{proof}
	
Recall that every real $\alpha < 0$ belongs to $\rho(\Hodge_{2})$. Therefore, using \eqref{Eq:AHPT} in the first case, the operators $\smash{	(\Hodge_{2}-\alpha)^{-1/2}\,\rme^{-\psi}\,\phi_1}$ 	and $\smash{(\Hodge_{2}-\alpha)^{-1/2}\,\rme^{-\psi}}$ 	are well-defined on $\smash{\Ell^\infty_\bs(T^*\mms)}$ (which is the space of all $\omega\in\Ell^\infty(T^*\mms)$ such that $\vert\omega\vert$ has bounded support), hence densely defined on $\Ell^p(T^*\mms)$ for every $p\in [1,\infty)$.
	
	\begin{lemma}\label{Le:Resolvent bounds} There exists $\alpha < 0$ such that for every $\varepsilon \in (0,1)$, there exist an even $n\in\N$ and a constant $C<\infty$ such that for every $\psi\in\Gamma_\varepsilon$, we have
		\begin{align*}
		\Vert \rme^\psi\,(\Hodge_{2}-\alpha)^{-n/2}\,\rme^{-\psi}\,\phi_1\Vert_{\Ell^1,\Ell^2} &\leq C,\\
		\Vert \phi_1\,\rme^\psi\,(\Hodge_{2}-\alpha)^{-n/2}\,\rme^{-\psi}\Vert_{\Ell^2,\Ell^\infty} &\leq C.
		\end{align*} 
	\end{lemma}
	
	\begin{proof} Fix any real $\alpha < 0$ to be determined later, an arbitrary $\beta\in\R$ as well as an even $n\in\N$ with $n\geq \lfloor N+ 4\rfloor$. Employing the formula \eqref{Eq:Resolvent formula with heat flow} and then \autoref{Th:Kato-Simon}, \eqref{Eq:Heat kernel bound} as well as \eqref{Eq:Bishop-Gromov}, there exist constants $c,C_1,C_2,C_3 < \infty$ with
		\begin{align*}
		&\vert(\Hodge_{2}-\alpha)^{-n/2}\eta\vert \leq c\int_0^\infty \rme^{\alpha t}\, t^{n/2-1}\,\vert\HHeat_{2,t}\eta\vert\d t\\
		&\qquad\qquad\leq c \int_0^\infty \rme^{(\alpha-K)t}\,t^{n/2-1}\int_\mms \sfp_t(\cdot,y)\,\vert \eta\vert(y)\d\meas(y)\d t\\
		&\qquad\qquad\leq c\,C_1\int_0^\infty\Big[\rme^{(\alpha-K+C_2)t}\,t^{n/2-1}\,\meas\big[B_{\sqrt{t}}(\cdot)\big]^{-1}\\
		&\qquad\qquad\qquad\qquad\qquad\qquad \times \int_\mms \rme^{-\met^2(\cdot,y)/5t}\,\vert\eta\vert(y)\d\meas(y)\Big]\d t\\
		&\qquad\qquad\leq c\,C_1\,C_3\int_0^\infty\Big[\rme^{(\alpha-K+C_2)t}\,t^{n/2-1}\max\!\big\lbrace t^{-N/2},1\big\rbrace\,\phi_1^{-2}\\
		&\qquad\qquad\qquad\qquad\qquad\qquad\int_\mms \rme^{-\met^2(x,y)/5t}\,\vert\eta\vert(y)\d\meas(y)\Big]\d t\\
		&\qquad\qquad\leq c\,C_1\,C_3\,\Big[\!\int_0^\infty\rme^{(\alpha-K+C_2+5\beta^2/4)t}\,t^{n/2-1}\,\max\!\big\lbrace t^{-N/2},1\big\rbrace\d t\Big]\,\phi_1^{-2}\\
		&\qquad\qquad\qquad\qquad \times \int_\mms \rme^{-\beta\met(\cdot,y)}\,\vert\eta\vert(y)\d\meas(y)\quad\meas\text{-a.e.}
		\end{align*}
		for every $\eta\in\Ell^2(T^*\mms)$. In the last inequality, we also used that
		\begin{align*}
		\frac{5t}{4}\,\beta^2 - \beta\,\met(x,y) + \frac{\met^2(x,y)}{5t}  \geq 0.
		\end{align*}
		Setting $\smash{\alpha := \min\!\big\lbrace \!-\!1-K+C_2+5\beta^2/4, -1\big\rbrace}$ gives the existence of a constant $C_4 <\infty$ together with
		\begin{align}\label{Eq:Key estimate}
		\vert(\Hodge_{2}-\alpha)^{-n/2}\eta\vert \leq C_4\,\phi_1^{-2}\int_\mms \rme^{-\beta\met(\cdot,y)}\,\vert\eta\vert(y)\d\meas(y)\quad\meas\text{-a.e.}
		\end{align}
		
		The next step is to use \eqref{Eq:Key estimate} subject to a particular choice of $\beta\in\R$ to be determined later. Let $\varepsilon\in (0,1)$ and $\psi\in\Gamma_\varepsilon$ be arbitrary. Since $\rme^\psi\,(\Hodge_{2}-\alpha)^{-n/2}\,\rme^{-\psi}\,\phi_1$ is the formal adjoint of $\phi_1\,\rme^{\psi}\,(\Hodge_{2}-\alpha)^{-n/2}\,\rme^{-\psi}$, the first estimate  will actually follow from the second inequality. To prove the latter, for every $\smash{\omega\in\Ell_\bs^\infty(T^*\mms)}$, inserting $\eta := \rme^{-\psi}\,\omega$ into \eqref{Eq:Key estimate} for arbitrary $\beta > \varepsilon$ and using \eqref{Eq:Intrinsic distance} yields
		\begin{align}\label{Eq:Inequality with sum}
		&\vert \phi_1\,\rme^{\psi}\,(\Hodge_{2}-\alpha)^{-n/2}\,\rme^{-\psi}\,\omega\vert
		\leq C_4\,\phi_1^{-1}\,\rme^{\psi}\int_\mms\rme^{-\beta\met(\cdot,y)}\,\rme^{-\psi(y)}\,\vert\omega\vert(y)\d\meas(y)\nonumber\\
		&\qquad\qquad\leq C_4\,\Big[\phi_1^{-2}\int_\mms \rme^{-2(\beta-\varepsilon)\met(\cdot,y)}\d\meas(y)\Big]^{1/2}\,\Vert\omega\Vert_{\Ell^2}\nonumber\\
		&\qquad\qquad\leq C_4\,\Big[\!\sum_{j=1}^\infty \rme^{-2(\beta-\varepsilon)(j-1)}\,\meas[B_j(\cdot)]\,\meas[B_1(\cdot)]^{-1}\Big]^{1/2}\,\Vert\omega\Vert_{\Ell^2}\quad\meas\text{-a.e.}
		\end{align}
		By \eqref{Eq:Bishop-Gromov}, the last sum is uniformly bounded uniformly on $\mms$ and in $\varepsilon\in (0,1)$ as soon as $\beta > 0$ is chosen large enough.
		
		The inequality \eqref{Eq:Inequality with sum} for arbitrary $\omega\in\Ell^2(T^*\mms)$ follows by density of $\smash{\Ell^\infty_\bs(T^*\mms)}$, after passing to pointwise $\meas$-a.e.~convergent subsequences.
	\end{proof}
	
	\begin{corollary}\label{Cor:Linfty boundedness of resolvent} For every compact $V\subset \rho(\Hodge_{2})$, there exist $\varepsilon\in (0,1)$, an even $n\in\N$ and a constant $C<\infty$ such that for every $\xi\in V$, one has
		\begin{align*}
		\Vert (\Hodge_{2}-\xi)^{-n}\Vert_{\Ell^\infty,\Ell^\infty} \leq C\,\sup_{x\in\mms}\int_\mms \rme^{-\varepsilon\met(x,y)}\,\phi_1^{-1}(x)\,\phi_1^{-1}(y)\d\meas(y).
		\end{align*}
	\end{corollary}
	
	\begin{proof} Let $V\subset\rho(\Hodge_{2})$ be compact, and let $\varepsilon \in (0,1)$ be as provided by \autoref{Le:L2 boundedness} and $n\in\N$ be as in \autoref{Le:Resolvent bounds}. For every $\xi\in V$ and every $\psi\in\Gamma_\varepsilon$, the first resolvent identity gives
		\begin{align*}
		&\phi_1\,\rme^\psi\,(\Hodge_{2}-\xi)^{-n}\,\rme^{-\psi}\,\phi_1\\
		&\qquad\qquad = \sum_{j=0}^n \Big[\binom{n}{j}\,(\xi-\alpha)^j\,\big(\phi_1\,\rme^\psi\,(\Hodge_{2}-\alpha)^{-n/2}\,\rme^{-\psi}\big)\,\big(\rme^\psi\,(\Hodge_{2}-\xi)^{-1}\,\rme^{-\psi}\big)^j\\
		&\qquad\qquad\qquad\qquad \big(\rme^\psi\,(\Hodge_{2}-\alpha)^{-n/2}\,\rme^{-\psi}\,\phi_1\big)\Big].
		\end{align*}
		By \autoref{Le:L2 boundedness} and \autoref{Le:Resolvent bounds} we find a constant $C<\infty$ such that for every $\xi\in V$ and $\psi\in \Gamma_\varepsilon$, 
		\begin{align*}
		\Vert \phi_1\,\rme^\psi\,(\Hodge_{2}-\xi)^{-n}\,\rme^{-\psi}\,\phi_1\Vert_{\Ell^1,\Ell^\infty}\leq C.
		\end{align*}
		
		Hence, $\smash{\phi_1\,\rme^\psi\,(\Hodge_{2}-\xi)^{-n}\,\rme^{-\psi}\,\phi_1}$ is representable as an integral operator in the sense of \autoref{Th:Dunford Pettis forms} --- in particular, for every $\smash{\eta\in\Ell^\infty_\bs(T^*\mms)}$, we obtain
		\begin{align*}
		\vert  \phi_1\,\rme^\psi\,(\Hodge_{2}-\xi)^{-n}\,\rme^{-\psi}\,\phi_1\eta\vert\leq C\int_\mms\vert\eta\vert(y)\d\meas(y)\quad\meas\text{-a.e}.
		\end{align*}
		For clarity, we point out that here we used the mentioned \autoref{Th:Dunford Pettis forms} (which is proven only later), but no circular reasoning occurs since \autoref{Cor:Linfty boundedness of resolvent} will not be used in its proof. Setting $\eta := \phi_1^{-1}\,\rme^\psi\,\omega$, where $\smash{\omega\in\Ell^\infty_\bs(T^*\mms)}$ is arbitrary,
		\begin{align*}
		\vert (\Hodge_2-\xi)^{-n}\omega\vert \leq C\int_\mms \rme^{-\psi}\,\rme^{\psi(y)}\,\phi_1^{-1}\,\phi_1^{-1}(y)\,\vert\omega\vert(y)\d\meas(y)\quad\meas\text{-a.e.}
		\end{align*}
		By the arbitrariness of $\psi\in\Gamma_\varepsilon$ and \eqref{Eq:Intrinsic distance}, we obtain
		\begin{align*}
		\vert (\Hodge_2-\xi)^{-n}\omega\vert \leq C\int_\mms \rme^{-\varepsilon\met(\cdot,y)}\,\phi_1^{-1}\,\phi_1^{-1}(y)\,\vert\omega\vert(y)\d\meas(y)\quad\meas\text{-a.e.}
		\end{align*}
		
		The latter estimate is indeed true for every $\omega\in\Ell^\infty(T^*\mms)$ by an elementary cutoff argument, which establishes the desired assertion.
	\end{proof}
	
	\begin{proof}[Proof of \autoref{Th:Lp independence theorem}] Fix an arbitrary $p\in [1,2)\cup(2,\infty]$. We concentrate on the inclusion $\sigma(\Hodge_{p})\subset \sigma(\Hodge_{2})$. The inclusion $\sigma(\Hodge_{p})\supset \sigma(\Hodge_{2})$ follows as for \cite[Prop.~9]{charalambous2005}, and the argument for the isolated eigenvalues is the same as in (the references given in) the proof of \cite[Prop.~2.2]{hempel1986}.

		Let $V\subset\rho(\Hodge_{2})$ be compact with $V\cap (-\infty,0) \neq \emptyset$. Let $n\in\N$ be as in \autoref{Cor:Linfty boundedness of resolvent}. Since $\meas$ is uniformly subexponentially integrable, by \autoref{Cor:Linfty boundedness of resolvent} and taking adjoints, we see that $(\Hodge_{2}-\xi)^{-n}$ is bounded from $\Ell^p(T^*\mms)$ into $\Ell^p(T^*\mms)$ for $p\in\{1,\infty\}$. By Riesz--Thorin's interpolation theorem, $(\Hodge_{2}-\xi)^{-n}$ is actually bounded from $\Ell^p(T^*\mms)$ into $\Ell^p(T^*\mms)$ for every $p\in[1,\infty]$.
		
		By \eqref{Eq:Resolvent formula with heat flow} and since $\HHeat_{2,t} = \HHeat_{p,t}$ on $\Ell^2(T^*\mms)\cap\Ell^p(T^*\mms)$ for every $t\geq 0$, we get
		\begin{align}\label{Eq:p-q resolvent identity}
		(\Hodge_{2} - \xi)^{-n} = (\Hodge_{p} - \xi)^{-n}\quad\text{on }\Ell^2(T^*\mms)\cap\Ell^p(T^*\mms)
		\end{align}
		for every $\xi \in \rho(\Hodge_{2})\cap (-\infty,0)$. Since $V\cap (-\infty,0)\neq \emptyset$ and the map $\xi\mapsto (\Hodge_{2}-\xi)^{-n}$ is analytic on $\rho(\Hodge_{2})$, the identity \eqref{Eq:p-q resolvent identity} holds for every $\xi\in V$. In particular, $(\Hodge_{p}-\xi)^{-n}$ extends to a bounded linear operator from $\Ell^p(T^*\mms)$ into $\Ell^p(T^*\mms)$ for every $\xi\in V$, with
		\begin{align*}
		(\Hodge_{p}-\xi)^{-n} = (\Hodge_{2}-\xi)^{-n}\quad\text{on }\Ell^p(T^*\mms).
		\end{align*}
		It follows that $V\subset \rho(\Hodge_{p})$. Taking complements, we deduce the claimed inclusion.
	\end{proof}
	
	\begin{example} Let $N\in\N$ and $K<0$. The $N$-dimensional hyperbolic space $\smash{\mathbf{H}_K^N}$ with constant sectional curvature $K$ is an $\RCD^*(K(N-1),N)$ space when endowed with its Riemannian distance and Riemannian volume measure. In this situation, it is due to \cite[Thm.~14]{charalambous2005} that the set $\sigma(\Hodge_{p})$ \emph{does} depend on $p\in [1,\infty]$.
	\end{example}
	
	\section{Heat kernel}\label{Ch:Heat kernel}
	
	\subsection{Dunford--Pettis' theorem}\label{Sec:Dunford-Pettis}
	
	A crucial step in proving the existence of a heat kernel is a Dunford--Pettis-type theorem for co- or contravariant objects, see \autoref{Th:Dunford Pettis forms} below. See \cite[Lem.~11]{charalambous2005} for a smooth analogue obtained via computations in local coordinates. Our \autoref{Th:Dunford Pettis forms} significantly enlarges the scope of the latter to the language of $\Ell^1$-normed $\Ell^\infty$-modules from \autoref{Sec:A glimpse}.
	
	The following definition provides the setting for our understanding of integral operators over an $\Ell^p$-normed $\Ell^\infty$-module $\scrM$, $p\in [1,\infty]$. 
	
	\begin{definition}\label{Def:Kernel space} Given any $p,r\in [1,\infty]$, let $\scrM$ be a separable $\Ell^p$-normed $\Ell^\infty$-module, and $\scrN$ be a separable $\Ell^r$-normed $\Ell^\infty$-module. Let $\scrM^0$ and $\scrN^0$ be their corresponding $\Ell^0$-modules as introduced in \autoref{Sec:A glimpse}. We denote by $\scrN^0\boxtimes \scrM^0$ the space of all $\Ell^0$-bilinear maps $\sfa\colon \scrN^0\times \scrM^0\to \Ell^0(\mms^2)$.
		
		In the case $p=r$ and $\scrM = \scrN$, we briefly write $\smash{(\scrM^0)^{\boxtimes 2} := \scrM^0\boxtimes \scrM^0}$.
		
		For $\sfa\in \scrN^0\boxtimes\scrM^0$, we define the $\meas^{\otimes 2}$-measurable function $\vert\sfa\vert_\otimes\colon \mms^2\to [0,\infty]$ by
		\begin{align*}
		\vert \sfa\vert_\otimes(x,y) &:= \esssup\!\big\lbrace \vert\sfa[s,v]\vert(x,y) :\\
		&\qquad\qquad  s\in \scrN^0,\ v\in \scrM^0\text{ with } \vert s\vert,\vert v\vert\leq 1\,\meas\text{-a.e.}\big\rbrace.
		\end{align*}
	\end{definition}
	
	A key ingredient for \autoref{Th:Dunford Pettis forms} is the subsequent result from \cite[Thm.~2.2.5]{dunford1940}. Its advantage compared to the more general result \cite[Thm.~VI.8.6]{dunford1958} --- providing a similar statement with the Banach dual of \emph{any} separable Banach space as target domain --- is described in \autoref{Re:Later DP}. 
	
	\begin{proposition}\label{Th:DP40} Assume that $\sfB\colon \Ell^1(\mms)\to\Ell^\infty(\mms)$ is a linear and bounded map. Then there exists an $\meas^{\otimes 2}$-measurable kernel $\sfb\colon \mms^2\to\R$ such that
		\begin{align*}
		\Vert\sfb\Vert_{\Ell^\infty} = \Vert \sfB\Vert_{\Ell^1,\Ell^\infty} < \infty
		\end{align*}
		and, for every $g\in\Ell^1(\mms)$, 
		\begin{align*}
		\sfB g = \int_\mms \sfb(\cdot,y)\,g(y)\d\meas(y)\quad\meas\text{-a.e.}
		\end{align*}
				Such a kernel is unique in the sense that if $\tilde{\sfb}\colon \mms^2\to\R$ is another $\meas^{\otimes 2}$-measurable kernel fulfilling the foregoing obstructions, then $\smash{\tilde{\sfb}}$ does $\meas^{\otimes 2}$-a.e.~coincide with $\sfb$.
	\end{proposition} 
	
	\begin{theorem}[Dunford--Pettis theorem for $\Ell^\infty$-modules]\label{Th:Dunford Pettis forms} Let $\scrM$ and $\scrN$ be separable $\Ell^1$-normed $\Ell^\infty$-modules defined over $(\mms,\met,\meas)$. Suppose that $\sfA\colon \scrM\to\scrN^*$ is a linear map with $\Vert \sfA\Vert_{\scrM,\scrN^*}<\infty$. Then there exists $\sfa\in\scrN^0\boxtimes\scrM^0$ such that
		\begin{align*}
		\Vert\vert\sfa\vert_\otimes\Vert_{\Ell^\infty} = \Vert\sfA\Vert_{\scrM,\scrN^*}
		\end{align*}
		and, for every $v\in\scrM$ and every $s\in\scrN$, we have $\sfa[s,v]\in\Ell^1(\mms^2)$ with
		\begin{align*}
		\langle s\mid \sfA v\rangle = \int_\mms \sfa[s,v](\cdot,y)\d\meas(y)\quad\meas\text{-a.e.}
		\end{align*}
		The element $\sfa$ is unique in the sense that for any other $\smash{\tilde{\sfa}\in\scrN^0\boxtimes\scrM^0}$ satisfying the foregoing obstructions, $\smash{\sfa[s,v] = \tilde{\sfa}[s,v]}$ holds $\meas^{\otimes 2}$-a.e.~for every $v\in\scrM^0$ and every $s\in\scrN^0$.
	\end{theorem}
	
	\begin{proof} \textsc{Step 1.} \textit{Integral kernel for the pointwise pairing with $\sfA$.} Let $\scrD_\infty$, $\scrM_\infty$ and $\scrN_\infty$ be countable dense subsets in $\Ell^0(\mms)$, $\scrM$ and $\scrN$ made of $\meas$-essentially bounded elements. We may and will assume that $\One_\mms,0\in\scrD_\infty$. Let $\scrD_\diamond$ be the smallest algebra of functions in $\Ell^0(\mms)$ --- w.r.t.~pointwise multiplication --- which contains $\scrD_\infty$. Furthermore, let $\scrM_\diamond$ and $\scrN_\diamond$ be the sets of all finite linear combinations of elements of the form $f\,v$ and $f\,s$, respectively, where $f\in\scrD_\diamond$, $v\in\scrM_\infty$ and $s\in\scrN_\infty$. The classes $\scrD_\diamond$, $\scrM_\diamond$ and $\scrN_\diamond$ are all countable, consist of $\meas$-essentially bounded elements, and are dense in $\Ell^0(\mms)$, $\scrM^0$ and $\scrN^0$, respectively.
		
		Given any $v\in\scrM_\diamond$ and any $s\in\scrN_\diamond$, define $\sfB[s,v]\colon\Ell^1(\mms)\to\Ell^\infty(\mms)$ by
		\begin{align*}
		\sfB[s,v]g := \langle s\mid\sfA(g\,v)\rangle.
		\end{align*}
		The map $\sfB[s,v]$ is clearly linear, and it is well-defined and bounded since
		\begin{align*}
		\Vert \sfB[s,v]g\Vert_{\Ell^\infty} &\leq \Vert \sfA(g\,v)\Vert_{\scrN^*}\,\Vert \vert s\vert\Vert_{\Ell^\infty}\leq \Vert\sfA\Vert_{\scrM,\scrN^*}\,\Vert g\,v\Vert_\scrM\,\Vert \vert s\vert\Vert_{\Ell^\infty}\\
		&\leq \Vert\sfA\Vert_{\scrM,\scrN^*}\,\Vert g\Vert_{\Ell^1}\,\Vert\vert v\vert\Vert_{\Ell^\infty}\,\Vert\vert s\vert\Vert_{\Ell^\infty}.
		\end{align*}
		By \autoref{Th:DP40}, there exists a kernel $\sfb[s,v]\in\Ell^\infty(\mms^2)$, $\meas^{\otimes 2}$-a.e.~uniquely determined in a proper way, such that for every $g\in\Ell^1(\mms)$,
		\begin{align}\label{Eq:B v s identity}
		\sfB[s,v]g = \int_\mms \sfb[s,v](\cdot,y)\,g(y)\d\meas(y)\quad\meas\text{-a.e.}
		\end{align}
		
		\textsc{Step 2.} \textit{Properties of the obtained integral kernel.} An immediate property coming from the fact that $\sfB[s,c\,v]g = \sfB[s,v](c\,g)$ and $\sfB[d\,s, v]g = d\,\sfB[s,v]g$ for every $g\in\Ell^1(\mms)$ and every $c,d\in\Ell^\infty(\mms)$, and the $\meas^{\otimes 2}$-a.e.~uniqueness of the induced integral kernel is the following bilinearity. For every $c,d\in\scrD_\diamond$, every $v,v'\in\scrM_\diamond$ and every $s,s'\in\scrN_\diamond$, we have
		\begin{align}\label{Eq:Linearity b[v,s]}
		\begin{split}
		\sfb[d\,s + s',c\,v+v'] &= d(\pr_1)\,c(\pr_2)\,\sfb[s,v] + d(\pr_1)\,\sfb[s,v']\\
		&\qquad\qquad + c(\pr_2)\,\sfb[s',v] +\sfb[s',v'] \quad\meas^{\otimes 2}\text{-a.e.}
		\end{split}
		\end{align}
		
		Moreover, for every $v\in\scrN_\diamond$ and every $s\in\scrN_\diamond$, we claim that
		\begin{align}\label{Eq:Pointwise bound}
		\vert\sfb[s,v]\vert\leq \Vert\sfA\Vert_{\scrM,\scrN^*}\,\vert s\vert(\pr_1)\,\vert v\vert(\pr_2)\quad\meas^{\otimes 2}\text{-a.e.}
		\end{align}
		Indeed, let $g,h\in\Ell^1(\mms)$ be nonnegative. Multiplying both sides of the identity \eqref{Eq:B v s identity} with $h$ and integrating w.r.t.~$\meas$ yields
		\begin{align}\label{Eq:Estim}
		&\int_{\mms^2} \sfb[s,v](x,y)\,g(y)\,h(x)\d\meas^{\otimes 2}(x,y) = \int_\mms \sfB[s,v]g(x)\,h(x)\d\meas(x)\\
		&\qquad\qquad \leq \int_\mms \vert\sfA(g\,v)\vert(x)\,\vert s\vert(x)\,h(x)\d\meas(x)\nonumber\\
		&\qquad\qquad \leq \int_{\mms^2} \Vert \sfA\Vert_{\scrM,\scrN^*}\,\vert v\vert(y)\,\vert s\vert(x)\,g(y)\,h(x)\d\meas^{\otimes 2}(x,y).\nonumber
		\end{align}
		Changing the sign in both sides of \eqref{Eq:Estim}, the claim follows by the arbitrariness of $g$ and $h$.
		
		\textsc{Step 3.} \textit{Definition of $\sfa$.} Since the topology of $\scrM^0$ and $\scrN^0$ is intrinsic, in the sense indicated in \autoref{Sec:A glimpse}, we consider the distances $\met_{\scrM^0}$ and $\met_{\scrN^0}$ as defined w.r.t.~a fixed partition $(E_j)_{j\in\N}$ of $\mms$ into Borel subsets of finite and positive $\meas$-measure. Then $\smash{(E_j\times E_{j'})_{j,j'\in\N}}$ is a partition of $\mms^2$ into Borel sets of finite and positive $\meas^{\otimes 2}$-measure.  Let $(F_k)_{k\in\N_{\geq 2}}$ be an enumeration of the latter sequence with the property that if $F_k = E_j \times E_{j'}$ for some $k\in\N_{\geq 2}$ and some $j,j'\in\N$, then $j+j'\leq k$ (e.g.~using Cantor's diagonal procedure). We define the distance $\met_{\Ell^0}$ on $\Ell^0(\mms^2)$ w.r.t.~this sequence $(F_k)_{k\in\N_{\geq 2}}$.
		
		If $v\in\scrM_\diamond$ and $s\in\scrN_\diamond$, we set 
		\begin{align*}
		\sfa[s,v] := \sfb[s,v].
		\end{align*}
		
		Next, let any $v\in\scrM^0$ and $s\in\scrN^0$ satisfy $\vert v\vert,\vert s\vert\in\Ell^\infty(\mms)$. By density of $\scrM_\diamond$ in $\scrM$ and by the definition of $\scrM^0$, there exists a sequence $(v_n)_{n\in\N}$ in $\scrM_\diamond$ converging to $v$ w.r.t.~$\met_{\scrM^0}$. Analogously, we extract a sequence $(s_i)_{i\in\N}$ in $\scrN_\diamond$ converging to $s$ w.r.t.~$\met_{\scrN^0}$. Define
		\begin{align*}
		C := \max\!\big\lbrace \Vert \sfA\Vert_{\scrM,\scrN^*}+1, \Vert\vert v\vert\Vert_{\Ell^\infty}+1, \Vert\vert s\vert\Vert_{\Ell^\infty}+1\big\rbrace.
		\end{align*} 
		Given any $\varepsilon > 0$, select $L\in\N$ such that, for every $n,n',i,i' \geq L$,
		\begin{align*}
		\max\!\big\lbrace \met_{\scrM^0}(v_n,v_{n'}),\ \met_{\scrM^0}(v,v_{n'}),\ \met_{\scrN^0}(s_i,s_{i'}),\ \met_{\scrN^0}(s_i,s) \big\rbrace \leq \frac{\varepsilon}{6C^2}.
		\end{align*}
		Using the elementary fact that
		\begin{align*}
		\min\{a+b,1\}&\leq \min\{a,1\} +\min\{b,1\},\\
		\min\{ab,1\}&\leq \min\{a,1\} + \min\{b,1\}
		\end{align*}
		for every $a,b\in [0,\infty)$ as well as \eqref{Eq:Linearity b[v,s]} and \eqref{Eq:Pointwise bound} thus yields
		\begin{align*}
		&\sum_{k=2}^\infty \frac{2^{-k}}{\meas^{\otimes 2}[F_k]}\int_{F_k}\min\!\big\lbrace \vert \sfa[s_i,v_n] - \sfa[s_{i'}, v_{n'}]\vert,1\big\rbrace\d\meas^{\otimes 2}\\
		&\qquad\qquad\leq \sum_{j,j'=1}^\infty \frac{2^{-j-j'}}{\meas[E_j]\,\meas[E_{j'}]} \int_{E_j\times E_{j'}}\min\!\big\lbrace \vert \sfa[s_i,v_n] - \sfa[s_{i'}, v_{n'}]\vert,1\big\rbrace\d\meas^{\otimes 2}\\
		&\qquad\qquad\leq \sum_{j,j'=1}^\infty \frac{2^{-j-j'}}{\meas[E_j]\,\meas[E_{j'}]}\int_{E_j\times E_{j'}} \min\!\big\lbrace \vert\sfa[s_i,v_n - v_{n'}]\vert,1\big\rbrace\d\meas^{\otimes 2}\\
		&\qquad\qquad\qquad\qquad + \sum_{j,j'=1}^\infty \frac{2^{-j-j'}}{\meas[E_j]\,\meas[E_{j'}]}\int_{E_j\times E_{j'}} \min\!\big\lbrace \vert\sfa[s_i-s_{i'},  v_{n'}]\vert,1\big\rbrace\d\meas^{\otimes 2}\\
		&\qquad\qquad\leq C\sum_{j,j'=1}^\infty \frac{2^{-j-j'}}{\meas[E_j]\,\meas[E_{j'}]}\,\int_{E_j\times E_{j'}} \min\!\big\lbrace\vert s_i\vert(\pr_1)\,\vert v_n - v_n'\vert(\pr_2),1\big\rbrace\d\meas^{\otimes 2}\\
		&\qquad\qquad\qquad\qquad + C\sum_{j,j'=1}^\infty \frac{2^{-j-j'}}{\meas[E_j]\,\meas[E_{j'}]}\,\int_{E_j\times E_{j'}} \min\!\big\lbrace\vert s_i-s_{i'}\vert(\pr_1)\,\vert v_{n'}\vert(\pr_2),1\big\rbrace\d\meas^{\otimes 2}\\
		&\qquad\qquad\leq C\sum_{j,j'=1}^\infty \frac{2^{-j-j'}}{\meas[E_j]\,\meas[E_{j'}]}\,\int_{E_j\times E_{j'}} \min\!\big\lbrace\vert s_i-s\vert(\pr_1)\,\vert v_n - v_n'\vert(\pr_2),1\big\rbrace\d\meas^{\otimes 2}\\
		&\qquad\qquad\qquad\qquad + C^2\sum_{j,j'=1}^\infty \frac{2^{-j-j'}}{\meas[E_{j'}]}\,\int_{E_{j'}} \min\!\big\lbrace \vert v_n - v_n'\vert,1\big\rbrace\d\meas\\
		&\qquad\qquad\qquad\qquad + C\sum_{j,j'=1}^\infty \frac{2^{-j-j'}}{\meas[E_j]\,\meas[E_{j'}]}\,\int_{E_j\times E_{j'}} \min\!\big\lbrace\vert s_i-s_{i'}\vert(\pr_1)\,\vert v_{n'}-v\vert(\pr_2),1\big\rbrace\d\meas^{\otimes 2}\\
		&\qquad\qquad\qquad\qquad + C^2\sum_{j,j'=1}^\infty \frac{2^{-j-j'}}{\meas[E_j]}\,\int_{E_j} \min\!\big\lbrace\vert s_i-s_{i'}\vert,1\big\rbrace\d\meas\leq \varepsilon.
		\end{align*}
		Thus $(\sfb[s_i,v_n])_{n,i\in\N}$ is a Cauchy sequence in $\Ell^0(\mms^2)$ --- we define $\sfa[s,v]$ as its $\meas^{\otimes 2}$-a.e.~unique limit in $\Ell^0(\mms^2)$. A similar argument shows that this definition of $\sfa[s,v]$ is independent of the particularly chosen approximating sequences in $\scrM_\diamond$ and $\scrN_\diamond$, respectively. Moreover, the identities \eqref{Eq:Linearity b[v,s]}, for arbitrary $c,d\in\Ell^\infty(\mms)$, and \eqref{Eq:Pointwise bound} remain true for $\sfb$ replaced by $\sfa$.
		
		Lastly, for arbitrary $v\in\scrM^0$ and $s\in\scrN^0$, the sequences $(v_n)_{n\in\N}$ and $(s_i)_{i\in\N}$ given by $v_n := \One_{[0,n]}(\vert v\vert)\, v$ and $s_i := \One_{[0,i]}(\vert s\vert)\, s$ converge to $v$ and $s$ in $\scrM^0$ and $\scrN^0$, respectively. Indeed, observe that $\vert v - v_n\vert\to 0$ pointwise $\meas$-a.e.~as $n\to\infty$ and $\vert s-s_i\vert \to 0$ pointwise $\meas$-a.e.~as $i\to\infty$, and the claim follows since pointwise $\meas$-a.e.~convergent sequences converge in measure on finite measure spaces. By \eqref{Eq:Linearity b[v,s]} and \eqref{Eq:Pointwise bound},  $(\sfb[s_i,v_n])_{n,i\in\N}$ is a Cauchy sequence in $\Ell^0(\mms^2)$ --- we again define $\sfa[s,v]$ as its $\meas^{\otimes 2}$-a.e.~unique limit. Once again, it is easily seen that this definition does not depend on the chosen approximating sequences for $v$ and $s$, respectively, and that \eqref{Eq:Linearity b[v,s]}, for arbitrary $c,d\in\Ell^0(\mms)$, and \eqref{Eq:Pointwise bound} hold for $\sfa$ instead of $\sfb$.
		
		\textsc{Step 4.} \textit{Properties of $\sfa$.} From the previous step, we already know that $\sfa\in\scrN^0\boxtimes\scrM^0$.
		
		Moreover, from \eqref{Eq:Pointwise bound} for $\sfa$, it already follows that $\sfa[s,v]\in\Ell^1(\mms^2)$ for every $v\in\scrM$ and every $s\in\scrN$, and that
		\begin{align*}
		\Vert\vert\sfa\vert_\otimes\Vert_{\Ell^\infty}\leq \Vert\sfA\Vert_{\scrM,\scrN^*}.
		\end{align*}
		To show the claimed integral identity, let $z\in\mms$. For any sequences $(v_n)_{n\in\N}$ in $\scrM_\diamond$ and $(s_i)_{i\in\N}$ in $\scrN_\diamond$ converging to $v$ and $s$ in $\scrM$ and $\scrN$, respectively, from \eqref{Eq:B v s identity} we get, for every $n,i,k\in\N$,
		\begin{align*}
		\langle s\mid A(\One_{B_k(z)}\,v)\rangle = \int_{B_k(z)} \sfa[s_i,v_n](\cdot,y)\d\meas(y)\quad\meas\text{-a.e.}
		\end{align*}
		Letting $k\to\infty$ together with the continuity of $\sfA$ and then $n\to\infty$ and $i\to\infty$, employing that $\sfa[s_i,v_n]\to \sfa[s,v]$ in $\Ell^1(\mms^2)$ by virtue of \eqref{Eq:Pointwise bound}, the desired claim is deduced. 
		
		From this, the inequality
		\begin{align*}
		\Vert\vert\sfa\vert_\otimes\Vert_{\Ell^\infty}\geq \Vert \sfA\Vert_{\scrM,\scrN^*}
		\end{align*}
		simply follows by observing that, by definition of the pointwise norm in $\scrN^*$,
		\begin{align*}
		\vert\sfA v\vert &\leq \int_\mms \esssup\!\big\lbrace \sfa[s,v] : s\in\scrN^0,\ \vert s\vert\leq 1\ \meas\text{-a.e.}\big\rbrace\d\meas\\
		&\leq \Vert\vert\sfa\vert_\otimes\Vert_{\Ell^\infty}\,\Vert v\Vert_\scrM\quad\meas\text{-a.e.}
		\end{align*}
		
		The uniqueness statement is clear by $\Ell^0$-bilinearity of all considered mappings.
	\end{proof}
	
	\begin{remark}\label{Re:Later DP} In the setting of \autoref{Th:Dunford Pettis forms}, the general result \cite[Thm.~VI.8.6]{dunford1958} would provide a map $a$ on $\mms$, $\meas$-essentially uniquely determined in a proper way, such that $a(y)\colon\scrM\to \scrN^*$ is linear for $\meas$-a.e.~$y\in\mms$ and, for every $v\in\scrM$ and every $s\in\scrN$, we have
		\begin{align*}
		\int_\mms \langle s\mid \sfA v\rangle\d\meas = \int_\mms\int_\mms \langle s\mid a(y) v\rangle(x)\d\meas(x)\d\meas(y).
		\end{align*}
		
		However, it is not clear that the map $(x,y)\mapsto \langle s\mid a(y)v\rangle(x)$ is $\meas^{\otimes 2}$-measurable --- a property which is implicitly used at many places in the proof of \autoref{Th:Dunford Pettis forms}. Even in functional treatises, this is considered as a delicate detail \cite[Ch.~3]{grigoryan2014} and explains why we chose the formulation of \autoref{Def:Kernel space} with target space $\Ell^0(\mms^2)$.
	\end{remark}

	\subsection{Explicit construction as integral kernel}\label{Sub:Expl constr}
	
	We are now in a position to introduce our main result. On weighted Riemannian manifolds with not necessarily uniform lower Ricci bounds, a version of it has been proven in \cite[Thm.~XI.1]{gueneysu2017} using Lebesgue's differentiation theorem and thus local compactness of the underlying space, an assumption we do not make. See also \cite{grigoryan2009} for a thorough functional treatment.
	
	\begin{theorem}[Heat kernel existence]\label{Th:Heat kernel existence} Let $(\mms,\met,\meas)$ be an $\RCD(K,\infty)$ space, $K\in\R$. Then there exists a mapping $\smash{\Hheat\colon (0,\infty)\to \Ell^0(T^*\mms)^{\boxtimes 2}}$ such that for all $p,q\in [1,\infty]$ with $1/p+1/q=1$, if $\omega\in\Ell^p(T^*\mms)$ and $\eta\in\Ell^q(T^*\mms)$, for every $t>0$ we have $\Hheat_t[\eta,\omega]\in\Ell^1(\mms^2)$, and
		\begin{align*}
		\langle\eta,\HHeat_t\omega\rangle = \int_\mms \Hheat_t[\eta,\omega](\cdot,y)\d\meas(y)\quad\meas\text{-a.e.}
		\end{align*}
		The previously mentioned mapping $\sfh$ is uniquely determined in the sense that for every mapping $\smash{\tilde{\Hheat}\colon (0,\infty)\to\Ell^0(T^*\mms)^{\boxtimes 2}}$ satisfying the foregoing obstructions, for every $\omega,\eta\in\Ell^0(T^*\mms)$ and every $t>0$, the identity $\Hheat_t[\eta,\omega] = \smash{\tilde{\Hheat}_t[\eta,\omega]}$ holds $\meas^{\otimes 2}$-a.e.
	\end{theorem}
	
	\begin{proof} \textsc{Step 1.} \emph{Kernel for a perturbation of $\HHeat_t$.} Let $t>0$. We define the weight $\phi_t\colon\mms\to\R$, locally bounded by the volume growth property \eqref{Eq:Volume growth}, and the operator $\smash{\sfA_t\colon\Ell_\bs^1(T^*\mms)\to\Ell^0(T^*\mms)}$ as
		\begin{align*}
		\phi_t(x) &:= \meas\big[B_{\sqrt{t}}(x)\big]^{1/2},\\
		\sfA_t &:= \phi_t\,\HHeat_t\,\phi_t.
		\end{align*}
		
		By \autoref{Th:Kato-Simon} and the functional heat kernel bound from \autoref{Th:Heat kernel bound thm}, there exist constants $C_1,C_2<\infty$ such that for every $\smash{\omega\in\Ell_\bs^1(T^*\mms)}$, 
		\begin{align*}
		\vert \sfA_t\omega\vert &\leq  \rme^{-Kt}\int_\mms \phi_t\,\sfp_t(\cdot,y)\,\phi_t(y)\,\vert\omega\vert(y)\d\meas(y)\\
		&\leq \rme^{-Kt}\,\rme^{C_1(1+C_2t)}\,\Vert\omega\Vert_{\Ell^1}\quad\meas\text{-a.e.}\textcolor{white}{\int_M}
		\end{align*}
		Therefore, $\sfA_t$ uniquely extends to a bounded and linear operator from $\Ell^1(T^*\mms)$ into $\Ell^\infty(T^*\mms)$, whose extension we still denote by $\sfA_t$.
		
		\autoref{Th:Dunford Pettis forms} thus provides us with some element $\smash{\sfa_t\in\Ell^0(T^*\mms)^{\boxtimes 2}}$, uniquely determined in a proper way, such that for every $\omega,\eta\in\Ell^1(T^*\mms)$, 
		\begin{align*}
		\langle\eta,\sfA_t\omega\rangle = \int_\mms \sfa_t[\eta,\omega](\cdot,y)\d\meas(y)\quad\meas\text{-a.e.}
		\end{align*}
		In fact, arguing as in the proof of \autoref{Th:Dunford Pettis forms}, we obtain
		\begin{align}\label{Eq:Heat kernel est}
		\vert\sfa_t\vert_\otimes \leq \rme^{-Kt}\,\phi_t(\pr_1)\,\phi_t(\pr_2)\,\sfp_t\quad\meas^{\otimes 2}\text{-a.e.}
		\end{align}
		
		\textsc{Step 2.} \emph{Removing the weights.} Given the element $\sfa_t$ extracted in the previous step and any $\varepsilon,\iota > 0$, we define $\smash{\Hheat_t^{\varepsilon,\iota}\in \Ell^0(T^*\mms)^{\boxtimes 2}}$ through
		\begin{align*}
		\Hheat_t^{\varepsilon,\iota}[\eta,\omega] := \sfa_t\Big[\frac{1}{\phi_t+\varepsilon}\,\eta,\frac{1}{\phi_t+\iota}\,\omega\Big].
		\end{align*}
		It is clear from \eqref{Eq:Heat kernel est} and \autoref{Th:Heat kernel bound thm} that $\smash{\Hheat_t^{\varepsilon,\iota}[\eta,\omega]}\in\Ell^1(\mms^2)$ for every $\omega,\eta\in\Ell^1(T^*\mms)$. Moreover, if in addition $\smash{\omega\in\Ell_\bs^1(T^*\mms)}$, then 
		\begin{align}\label{Eq:Epsilon equality}
		\Big\langle \eta, \frac{\phi_t}{\phi_t+\varepsilon}\,\HHeat_t\Big[\frac{\phi_t}{\phi_t+\iota}\,\omega\Big]\Big\rangle = \int_\mms \Hheat_t^{\varepsilon,\iota}[\eta,\omega](\cdot,y)\d\meas(y)\quad\meas\text{-a.e.}
		\end{align}
		
		Next, observe that for every $\omega,\eta\in\Ell^0(T^*\mms)$ and every $\varepsilon',\iota'>0$, by \eqref{Eq:Heat kernel est},
		\begin{align*}
		&\vert\Hheat_t^{\varepsilon,\iota}[\eta,\omega] - \Hheat_t^{\varepsilon',\iota'}[\eta,\omega]\vert\\
		&\qquad\qquad \leq \Big\vert\sfa_t\Big[\frac{1}{\phi_t+\varepsilon}\,\eta, \frac{1}{\phi_t+\iota}\,\omega - \frac{1}{\phi_t+\iota'}\,\omega\Big]\Big\vert\\
		&\qquad\qquad\qquad\qquad + \Big\vert \sfa_t\Big[\frac{1}{\phi_t+\varepsilon}\,\eta - \frac{1}{\phi_t+\varepsilon'}\,\eta, \frac{1}{\phi_t+\iota'
		}\,\omega\Big]\Big\vert\\
		&\qquad\qquad\leq \rme^{-Kt}\,\vert\eta\vert(\pr_1)\,\vert\omega\vert(\pr_2)\,\sfp_t\times\,\Big\vert \frac{\phi_t(\pr_2)}{\phi_t(\pr_2)+\iota} - \frac{\phi_t(\pr_2)}{\phi_t(\pr_2)+\iota'}\Big\vert\\
		&\qquad\qquad\qquad\qquad + \bigg[ \rme^{-Kt}\,\vert\eta\vert(\pr_1)\,\vert\omega\vert(\pr_2)\,\sfp_t\\
		&\qquad\qquad\qquad\qquad\qquad\qquad \times\,\Big\vert \frac{\phi_t(\pr_1)}{\phi_t(\pr_1)+\varepsilon} - \frac{\phi_t(\pr_1)}{\phi_t(\pr_1)+\varepsilon'}\Big\vert \bigg]\quad\meas^{\otimes 2}\text{-a.e}.
		\end{align*}
		Thus, independently of the choice of sequences $(\varepsilon_n)_{n\in\N}$ and $(\iota_n)_{n\in\N}$ in $(0,\infty)$ converging to $0$ in place of $\varepsilon$ and $\iota$, the two-parameter family $\smash{(\sfh_t^{\varepsilon,\iota}[\eta,\omega])_{\varepsilon,\iota>0}}$ has a unique limit in $\Ell^0(\mms^2)$ --- we define $\Hheat_t[\eta,\omega]$ to be this limit and denote by $\smash{\Hheat_t \in \Ell^0(T^*\mms)^{\boxtimes 2}}$ the induced element.
		
		\textsc{Step 3.} \textit{Properties of $\Hheat$.} Turning to the claimed integral representation of $\HHeat_t$, given any $\omega\in\Ell^p(T^*\mms)$ and $\eta\in\Ell^q(T^*\mms)$ where $p,q\in [1,\infty]$ are dual to each other, we integrate \eqref{Eq:Epsilon equality} and let $\varepsilon,\iota\to 0$. On the one hand, by Hölder's inequality, \autoref{Cor:Extension} and Lebesgue's theorem,
		\begin{align*}
		\lim_{\varepsilon,\iota\to 0} \int_\mms \Big\langle\eta, \frac{\phi_t}{\phi_t+\varepsilon}\,\HHeat_t\Big[\frac{\phi_t}{\phi_t+\iota}\,\omega\Big]\Big\rangle\d\meas = \int_\mms \langle\eta,\HHeat_t\omega\rangle\d\meas.
		\end{align*}
		On the other hand, from the construction in the previous step,
		\begin{align*}
		\vert \sfh_t\vert_\otimes \leq \rme^{-Kt}\,\sfp_t\quad\meas^{\otimes 2}\text{-a.e.}
		\end{align*}
		so that a further application of Lebesgue's theorem to \eqref{Eq:Epsilon equality} entails
		\begin{align*}
		\int_\mms \langle\eta,\HHeat_t\omega\rangle\d\meas = \int_{\mms^2}\Hheat_t[\eta,\omega]\d\meas^{\otimes 2}.
		\end{align*}
		
		Replacing $\eta$ by $f\,\eta$ for arbitrary $f\in\Ell^\infty(\mms)$ finally gives the claimed pointwise $\meas$-a.e.~equality.
		
		The uniqueness statement is as clear as in the proof of \autoref{Th:Dunford Pettis forms}.
	\end{proof}
	
	\begin{definition}\label{Def:Heat kernel forms} We call the mapping $\sfh$ from \autoref{Th:Heat kernel existence} the \emph{$1$-form heat kernel} of $(\mms,\met,\meas)$.
	\end{definition}
	
	It is straightforward to check the following result using the symmetry and the semigroup property of $(\HHeat_t)_{t\geq 0}$, the Chapman--Kolmogorov formula for the functional heat kernel \cite[Thm.~6.1]{ambrosio2014b} as well as \autoref{Th:Kato-Simon}. \autoref{Cor:Heat kernel bds} then follows from \autoref{Th:Heat kernel bound thm} and \eqref{Eq:Heat kernel bound}.
	
	\begin{theorem}\label{Th:Properties heat kernel} For every $\omega,\eta\in\Ell^0(T^*\mms)$ and every $s,t>0$, the $1$-form heat kernel $\sfh$ from \autoref{Def:Heat kernel forms} obeys the following relations at $\meas^{\otimes 2}$-a.e.~$(x,y)\in\mms^2$.
		\begin{enumerate}[label=\textnormal{(\roman*)}]
			\item\label{La:Symm} \textnormal{\textsc{Symmetry.}} We have
			\begin{align*}
\Hheat_t[\eta,\omega](x,y) = \Hheat_t[\omega,\eta](y,x)
			\end{align*}
			\item\label{La:KSI} \textnormal{\textsc{Pointwise Hess--Schrader--Uhlenbrock inequality.}} We have 
			\begin{align*}
				\vert\Hheat_t\vert_\otimes(x,y) \leq \rme^{-Kt}\,\sfp_t(x,y)
			\end{align*}
			\item\label{La:CKE} \textnormal{\textsc{Chapman--Kolmogorov equation.}} We have
			\begin{align*}
			\int_\mms \Hheat_{t+s}[\eta,\omega](\cdot,y)\d\meas(y) = \int_\mms \Hheat_s[\eta,\HHeat_t\omega](\cdot,y)\d\meas(y)\quad\meas\text{-a.e.}
			\end{align*}
		\end{enumerate}
	\end{theorem}
	
	\begin{remark}\label{Re:Local basis} Since $(\HHeat_t)_{t\geq 0}$ does not localize, we cannot state Chapman--Kolmogorov's formula from \ref{La:CKE} in \autoref{Th:Properties heat kernel} as a pointwise $\meas^{\otimes 2}$-a.e.~equality in the previous sense.
		
		However, if $(\mms,\met,\meas)$ is an $\RCD^*(K,N)$ space, $N\in (1,\infty)$, this can be circumvented by using the \emph{dimensional decomposition} of $\Ell^2(T^*\mms)$ \cite[Prop.~1.4.5]{gigli2018}. In this case (see \autoref{Sub:Fund sol} for details), there exists a unique $n\in\N$ as well as a Borel set  $E_n\subset\mms$ with $\meas[\mms\setminus E_n] = 0$ such that for every Borel set $B\subseteq E_n$ with $\meas[B] < \infty$, there exist local basis vectors $\rho_1,\dots,\rho_n\in\Ell^2(T^*\mms)$ \cite[Sec.~1.4]{gigli2018} such that for every $i,j\in\{1,\dots,n\}$, we have
		\begin{align*}
		\One_{B^\rmc}\,\rho_i &= 0,\\
		\langle\rho_i,\rho_j\rangle &= \delta_{ij}\quad\meas\text{-a.e.}\quad\text{on }B.
		\end{align*}
		
		In this framework, given any $\omega,\eta\in\Ell^0(T^*\mms)$, we claim that for every $s,t>0$,
		\begin{align*}
		\Hheat_{t+s}[\eta,\omega] = \sum_{i=1}^n\int_B \Hheat_s[\eta,\rho_i](\pr_1,z)\,\Hheat_t[\rho_i,\omega](z,\pr_2)\d\meas(z)\quad\meas^{\otimes 2}\text{-a.e.}\quad\text{on }B^2.
		\end{align*}
		Indeed, let $z\in B$ and $R>0$ be arbitrary, and set 
		\begin{align*}
		\omega_R &:= \One_{B_R(z)}\,\One_{[0,R]}(\vert\omega\vert)\,\omega,\\
		\eta_R &:= \One_{B_R(z)}\,\One_{[0,R]}(\vert\eta\vert)\,\eta.
		\end{align*}
		By \autoref{Th:Heat kernel existence} and \autoref{Th:Properties heat kernel} applied to $\One_B\,\omega_R$ and $\One_B\,\eta_R$ in place of $\omega$ and $\eta$, we obtain
		\begin{align*}
		\int_{B^2}\Hheat_{t+s}[\eta_R,\omega_R]\d\meas^{\otimes 2} &= \int_{B^2}\Hheat_s[\eta_R, \One_B\,\HHeat_t(\One_B\,\omega_R)]\d\meas^{\otimes 2}\\
		&= \sum_{i=1}^n\int_{B^2} \Hheat_s[\eta_R,\rho_i]\,\langle \rho_i, \HHeat_t(\One_B\,\omega_R)\rangle(\pr_2)\,\d\meas^{\otimes 2}\\
		&= \sum_{i=1}^n \int_{B^2} \int_B \Big[\Hheat_s[\eta_R,\rho_i](\pr_1,z)\\
		&\qquad\qquad \times \Hheat_t[\rho_i,\omega_R](z,\pr_2)\Big]\d\meas(z)\d\meas^{\otimes 2}.\textcolor{white}{\sum_{i}^n}
		\end{align*}
		
		The integrands on both sides of this chain of equalities are local in their respective components, and the claim follows by letting $R\to\infty$.
	\end{remark}
	
	\begin{theorem}\label{Cor:Heat kernel bds} Let $(\mms,\met,\meas)$ be an $\RCD(K,\infty)$ space, and let $\varepsilon > 0$. Then there exist finite constants $C_1 > 0$, depending only on $\varepsilon$, and $C_2 \geq 0$, depending only on $K$, with $C_2 := 0$ if $K\geq 0$, such that for $\meas^{\otimes 2}$-a.e.~$(x,y)\in\mms^2$, we have
		\begin{align*}
		\vert\Hheat_t\vert_\otimes(x,y) &\leq \meas\big[B_{\sqrt{t}}(x)\big]^{-1/2}\,\meas\big[B_{\sqrt{t}}(y)\big]^{-1/2}\\
		&\qquad\qquad \exp\!\Big(C_1\big(1+ (C_2-K)t\big) - \frac{\met^2(x,y)}{(4+\varepsilon)t}\Big).
		\end{align*}
		In particular, if $(\mms,\met,\meas)$ obeys the $\RCD^*(K,N)$ condition with $N\in (1,\infty)$, there exist  finite constants $C_3,C_4>1$ depending only on $\varepsilon$, $K$ and $N$ such that at $\meas^{\otimes 2}$-a.e.~$(x,y)\in\mms^2$, we have
		\begin{align*}
		\vert\Hheat_t\vert_\otimes(x,y) \leq C_3\,\meas\big[B_{\sqrt{t}}(y)\big]^{-1}\,\exp\!\Big((C_4-K)t - \frac{\met^2(x,y)}{(4+\varepsilon)t}\Big).
		\end{align*}
	\end{theorem}
	
	\subsection{Trace inequality and spectral resolution}\label{Sub:Fund sol}
	
	Let $(\mms,\met,\meas)$ be an $\RCD^*(K,N)$ space for some $N\in (1,\infty)$. First, relying on \autoref{Th:Properties heat kernel} and \autoref{Re:Local basis}, we prove a trace inequality between $\HHeat_t$ and $\ChHeat_t$ in \autoref{Pr:Trace}. In the smooth case, the corresponding bound is classical, see e.g.~\cite[Cor.~XI.8]{gueneysu2017}, (1.1) in \cite{hess1980} or \cite[Thm.~3.5]{rosenberg1988}. A key feature of the $\RCD^*(K,N)$ framework is that, by \cite{giglipasqualetto2016,han2018}, there exists precisely one $n\in\N$ such that $\meas[E_n] > 0$ within the dimensional decomposition $(E_n)_{n\in\N}$ of $\Ell^2(T^*\mms)$ --- actually, $n$ is equal to the \emph{essential dimension} $\dim_{\met,\meas}\mms \in \{1,\dots, \lfloor N\rfloor\}$ of $(\mms,\met,\meas)$. See \cite{brue2020, giglipasqualetto2016, mondino2019} for comprehensive accounts on the latter from the structure theoretic point of view.
	
	Via \autoref{Pr:Eigenbasis}, if $\mms$ is additionally compact, we also prove a spectral resolution identity for $\Hheat_t$, $t>0$, in \autoref{Cor:Spectral resol}. More precisely, we show that $\Hheat_t$ can be viewed as an element in the two-fold Hilbert space tensor product $\Ell^2(T^*\mms)^{\otimes 2}$ of $\Ell^2(T^*\mms)$ indicated in \autoref{Sec:A glimpse}. 
	
	\begin{theorem}\label{Pr:Trace} Let $(\mms,\met,\meas)$ be an $\RCD^*(K,N)$ space with $K\in\R$ and $N\in (1,\infty)$. Intending the traces in the usual Hilbert space sense, for every $t>0$,
		\begin{align*}
		\tr \HHeat_t \leq (\dim_{\met,\meas}\mms)\,\rme^{-Kt}\,\tr\ChHeat_t.
		\end{align*}
	\end{theorem}
	
	\begin{proof} Abbreviate $d:= \dim_{\met,\meas}\mms$ and let $B\subset E_d$ be any bounded Borel set with $\meas[B]\in (0, \infty)$. Let $\rho_1,\dots,\rho_d \in \Ell^2(T^*\mms)$ be local basis vectors of $\Ell^2(T^*\mms)$ on $B$ as in \autoref{Re:Local basis}. We set $\omega_k := \meas[B]^{-1/2}\,\rho_k$ with $k\in\{1,\dots,n\}$ and complete this set of $1$-forms to a countable orthonormal basis $(\omega_k)_{k\in\N}$ of $\Ell^2(T^*\mms)$ --- we may and will assume that $\One_B\,\langle\omega_k,\rho_i\rangle = 0$ for all $k>d$ and $i\in\{1,\dots,d\}$. 
		¸
		Given any $\alpha\in\N$ with $\alpha> d$, by \autoref{Th:Properties heat kernel} and Cauchy--Schwarz's inequality, 
		\begin{align}\label{Eq:d est for trace}
		&\sum_{k=1}^\alpha \int_\mms \langle\HHeat_t\omega_k,\omega_k\rangle\d\meas\nonumber\\
		&\qquad\qquad = \sum_{k=1}^\alpha \sum_{i,i'=1}^d \int_{B^2} \langle\omega_k,\rho_i\rangle(\pr_1)\,\langle\omega_k,\rho_{i'}\rangle(\pr_2)\,\Hheat_t[\rho_i,\rho_{i'}]\d\meas^{\otimes 2}\nonumber\\
		&\qquad\qquad= \meas[B]^{-1}\sum_{i=1}^d \int_{B^2} \Hheat_t[\rho_i,\rho_i]\d\meas^{\otimes 2}\nonumber\\
		&\qquad\qquad = \meas[B]^{-1}\sum_{i=1}^d \int_{B^2} \Hheat_{t/2}[\rho_i,\HHeat_{t/2}\rho_i] \d\meas^{\otimes 2}\nonumber\\
		&\qquad\qquad \leq \meas[B]^{-1}\sum_{i=1}^d\int_{B^2} \vert\Hheat_{t/2}\vert_\otimes\,\vert\HHeat_{t/2}\rho_i\vert(\pr_2)\d\meas^{\otimes 2}\nonumber\\
		&\qquad\qquad\leq \meas[B]^{-1}\, d\int_{B^3}\vert\Hheat_{t/2}\vert_\otimes(\pr_1,\pr_3)\,\vert\Hheat_{t/2}\vert_\otimes(\pr_3,\pr_2)\d\meas^{\otimes 3}\textcolor{white}{\sum_{j=1}^d}\\
		&\qquad\qquad \leq d\int_{B^2} \vert\Hheat_{t/2}\vert_\otimes^2\d\meas^{\otimes 2} \leq d\,\rme^{-Kt}\int_{\mms^2}\sfp_{t/2}^2\d\meas^{\otimes 2} = d\,\rme^{-Kt}\,\tr\ChHeat_t.\textcolor{white}{\sum_{j=1}^d}\nonumber
		\end{align}
		In \eqref{Eq:d est for trace}, we used \autoref{Th:Heat kernel existence} together with duality for the pointwise norm, see \autoref{Sec:A glimpse}. The last identity follows from the self-adjointness of the functional heat flow in $\Ell^2(\mms)$. 
		
		The asserted inequality then follows by letting $\alpha\to\infty$.
	\end{proof}
	
	Now, let $\mms$ be compact. Let $(\omega_i)_{i\in \N}$ be an orthonormal basis of $\Ell^2(T^*\mms)$ consisting of eigenforms for $\Hodge$ provided by \autoref{Pr:Eigenbasis}, i.e.~$\omega_i\in \rmE_{\lambda_i}(\Hodge)$ for some $\lambda_i\in\sigma(\Hodge)$ and every $i\in\N$.
	
	\begin{theorem}\label{Cor:Spectral resol} Let $(\mms,\met,\meas)$ be a compact $\RCD^*(K,N)$ space, with $K\in\R$ and $N\in (1,\infty)$. For every $t>0$, there exists a unique element $\sfg_t\in\Ell^2(T^*\mms)^{\otimes 2}$ such that, for every $\omega,\eta\in\Ell^2(T^*\mms)$,
		\begin{align*}
		\langle\sfg_t, \eta\otimes\omega\rangle = \sfh_t[\eta,\omega]\quad\meas^{\otimes 2}\text{-a.e}.
		\end{align*}
		Moreover, w.r.t.~strong convergence in $\Ell^2(T^*\mms)^{\otimes 2}$, $\sfg_t$ admits the  representation
		\begin{align*}
		\sfg_t = \sum_{i=1}^\infty \rme^{-\lambda_i t}\,\omega_i\otimes\omega_i.
		\end{align*}
	\end{theorem}
	
	\begin{proof} By virtue of \eqref{Eq:AHPT} and \autoref{Cor:Heat kernel bds}, we have $\vert\Hheat_t\vert_\otimes\in\Ell^\infty(\mms^2)$. Hence, the bilinear map $\sfG_t\colon\Ell^2(T^*\mms)^2\to\R$ is well-defined, where\begin{align*}
		\sfG_t(\eta,\omega) := \int_{\mms^2} \Hheat_t[\eta,\omega]\d\meas^{\otimes 2}.
		\end{align*}
		Moreover, $\sfG_t$ is weakly Hilbert--Schmidt by \autoref{Th:Heat kernel existence} and \autoref{Cor:Hilbert-Schmidt}. Therefore, by \autoref{Th:Univ property} there exists a unique bounded $\sfT_t\colon\Ell^2(T^*\mms)^{\otimes 2}\to \R$ such that $\sfG_t(\eta,\omega) = \sfT_t(\eta\otimes\omega)$ for every $\omega,\eta\in\Ell^2(T^*\mms)$. Recalling \autoref{Pr:HS tensor product}, by Riesz' theorem for Hilbert modules \cite[Thm.~1.2.24]{gigli2018} there exists a unique $\sfg_t\in\Ell^2(T^*\mms)^{\otimes 2}$ such that for every $\omega,\eta\in\Ell^2(T^*\mms)$,
		\begin{align*}
		\int_{\mms^2} \langle\sfg_t,\eta\otimes\omega\rangle\d\meas^{\otimes 2} = \sfT_t(\eta\otimes\omega) = \sfG_t(\eta,\omega)= \int_{\mms^2} \Hheat_t[\eta,\omega]\d\meas^{\otimes 2}.
		\end{align*}
		Replacing $\omega$ and $\eta$ by $f\,\omega$ and $g\,\eta$ for arbitrary $f,g\in\Ell^\infty(\mms)$, respectively, provides the claimed $\meas^{\otimes 2}$-a.e.~valid identity $\langle\sfg_t,\eta\otimes\omega\rangle = \Hheat_t[\eta,\omega]$.
		
		It remains to prove the series representation of $\sfg_t$. Since $(\omega_i\otimes\omega_j)_{i,j\in\N}$ is an orthonormal basis of $\Ell^2(T^*\mms)^{\otimes 2}$,  this simply follows by writing
		\begin{align*}
		\sfg_t = \sum_{i,j=1}^\infty c_{ij}(t)\,\omega_i\otimes\omega_{j}
		\end{align*}
		w.r.t.~strong convergence in $\Ell^2(T^*\mms)^{\otimes 2}$, 
		where the coefficients are given by
		\begin{align*}
		c_{ij}(t) = \int_{\mms^2}\langle\sfg_t,\omega_i\otimes\omega_j\rangle\d\meas^{\otimes 2} = \int_\mms\langle\omega_i,\HHeat_t\omega_j\rangle\d\meas =  \rme^{-\lambda_i t}\,\delta_{ij}.\tag*{\qedhere}
		\end{align*}
	\end{proof}
	
	\begin{remark}\label{Re:Compat} By \autoref{Cor:Spectral resol} and \cite[Prop.~3.1]{rosenberg1997}, our notion of the $1$-form heat kernel from \autoref{Def:Heat kernel forms} is fully compatible with the so-called \emph{parametrix approach} to it on compact, non-weighted Riemannian manifolds  \cite[Ch.~4]{patodi1971}. More precisely, denoting the smooth heat kernel by $\Hheat\colon (0,\infty)\times\mms^2\to (T^*\mms)^*\boxtimes T^*\mms$ as in \autoref{Ch:Introduction} by a slight abuse of notation, for every smooth $1$-forms $\omega$ and $\eta$, we have $\sfg_t(\eta\otimes\omega)(x,y) = \langle\eta(x), \Hheat_t(x,y)\,\omega(y)\rangle$ for $\meas^{\otimes 2}$-a.e.~$(x,y)\in\mms^2$.
	\end{remark}

\end{document}